\documentclass{article}

\usepackage[margin = 1.3in]{geometry}

\usepackage{tikz}
\usepackage{tikz-cd}
\usetikzlibrary{cd, arrows.meta, calc}
\usepackage{url, hyperref}
\usepackage{float}

\tikzset{->-/.style={decoration={
			markings,
			mark=at position .45 with {\arrow{>}}},postaction={decorate}}}
\usepackage{amsmath}
\usepackage{amssymb}
\usepackage{enumitem}
\usepackage{graphicx}
\usepackage{mathdots}
\usepackage{color}
\usepackage{diagbox}
\usepackage{array, makecell}
\usepackage{rotating}
\usepackage{amsthm}
\usepackage{dsfont}
\usepackage{adjustbox}
\usetikzlibrary{arrows}
\usepackage{mathtools}

\theoremstyle{definition}
\newtheorem{definition}{Definition}[section]

\newtheorem{remark}[definition]{Remark}
\newtheorem{notation}[definition]{Notation}

\theoremstyle{theorem}
\newtheorem{theorem}[definition]{Theorem}

\newtheorem{proposition}[definition]{Proposition}
\newtheorem{corollary}[definition]{Corollary}
\newtheorem{lemma}[definition]{Lemma}

\def\Z{\mathbb{Z}}

\def\R{\mathbb{R}}

\def\P{\mathbb{P}}
\def\M{\mathcal{M}}
\def\R{\mathcal{R}}
\def\CH{\mathrm{CH}}
\def\Sym{\mathrm{Sym}}
\def\RH{\mathcal{RH}}
\def\Sym{\mathrm{Sym}}
\def\det{\mathrm{det}}
\def\Gm{\mathbb{G}_m}
\def\GL{\mathrm{GL}}
\def\GI{(\Gm \times \Gm) \rtimes \mu_2}

\def\PGL{\mathrm{PGL}}
\def\SL{\mathrm{SL}}
\def\cO{\mathcal{O}}
\def\uD{\underline{\Delta}}

\def\cH{\mathcal{H}}
\def\cD{\mathcal{D}}
\def\cX{\mathcal{X}}
\def\cY{\mathcal{Y}}
\def\cZ{\mathcal{Z}}
\def\cL{\mathcal{L}}

\def\cT{\mathcal{T}}
\def\cM{\mathcal{M}}

\def\cS{\mathcal{S}}

\title{The Integral Chow Rings of the Moduli Stacks of Hyperelliptic Prym Pairs II}
\author{Alessio Cela and Alberto Landi}
\date{\vspace{-5ex}}

\begin{document}
	
	\maketitle
	
	\begin{abstract}
		
		This paper is the second in a series devoted to describing the integral Chow ring of the moduli stacks $\mathcal{RH}_g$ of hyperelliptic Prym pairs. For fixed genus $g$, the stack $\mathcal{RH}_g$ is the disjoint union of $\lfloor (g+1)/2 \rfloor$ components $\mathcal{RH}_g^n$ for $n = 1, \ldots, \lfloor (g+1)/2 \rfloor$. In this paper, we give presentations and compute the integral Chow rings of the components $\mathcal{RH}_g^{(g+1)/2}$ for odd $g$. As an application, we also obtain presentations and Chow rings for all irreducible components of the moduli stack of hyperelliptic Spin curves of odd genus. An intermediate result of independent interest is the computation of the integral Chow ring of the moduli stack of unordered pairs of divisors of the same even degree in $\mathbb{P}^1$.

	\end{abstract}
	
	\tableofcontents
	
	\section{Introduction}
	
	This paper concerns the study of the geometry of the stack parametrizing hyperelliptic Prym pairs and the computation of its integral Chow ring.
	
	The computation of Chow rings of moduli stacks has attracted significant interest in algebraic geometry, starting from the seminal work of Mumford~\cite{Mumford1983}, where he studies the rational Chow ring of the stack $\cM_g$ parametrizing smooth genus $g$ curves, as well as its Deligne-Mumford compactification $\overline{\cM}_g$. Since then, numerous computations have been carried out. As for $\cM_{g,n}$ the rational Chow ring of $\cM_g$ is known for $g\leq9$ (\cite{Mumford1983,Fab90I,Faber,Izadi,PenevVakil,CL23}). On the other hand, the integral case is considerably more difficult, as reflected by the comparatively smaller number of known results.  The integral Chow ring of the stack of $n$-pointed smooth genus $g$ curves $\cM_{g,n}$ is currently known only for the following few values of $(g,n)$: $(0,n)$ all $n$ (\cite{Kee92}), $(1,n)$ for $n\leq10$ (\cite{EG98,Inc22,AreObAbr,Bis24}), $(2,n)$ for $n\leq2$ (\cite{Vis98,Per22,Lan24}). When $(g,n)=(3,0)$ the almost integral Chow ring of $\cM_{g,n}$ is computed in\cite{Per23}.
	
	Other computations of Chow rings of moduli spaces related to curves involve stable maps \cite{Cavalieri}, admissible covers \cite{CL-hurwitz,CL-hurwitz2}, or geometrically significant substacks of $\mathcal{M}_{g,n}$, namely the hyperelliptic locus~\cite{EF09,FV11,DL18,Lan23,Lan24,EH22}.
	
	While substantial progress has been made and many conjectures formulated regarding the intersection theory of moduli stacks of curves, relatively little is known about the intersection theory of the moduli stack of Prym pairs. The program to compute the Chow rings of the moduli stack of Prym pairs was initiated in \cite{CIL24}, where the genus 2 case was addressed, and further developed in \cite{CLI}, where the computation of the integral Chow ring of the hyperelliptic locus in arbitrary genus was begun. This paper is the second in this ongoing series of works.
	
	A Prym pair of genus $g \geq 2$ is essentially an étale double cover $\pi: D \to C$ between smooth curves, with $C$ of genus $g$. In this sense, these stacks fit naturally into the framework of moduli stacks of covers of curves. Moreover, such a cover defines a principally polarized abelian variety of dimension $g-1$, known as the associated Prym variety. The close relationship between Prym pairs and abelian varieties has been extensively studied; see, for example, \cite{Mumford, Farkas, Bea, Beabis, DonagiI}. In particular, moduli stacks of Prym pairs play a significant role due to their rich interactions with several other important moduli stacks.
	
	\subsection{Main definitions}
	
	We briefly recall the definition of Prym pairs and their associated stack, referring to \cite[Section 2.1]{CIL24} and \cite[Section 1.1]{CLI} for a more extensive treatment.
	
	A \emph{Prym curve} of genus~$g$ over a scheme $S$ is a triple $(C/S, \eta, \beta)$, where $C \to S$ is a smooth, proper curve of genus~$g$ with geometrically connected fibers, $\eta \in \mathrm{Pic}(C)$ is a line bundle whose restriction to every geometric fiber of $C \to S$ is nontrivial, and $\beta: \eta^{\otimes 2} \xrightarrow{\sim} \mathcal{O}_C$ is an isomorphism of line bundles on $C$.
	
	The stack $\mathcal{R}_g$ is defined as the stackification of the prestack whose objects over a scheme $S$ are families $(C \to S, \eta, \beta)$ of genus $g$ Prym curves over $S$ and a morphism from
	$(C \to S, \eta, \beta)$ to $(C' \to S', \eta', \beta')$ consists of a cartesian diagram
	\begin{equation}\label{eqn: morphism prestack}
		\begin{tikzcd}
			C \arrow{r}{\varphi} \arrow{d} & C' \arrow{d} \\
			S \arrow{r}{f} & S'
		\end{tikzcd}
	\end{equation}
	for which there exists an isomorphism $\tau: \varphi^* \eta' \to \eta$ such that the diagram
	\begin{equation}\label{eqn: diagram of sheaves}
		\begin{tikzcd}
			\varphi^*(\eta'^{\otimes 2}) \arrow{r}{\tau^{\otimes 2}} \arrow{d}[swap]{\varphi^*(\beta')} & \eta^{\otimes 2} \arrow{d}{\beta} \\
			\varphi^* \mathcal{O}_{C'} \arrow{r} & \mathcal{O}_C
		\end{tikzcd}
	\end{equation}
	commutes.
	
	The isomorphism $\tau$ is not part of the data of the morphism in $\mathcal{R}_g$; if it is included, one obtains a $\mu_2$-gerbe over $\mathcal{R}_g$, denoted in \cite{CIL24} by $\widetilde{\R}_g$.
	
	The moduli stack $\RH_g$ of hyperelliptic Prym pairs of genus $g$ is the pullback of the hyperelliptic locus $\mathcal{H}_g \subseteq \mathcal{M}_g$ in the moduli stack of smooth genus $g$ curves under the natural forgetful map $\R_g \to \M_g$. 
	
	This paper concerns the stacks $\mathcal{R}\mathcal{H}_g$. These decompose as the disjoint union of open and closed substacks 
	\begin{equation}\label{eqn: decomposition RH}
		\mathcal{R}\mathcal{H}_g = \bigsqcup_{1 \leq n \leq \frac{g+1}{2}} \mathcal{R}\mathcal{H}_g^n,
	\end{equation}
	where $\mathcal{R}\mathcal{H}_g^n$ is the locus where the line bundle $\eta$ on the curve $C$ is isomorphic to $(n\cdot g^1_2)\otimes \mathcal{O}_C(-e)$, for a reduced Cartier divisor $e$ consisting of $2n$ (distinct) Weierstrass points. Here, $g^1_2$ denotes the unique $g^1_2$ on the curve $C$.
	
	The aim of this paper is to continue the computation of the integral Chow rings of the stacks $\mathcal{R}\mathcal{H}_g^n$ started in \cite{CIL24} and \cite{CLI}. In the present work, we deal with the case where $g$ is odd and $n = (g+1)/2$. This is the most challenging case, as the natural presentation of $\mathcal{RH}_g^n$ as a quotient stack involves a complicated group. Even after expressing $\mathcal{RH}_g^n$ as a $\mu_2$-gerbe over a simpler quotient stack, this new space involves the classifying space of $\PGL_2$, combined with an inconvenient $\mu_2$-action.
	
	\subsection{Statements of the main results}
	
	Let $C$ be a smooth hyperelliptic curve of genus $g$ defined over our base field $k$ and assume $g \geq 2$. Let $E_n$ denote the family of effective reduced divisors supported on $2n$ distinct Weierstrass points of $C$ and let $\beta_n: E_n \to \mathrm{Pic}^0(C)[2] \smallsetminus \{ \cO\}$ be the map sending $e \in E_n$ to $(n \cdot g^1_2) \otimes \cO_C(-e)$ and set $B_n= \beta_n(E_n)$. Then by \cite[Lemma 4.3]{Ver13} (see also \cite[Lemma 3]{CIL24} and \cite[Lemma 1.1]{CLI}) 
	\begin{itemize}
		\item[(i)] for $n \leq \lfloor \frac{g}{2} \rfloor$ the map $\beta_n$ is injective;
		\item[(ii)] for $n= \frac{g+1}{2}$ the map $\beta_n$ is $2:1$ over its image;
		\item[(iii)] we have a disjoint union $ \mathrm{Pic}^0(C)[2] \smallsetminus \{ \mathcal{O}_C \}= \bigsqcup_{1 \leq n \leq \frac{g+1}{2}} B_n$.
	\end{itemize}
	
	The decomposition \eqref{eqn: decomposition RH} follows from this. See \cite[Section 1.2]{CLI}.
	Let $\mathcal{H}_g$ denote the moduli stack of hyperelliptic curves of genus $g$. For $a \geq 1$, let $\mathcal{D}_a$ denote the moduli stack parametrizing pairs $(\mathcal{P} \to S, D_a \subseteq \mathcal{P})$, where $\mathcal{P} \to S$ is a Brauer–Severi scheme of relative dimension $1$, and $D_a$ is a Cartier divisor that is finite and \'etale over $S$ of degree $a$. There is a natural map $\mathcal{H}_g \to \mathcal{D}_{2g+2}$ which associate to a genus $2$ curve $C$ the branch locus of the unique (up to $\PGL_2$ action) $2:1$ map $ C \to \P^1$.
	
	For $a, b \geq 1$, let $\mathcal{D}_{a,b}$ denote the open substack of $\mathcal{D}_a \times_{B\mathrm{PGL}_2} \mathcal{D}_b$ parametrizing tuples $(\mathcal{P} \to S, D_a, D_b)$, where $D_a$ and $D_b$ are disjoint Cartier divisors in $\mathcal{P}$ that are finite and étale of degrees $a$ and $b$, respectively, over $S$.
	
	\begin{lemma}\label{lemma: key cartesian diagram}
		We have a cartesian diagram
		\[
		\begin{tikzcd}
			\bigsqcup_{1 \leq n \leq \frac{g+1}{2}} \mathcal{R}\mathcal{H}_g^n \arrow[d] \arrow[r] & \mathcal{H}_g \arrow[d] \\
			\left[ \frac{\mathcal{D}_{g+1,g+1}}{\mu_2} \right] \sqcup \bigsqcup_{1 \leq n < \frac{g+1}{2}} \mathcal{D}_{2n, 2g+2 - 2n} \arrow[r] & \mathcal{D}_{2g+2}
		\end{tikzcd}
		\]
		
		Here, the bottom horizontal map corresponds to taking the sum of $D_{2n}$ and $D_{2g+2 - 2n}$, while the top horizontal map records the associated hyperelliptic curve. The left vertical map records the image, under the unique (up to $\mathrm{PGL}_2$) degree $2$ map from the hyperelliptic curve to $\mathbb{P}^1$, of the degree $2n$ divisor $e \in E_n$ and OF its complement in the full Weierstrass divisor $W \subseteq C$.
		
		Moreover, the vertical map on the left respects the disjoint union indexed by $n$; that is, each $\mathcal{RH}_g^n$ maps to $\mathcal{D}_{2n, 2g+2 - 2n}$ (or to $[\mathcal{D}_{g+1, g+1}/\mu_2]$ when $g$ is odd and $n = (g+1)/2$). Here, the group $\mu_2$ acts on $\mathcal{D}_{g+1, g+1}$ by swapping the two degree $g+1$ divisors.
		
	\end{lemma} 
	
	\begin{proof}
		The proof is essentially the same as that of \cite[Lemma 1.14]{CLI}. The case where $g$ is odd and $n = (g+1)/2$ is special because, in this case, the map
		\[
		\beta_n\colon E_n \to B_n \subseteq \mathrm{Pic}^0(C)[2] \smallsetminus \{ \mathcal{O} \}
		\]
		is $2:1$ onto its image. More precisely, for all $e \in E_n$, we have
		\[
		\beta_n^{-1}(\beta_n(e)) = \{ e,\ w - e \},
		\]
		where $w$ is the Weierstrass divisor of $C$. In particular, two Prym pairs $(C,g^1_2(-e),\beta)$ and $(C,g^1_2(-(w-e)),\beta')$ are isomorphic in $\RH_g$.

		In other words, given a Prym pair $(C/S, \eta, \beta)$, we can extract from it two divisors: one of degree $2n$ and one of degree $2g + 2-2n$. However, when $n = (g+1)/2$, we have $2n =2g + 2-2n$ and thus the two degrees coincide, making it impossible to distinguish the two divisors. This explains the necessity of taking a further quotient by $\mu_2$.
	\end{proof}

	The map $\mathcal{H}_g \to \mathcal{D}_{2g+2}$ has been studied extensively and we will use several of the results in  \cite{AV04},\cite{EH22}.
	
	\subsubsection{Presentations of $[\cD_{g+1,g+1}/ \mu_2]$ and $\RH_g^{(g+1)/2}$ }\label{subsec: pres g odd}
	
	Suppose that $g \geq 2$ is odd and that $n=(g+1)/2$.
	\begin{notation}\label{notation: groups and representations}
		Let $G:=\GI \subseteq \GL_2$ be the subgroup of matrices preserving the set of lines $\{ k (1,0), k(0,1)\} \subseteq k^2$. Equivalently, this is the subgroup of $\GL_2$ consisting of matrices of the form 
		$$
		(a,b;0):=
		\begin{pmatrix}
			a & 0 \\
			0 & b 
		\end{pmatrix}
		\ \text{or} \ 
		(a,b;1):=
		\begin{pmatrix}
			0 & a \\
			b & 0 
		\end{pmatrix}
		$$
		for $a,b \in k^*$. 
		
		We will denote by $V$ be the representation of $G$ induced by the standard representation of $\GL_2$ via the inclusion $G \subseteq \GL_2$ and by $\Gamma$ the representation of $G$ induced by the sign representation of $\mu_2$ via the quotient map $G \to \mu_2$.
		
		For every $j \geq 1$, we will denote by $W_j$ the $2j+1$-dimensional representation of $\PGL_2$ where the underlying vector space is the space of homogeneous polynomials of degree $2j$ in two variables and 
		$$
		[B] \cdot f(X,Y)= \det (B)^j (f \circ B^{-1}) (X,Y).
		$$
		
		Finally, we will denote by $\chi$ the standard representation of $\Gm$ and by $t \in \CH^*(B \Gm)$ its first Chern class.
		
	\end{notation}  
	
	When $g$ is odd, there exists an explicit description of the map $\cH_g \to \cD_{2g+2}$ in terms of quotient stacks:
	\begin{equation}\label{eqn: map in EH}
		\begin{tikzcd}
			\mathcal{H}_g \cong \bigg[ \frac{\chi^{\otimes -2} \otimes W_{g+1} \smallsetminus \Delta}{\mathbb{G}_m \times \mathrm{PGL}_2} \bigg] \arrow[r] & \mathcal{D}_{2g+2} \cong \bigg[ \frac{\chi^{\otimes -1} \otimes W_{g+1} \smallsetminus \Delta}{\mathbb{G}_m \times \mathrm{PGL}_2} \bigg] \cong \bigg[ \frac{\P(W_{g+1}) \smallsetminus \underline{\Delta}}{ \mathrm{PGL}_2} \bigg].
		\end{tikzcd}
	\end{equation}
	See  \cite[Theorem 4.1]{AV04} and \cite[Remark 2.1]{EH22}. This is the base change of the morphism $B \Gm \to B\Gm$ given by $t \mapsto t^2$ via $[ (\chi^{\otimes -1} \otimes W_{g+1} \smallsetminus \Delta) /\mathbb{G}_m \times \mathrm{PGL}_2 ] \to B\Gm$. In particular, $\cH_g$ is the $\mu_2$-root-stack over $\cD_{2g+2}$ obtained by adding a square root of the line bundle $\cO_{\P(W_{g+1})}(-1)$. The loci $\Delta$ and $\underline{\Delta}$ denote the locus of singular polynomials in the affine and projective spaces, respectively. 
	
	\begin{theorem}\label{thm: presentation Dg+1g+1mu2}
		Suppose the genus $g \geq 3$ is odd. Then, we have an isomorphism of algebraic stacks
		$$
		\left[ \frac{\mathcal{D}_{g+1,g+1}}{\mu_2} \right] \cong \Bigg[ \frac{V^\vee \otimes W_{\frac{g+1}{2}} \smallsetminus \Delta}{G\times \PGL_2 }  \Bigg]
		$$
		where $\Delta$ is the locus of pairs of polynomials $(f,g)$ whose product $fg$ is singular.
	\end{theorem}

	The proof is presented in \S\ref{subsec: pres g odd}. Combining this with Lemma \ref{lemma: key cartesian diagram}, we will obtain the following.
	
	\begin{theorem}\label{thm: presentation RHg,g+1/2}
		Suppose the genus $g \geq 3$ is odd. Then, we have an isomorphism of algebraic stacks
		$$
		\RH_g^{(g+1)/2} \cong \Bigg[ \frac{V^\vee \otimes W_{\frac{g+1}{2}} \smallsetminus \Delta}{H }  \Bigg]
		$$
		where $H =\{ ((a,b;\varepsilon),t,[B]) \ | \ t^2= (-1)^\varepsilon \mathrm{det}(A)\} \subseteq G \times \Gm \times \PGL_2$ acts on $V^\vee \otimes W_{\frac{g+1}{2}}$ via the projection $H \to G \times \PGL_2$.
	\end{theorem}

	\subsubsection{The Chow rings $\CH^*(\RH_g^{(g+1)/2})$ and $\CH^*([D_{g+1,g+1}/\mu_2])$}
	
	The difficulty in computing $\CH^*(\mathcal{R}\mathcal{H}_g^{(g+1)/2})$ using the presentation in Theorem~\ref{thm: presentation Dg+1g+1mu2} arises already from the fact that the Chow ring of the classifying space $BH$ is not known. To avoid this issue, we will not use the above presentation. Instead, we apply \cite[Proposition~3.5]{CLI}, which computes the integral Chow ring of a $\mu_2$-root gerbe from that of the base and the first Chern class of the line bundle whose root has been added.
	
	Before stating our result we require further notation. Recall that the Chow ring of the classifying stack $B \PGL_2$ is known \cite{pandharipande1996chow,Vez98,DL18}:
	
	\begin{equation}
		\CH^*(B \PGL_2)= \frac{\Z[c_2,c_3]}{(2c_3)}.
	\end{equation}
	
	The classes $c_i$ for $i = 2,3$ denote the Chern classes of the representation $W_1$. Since $W_1$ is isomorphic to the adjoint representation $\mathsf{sl}_2$ of $\PGL_2$, these are equivalently the Chern classes of $\mathsf{sl}_2$, as discussed in~\cite{Vez98}. Moreover, there is an isomorphism of stacks  $B \PGL_2 \cong [\mathcal{S}/\GL_3]$ 
	where $\mathcal{S}$ denotes the space of smooth homogeneous polynomials of degree~2 in two variables (see~\cite[Proposition~1.5]{DL18}). Under this identification, the classes $c_i$ arise as the pullbacks of the Chern classes of the standard $\GL_3$-representation from $B \GL_3$, as shown in~\cite[Lemma~1.3]{DL18}.

	\begin{notation}
		We set $\gamma = c_1(\Gamma) \in \CH^*(B\mu_2)$, $\beta_i = c_i(V) \in \CH^*(B \GL_2)$ and will use the same symbols to denote their pullbacks to other spaces.
	\end{notation}

	We can now state our first Chow ring computation.
	
	\begin{theorem}\label{thm: Chow Dg+1g+1mu2}
		Suppose that $g \geq 3$ is odd. Then we have
		$$
		\CH^*([\mathcal{D}_{g+1,g+1}/\mu_2]) \cong \frac{\Z[\beta_1, \beta_2, \gamma, c_2, c_3]}{I}
		$$
		where $I$ is the ideal generated by the following relations:
		\begin{align*}
			\bullet & \ c_1, \ 2c_3, \ 2\gamma, \ \gamma(\beta_1+\gamma), \quad \text{coming from }B(G\times\PGL_2)\\
			\bullet & \ 2\beta_1\\
			\bullet & \ 4\beta_2+(g^2-1)c_2\\
			\bullet & \ \beta_1^4+\beta_1^3\gamma+2\beta_2^2+\beta_1^2c_2+\beta_1\gamma c_2-(g^2+1)\beta_2c_2+c_3(\beta_1+\gamma)+\frac{(g^2-1)^2}{8}c_2^2\\
			\bullet & \ \beta_1^3\beta_2+\beta_1^2\beta_2\gamma+\beta_1\beta_2^2-\beta_2^2\gamma+\beta_1\beta_2c_2+\beta_2\gamma c_2\\
			\bullet & \ 2\beta_2-\frac{(g+1)^2}{2}c_2\\
			\bullet & \ -g\beta_2+\left(\frac{g+1}{2}\right)\beta_1^2+g\left(\frac{g+1}{2}\right)^2c_2 \\
			\bullet & \ \beta_1\beta_2+\left(\frac{g+1}{2}\right)^2c_3 \\
			\bullet & \ \beta_2^2+\left(\frac{g^2-1}{2}\right)\beta_2c_2+\left(\frac{g+1}{2}\right)^2\left(\frac{g-1}{2}\right)^2c_2^2+\left(\frac{g+1}{2}\right)\beta_1c_3.
		\end{align*}
	\end{theorem}
	
	Using \cite[Proposition 3.5]{CLI} and identifying the pullback of $\cO_{\P(W_{g+1})}(1)$ under $[\cD_{g+1,g+1}/\mu_2] \to \cD_{2g+2}$, in \S\ref{subsec: conclusion computation} we obtain:
	
	\begin{theorem}\label{thm: Chow RHgg+1 odd}
		Suppose that $g \geq 3$ is odd. Then we have
		$$
		\CH^*(\RH_g^{(g+1)/2}) \cong \frac{\Z[\beta_1, \beta_2, \gamma, c_2, c_3,t]}{I+(2t-(\beta_1+\gamma))}
		$$
		where $I$ is the ideal generated by the classes listed in Theorem~\ref{thm: Chow Dg+1g+1mu2}.
	\end{theorem}
	
	To conclude, in \S\ref{subsec: interpretation generators g+1}, we will give a geometric interpretation of these generators as Chern classes of certain natural vector bundles on $\RH_g^{(g+1)/2}$.
	
	\subsubsection{Hyperelliptic Spin Pairs of odd genus}\label{sec: Spin curves}
	
	An interesting application of our result concerns hyperelliptic Spin curves of odd genus.
	
	Fix a genus $g \geq 2$. The moduli stack of Spin curves $\cS_g$ is defined similarly to $\mathcal{R}_g$, as the stackification of the prestack whose objects are triples $(C/S, \theta, \alpha)$, where $C \to S$ is a smooth genus $g$ curve, $\theta \in \mathrm{Pic}(C)$ is a line bundle, and $\alpha: \theta^{\otimes 2} \to \omega_{C/S}$ is an isomorphism. Morphisms $(C/S, \theta, \alpha) \to (C'/S', \theta', \alpha')$ in this prestack are given by cartesian diagrams as in \eqref{eqn: morphism prestack}, for which there exists a morphism $\tau: \varphi^* \theta' \to \theta$ such that the diagram corresponding to \eqref{eqn: diagram of sheaves} commutes, with $\cO_{C'}$ and $\cO_C$ replaced respectively by $\omega_{C'/S'}$ and $\omega_{C/S}$. We refer to \cite{Cor89} for an introduction to Spin curves, often also called theta characteristics. 
	
	There is a natural (representable) finite and étale morphism $\mathcal{S}_g \to \mathcal{M}_g$ of degree $2^{2g}$. We denote by $\mathcal{SH}_g$ the restriction of $\cS_g$ to $\cH_g$.
	
	Let $\widetilde{\mathcal{J}}_g$ be the stack parametrizing smooth genus~$g$ curves together with line bundles of degree~$0$, with morphisms given by maps between the curves and the corresponding line bundles. Then $\widetilde{\mathcal{J}}_g$ carries a natural $\mathbb{G}_m$-2-structure, whose rigidification we denote by $\mathcal{J}_g$. Let $\mathcal{J}_g[2] \subseteq \mathcal{J}_g$ be the closed substack where the line bundle is $2$-torsion. There is an open and closed decomposition (see \cite[Proposition 11]{CIL24})
	\begin{equation}\label{eqn: decomposition J[2]}
		\mathcal{J}_g[2] = \mathcal{R}_g \sqcup \mathcal{M}_g,
	\end{equation}
	where $\mathcal{M}_g \subseteq \mathcal{J}_g[2]$ is the substack corresponding to the section given by the structure sheaf.
	
	Furthermore, there is a natural free and transitive group action 
	$$
	\mathcal{J}_g[2] \curvearrowright \cS_g \to \mathcal{M}_g
	$$
	over $\mathcal{M}_g$, obtained by twisting a Spin curve with a 2-torsion line bundle.
	
	\begin{proposition}\label{prop: decomposition Spin}
		When $g$ is odd, the forgetful map $\mathcal{SH}_g \to \mathcal{H}_g$ has a section. In particular, we have a decomposition 
		$$
		\mathcal{SH}_g \cong \cH_g \sqcup \mathcal{RH}_g.
		$$
	\end{proposition}
	
	The proof is given in \S\ref{sec: Spin pairs}. As a consequence, combining Theorems~\ref{thm: presentation RHg,g+1/2} and~\ref{thm: Chow RHgg+1 odd} with \cite[Theorems~1.18 and~1.19]{CLI}, we obtain presentations and Chow rings for all components of $\mathcal{SH}_g$ when $g$ is odd.
	\subsection*{Convention}
	We work over an algebraically closed field $k$ of characteristic $0$ or bigger than $2g+2$. For a discussion on why this assumption is necessary, see~\cite[Example 1.1]{CLI}.
	
	\subsection{Acknowledgments}
	Part of this work was carried out during the second author's visit to the University of Cambridge. We are grateful to this institution for its hospitality. We thank Samir Canning, Brendan Hassett, Aitor Iribar Lopez, Michele Pernice, Dhruv Ranganathan, and Bernardo Tarini for discussions related to the topic of this work. We are also grateful to Suichu Zhang for pointing out the connection between Prym and Spin curves described in \S\ref{sec: Spin curves}, during the 2025 Summer Research Institute in Algebraic Geometry at Colorado State University. A.C. is supported by SNF grant P500PT-222363.
	
	\section{Proof of the presentations of $[\cD_{g+1,g+1}/ \mu_2]$ and $\RH_g^{(g+1)/2}$ }\label{sec: presentation g odd}
	
	In this section we prove Theorems \ref{thm: presentation Dg+1g+1mu2} and \ref{thm: presentation RHg,g+1/2}.
	
	\begin{proof}[Proof of Theorem \ref{thm: presentation Dg+1g+1mu2}]
		We start with recalling that, by \cite[Equation 4]{CLI}, we have
		
		$$
		\mathcal{D}_{g+1,g+1} \cong \bigg[ \frac{(\chi^{(1)})^{-1} \otimes W_{\frac{g+1}{2}} \times (\chi^{(2)})^{-1} \otimes W_{\frac{g+1}{2}} \smallsetminus \Delta}{\mathbb{G}_m^{(1)} \times \mathbb{G}_m^{(2)} \times \mathrm{PGL}_2} \bigg]
		$$
		
		where $\Delta$ denotes the locus of pairs of polynomials $(f, g)$ such that the product $fg$ is singular. The groups $\mathbb{G}_m^{(1)}$ and $\mathbb{G}_m^{(2)}$ are both equal to $\mathbb{G}_m$, with the superscripts indicating that each acts only on the corresponding factor in the product. Similarly, the notation $\chi^{(i)}$ refers to the same character $\chi$, where the superscript specifies which $\mathbb{G}_m$ factor is acting.
		
		There is a natural map
		\[
		\tau^{\mathrm{pre}}: \Bigg[ \frac{V^\vee \otimes W_{\frac{g+1}{2}} \smallsetminus \Delta}{G \times \PGL_2} \Bigg]^{\mathrm{pre}} \longrightarrow \left[ \frac{\mathcal{D}_{g+1,g+1}}{\mu_2} \right],
		\]
		defined as follows. On objects, it assigns to a pair of polynomials $(f,g)$ over a scheme $S$ their respective zero loci in $\mathbb{P}^1_S$.
		
		Let $(f_1,g_1) \to (f_2,g_2)$ be a morphism in the quotient prestack, given by a morphism $\varphi: S_1 \to S_2$ and an element $g = (A, [B]) \in G \times \PGL_2(S_1)$ such that
		\[
		\varphi^*(f_2, g_2) = g \cdot (f_1, g_1).
		\]
		Let $F_i$ and $G_i$ denote the zero divisors of $f_i$ and $g_i$ in $\mathbb{P}^1_{S_i}$ for $i = 1,2$. Let $\phi: \mathbb{P}^1_{S_1} \to \mathbb{P}^1_{S_2}$ be the morphism induced by $[B]$ and $\varphi$. Then, as Cartier divisors, we have either $\phi^* F_2 = F_1$ and $\phi^* G_2 = G_1$, or $\phi^* F_2 = G_1$ and $\phi^* G_2 = F_1$. This defines a morphism
		\[
		(F_1, G_1 \subseteq \mathbb{P}^1_{S_1} \to S_1) \longrightarrow (F_2, G_2 \subseteq \mathbb{P}^1_{S_2} \to S_2)
		\]
		in the stack $[ \mathcal{D}_{g+1,g+1} / \mu_2 ]$, and thereby defines $\tau^{\mathrm{pre}}$ on morphisms.

		We denote by $\tau$ the stackification of $\tau^{\mathrm{pre}}$. Since the $\mu_2$-torsor $\mathcal{D}_{g+1,g+1} \to [\mathcal{D}_{g+1,g+1}/\mu_2]$ factors through $\tau$ via the natural $\mu_2$-torsor
		\[
		\mathcal{D}_{g+1,g+1} \cong \bigg[ \frac{(\chi^{(1)})^{-1} \otimes W_{\frac{g+1}{2}} \times (\chi^{(2)})^{-1} \otimes W_{\frac{g+1}{2}} \smallsetminus \Delta}{\mathbb{G}_m^{(1)} \times \mathbb{G}_m^{(2)} \times \mathrm{PGL}_2} \bigg] \longrightarrow \Bigg[ \frac{V^\vee \otimes W_{\frac{g+1}{2}} \smallsetminus \Delta}{G \times \PGL_2} \Bigg],
		\]
		the morphism $\tau$ must be an isomorphism.

	\end{proof}

	\begin{proof}[Proof of Theorem \ref{thm: presentation RHg,g+1/2}]
		The cartesian diagram in Lemma~\ref{lemma: key cartesian diagram} and Theorem~\ref{thm: presentation Dg+1g+1mu2} show that $\RH_g^{(g+1)/2}$ is obtained as the base-change
		\[
		\begin{tikzcd}
			\RH_g^{(g+1)/2} \arrow[r]\arrow[d] & \cH_g = \bigg[ \frac{\chi^{ \otimes -2} \otimes W_{g+1} \smallsetminus \Delta}{\Gm \times \PGL_2} \bigg] \arrow[r]\arrow[d] & B \Gm \arrow[d]\\ 
			\big[\mathcal{D}_{g+1,g+1} / \mu_2 \big] = \Bigg[ \frac{V^\vee \otimes W_{\frac{g+1}{2}} \smallsetminus \Delta}{ G \times \PGL_2 }  \Bigg] \arrow[r] & \mathcal{D}_{2g+2} = \bigg[ \frac{\chi^{\otimes -1} \otimes W_{g+1} \smallsetminus \Delta}{\mathbb{G}_m \times \mathrm{PGL}_2} \bigg] \arrow[r] & B\Gm
		\end{tikzcd}
		\]
		where both squares are cartesian, the rightmost map is induced by the squaring homomorphism $\Gm \to \Gm$, and the first horizontal arrow at the bottom is given by the product $(f,g) \mapsto fg$, together with the group homomorphism $G \times \PGL_2 \to \Gm \times \PGL_2$ defined by $((a,b;\varepsilon), [B]) \mapsto (ab, [B])$.
		
		The theorem follows observing that $H$ is precisely the base-change of this map with the squaring map $\Gm \to \Gm$.
	\end{proof}
	
	\section{Auxiliary computations}

	\subsection{Chow rings of the classifying spaces of the groups $H_{\ell_1,\ell_2}$} \label{sec: groups H}
	
	For every $l_1,l_2\geq0$, we define $H_{\ell_1,\ell_2}:=(\Gm^{\ell_1}\times (\Gm^{\ell_2})^{\times 2})\rtimes\mu_2$, where $\mu_2$ acts on $\Gm^{\ell_1}$ by $-1\cdot t=t^{-1}$, and on $(\Gm^{\ell_2})^{\times 2}$ by swapping the two components. Note that, $H_{0,1}=G$.
	
	Let $\psi:B(\Gm^{\ell_1}\times (\Gm^{\ell_2})^{\times 2})\rightarrow BH_{\ell_1,\ell_2}$ be the double cover induced by the natural inclusion of groups. Denote by $\phi:BH_{\ell_1,\ell_2}\rightarrow B(\Gm^{\ell_1}\rtimes\mu_2)$ induced by projection $H_{\ell_1,\ell_2} \rightarrow \Gm^{\ell_1}\rtimes\mu_2$. For every $1\leq j\leq \ell_2$, denote by $\pi_j:BH_{\ell_1,\ell_2}\rightarrow B\GL_2$ the morphism induced by the composite $H_{\ell_1,\ell_2}\rightarrow G \subset\GL_2$, where the first morphism is the projection to the $j$-th component of $(\Gm^{\ell_2})^{\times 2}\rtimes\mu_2$.
	
	Recall also that $\beta_2\in\CH^*(B\GL_2)$ denotes the second Chern class of the standard $\GL_2$-representation $V$.
	
	\begin{proposition}\label{prop: Chow classifying Hl1l2}
		Every element $\alpha\in\CH^*(BH_{\ell_1,\ell_2})$ can be written as $\alpha=\alpha_1+\alpha_2$, where $\alpha_1\in\mathrm{im}(\psi_*)$, and $\alpha_2$ is sum of products of classes of the form $\pi_j^*(\beta_2^i)$ for $1\leq j\leq l_2$ and $i\geq0$, and classes of the form $\phi^*(\epsilon)$ for $\epsilon\in\CH^*(B(\Gm^{\ell_1}\rtimes\mu_2))$.     
	\end{proposition}
	\begin{proof}
		We prove the statement by induction on $\ell_2$, where the case $\ell_2=0$ is obvious. Then, assume $\ell_2\geq1$ and that the statement is true for every $\ell_2'\leq \ell_2$. For $n \geq 0$ let $E_n$ be the the $H_{\ell_1,\ell_2}$-representation $E_n=\pi_{\ell_2}^*((V^\vee)^{n+1})\cong \mathbb{A}^{n+1}\times\mathbb{A}^{n+1}$ and let $B_n:=\{0\}\times\mathbb{A}^{n+1}\cup\mathbb{A}^{n+1}\times\{0\}\subseteq E_n$. Note that the surjective ring homomorphism
		\[
		\begin{tikzcd}
			\CH^*(BH_{\ell_1,\ell_2})\arrow[r] & \CH^*\bigg( \bigg[ \frac{E_n \smallsetminus B_n }{ H_{\ell_1,\ell_2}}\bigg] \bigg) \cong \CH^*\bigg( \bigg[ \frac{\P^n\times\P^n }{ H_{\ell_1,\ell_2-1}}\bigg] \bigg) 
		\end{tikzcd}
		\]
		is an isomorphism in degree $\leq n$; see~\cite[Lemma 4.3]{AI17} for the last isomorphism. It therefore suffices to show that for all $n \geq 0$ the group $\CH^*_{H_{\ell_1,\ell_2-1}}(\P^n \times \P^n)$ is generated in degrees $\leq n$ by the classes appearing in the statement (viewed in this ring via the map above). Consider the open/closed decomposition 
		\[
		\bigg[\frac{ \mathbb{A}^n\times\mathbb{A}^n}{H_{\ell_1,\ell_2-1}} \bigg] \subset  \bigg[\frac{ \P^n\times\P^n}{H_{\ell_1,\ell_2-1}} \bigg] \supset   \bigg[\frac{ \P^{n}\times\P^{n-1} \cup \P^{n-1}\times\P^{n}}{H_{\ell_1,\ell_2-1}} \bigg] :=\mathcal{W},
		\]
		where $\mathbb{A}^n\subset\P^n$ is the open subscheme where the last coordinate is different from $0$. Notice that the morphism
		$$
		\bigg[ \frac{\P^n\times\P^{n-1}\sqcup\P^{n-1}\times\P^n}{H_{\ell_1,\ell_2-1}} \bigg ] \sqcup \bigg[ \frac{\P^{n-1}\times\P^{n-1}}{H_{\ell_1,\ell_2-1}} \bigg] \xrightarrow{\phi_1 \sqcup \phi_2} \mathcal{W}
		$$
		given by the normalization map and the natural inclusion, induces a surjection on Chow rings. Moreover, we have an isomorphism
		$$
		\bigg[ \frac{\P^n\times\P^{n-1}}{\Gm^{\ell_1}\times (\Gm^{\ell_{2}-1})^{\times 2}} \bigg ] \cong \bigg[ \frac{\P^n\times\P^{n-1}\sqcup\P^{n-1}\times\P^n}{H_{\ell_1, \ell_2-1}} \bigg ]
		$$
		induced by the inclusion on thee first factor. Putting everything together, we get an exact sequence
		\[
		\begin{tikzcd}[column sep=small]
			\CH_{\Gm^{\ell_1}\times (\Gm^{\ell_2})^{\times 2}}^*(\P^n\times\P^{n-1})\oplus\CH_{H_{\ell_1,\ell_2-1}}^*((\P^{n-1})^2)\arrow[r] & \CH_{H_{\ell_1,\ell_2-1}}^*((\P^n)^2)\arrow[r] & \CH^*(BH_{\ell_1,\ell_2-1})\arrow[r] & 0.
		\end{tikzcd}
		\]
		By the inductive hypothesis and the exact sequence above, it suffices to show that every element in the image of the first homomorphism of groups is of the form described in the statement. We prove this by induction on $n$, with the case $n = 0$ being immediate. The map $\phi_1$ factors as
		\[
		\begin{tikzcd}
			\bigg[ \frac{\P^n\times\P^{n-1}}{\Gm^{\ell_1}\times (\Gm^{\ell_{2}-1})^{\times 2}} \bigg ]  \arrow[r,hookrightarrow] & \bigg[ \frac{\P^n\times\P^{n}}{\Gm^{\ell_1}\times (\Gm^{\ell_{2}-1})^{\times 2}} \bigg ] \arrow[r,"\psi_n"] & \bigg[ \frac{\P^n\times\P^{n}}{H_{\ell_1,\ell_2-1}} \bigg ] ,
		\end{tikzcd}
		\]
		implying that the image of $\phi_{1*}$ is contained in $\mathrm{im}(\psi_*)$. For $\phi_{2*}$, by the inductive hypothesis and the fact that all the classes on $[ (\P^{n-1})^2 / H_{\ell_1, \ell_2 - 1} ]$ appearing in the statement come from the image of $\phi_2^*$, it remains only to show that $\phi_{2*}(1) = \beta_2$.
		
		This follows from the following cartesian diagram
		\[
		\begin{tikzcd}
			\bigg[ \frac{\P^{n-1}\times\P^{n-1}}{H_{\ell_1,\ell_2-1}} \bigg]\arrow[d,"\cong"]\arrow[r,"\phi_2"] & {\bigg[ \frac{\P^n\times\P^n}{H_{\ell_1,\ell_2-1}} \bigg]}\arrow[d,"\cong"]\\
			\bigg[ \frac{(V^\vee)^n\smallsetminus B_{n-1}}{H_{\ell_1,\ell_2}} \bigg] \arrow[r]\arrow[d,"\pi_{\ell_2}"] & \bigg[ \frac{(V^\vee)^{n+1}\smallsetminus B_{n}}{H_{\ell_1,\ell_2}} \bigg]\arrow[d,"\pi_{\ell_2}"]\\
			\bigg[ \frac{(V^\vee)^n\smallsetminus B_{n-1}}{\GL_2} \bigg]\arrow[r,"v"] & \bigg[ \frac{(V^\vee)^{n+1}\smallsetminus B_{n}}{\GL_2} \bigg],
		\end{tikzcd}
		\]
		and the fact that $v_*(1)=\beta_2$.
	\end{proof}
	
	\begin{corollary}\label{cor: Chow2 of H11}
		The group $\CH^2(BH_{1,1})$ is given by 
		$$
		\CH^2(BH_{1,1})= \Z \langle \pi_1^*(\beta_2) \rangle \oplus  \Z \langle \psi_*(t x_1) \rangle \oplus \Z \langle \psi_*(x_1^2) \rangle \oplus  \Z \langle \phi^*(c_2) \rangle \oplus \frac{ \Z}{2\Z} \langle \phi^*(\gamma^2) \rangle.
		$$ 
		Moreover, the kernel of the homomorphism of groups $\psi^*: \CH^2(B H_{1,1}) \to \CH^2(B \Gm \times \Gm^{\times 2})$ is the subgroup $\langle \gamma^2 \rangle$.
	\end{corollary}
	\begin{proof}
		By Proposition~\ref{prop: Chow classifying Hl1l2}, we know that $\CH^2(BH_{1,1})$ is additively generated by the image of $\psi_*$ and the classes $\phi^*(\gamma^2)$, $\phi^*(c_2)$ and $\pi_1^*(\beta_2)$. See \cite[Theorem 5.2]{Lar19} for the Chow ring of $G=H_{0,1}$. Since $\CH(B(\Gm\times\Gm^2))\cong\Z[t,x_1,x_2]$, where $t$, $x_1$ and $x_2$ are the first Chern classes of $\chi$ associated to the first, second and third projection respectively, we know that the image of $\psi_*$ is additively generated by the classes 
		\[
		\psi_*(t^2),\quad\psi_*(tx_1),\quad\psi_*(tx_2),\quad\psi_*(x_1^2),\quad\psi_*(x_2^2),\quad\psi_*(x_1x_2).
		\]
		First, notice that $\psi_*(x_1x_2)=\psi_*(\psi^*\pi_1^*\beta_2)=2\pi_1^*(\beta_2)\in\Z\langle\pi_1^*(\beta_2)\rangle$. Moreover, by $\mu_2$-invariance, we have $\psi_*(x_1^2)=\psi_*(x_2^2)$ and $\psi_*(tx_1)=-\psi_*(tx_2)$. Also, $\psi^*(\phi^*(c_2))=t^2$, hence $\psi_*(t^2)=2\phi^*(c_2)$. It follows that $\CH^*(BH_{1,1})$ is additively generated by the classes in the statement and it is enough to prove that the only relation between them is $2\gamma^2=0$ ( which we already know to hold). To prove it, we compute the pullback along $\psi$, showing that the only element that is sent to $0$ to zero is $\gamma^2$. Notice that this would leave open the possibility of $\gamma^2$ being 0, which is ruled out by considering its the pullback along $B(\Gm\rtimes\mu_2)\rightarrow BH_{1,1}$. Clearly, $\psi^*(\gamma^2)=0$. By Lemma~\ref{lemma: push-pull psi} below, we have that $\psi^*\psi_*(tx_1)=t(x_1-x_2)$ and $\psi^*\psi_*(x_1^2)=x_1^2+x_2^2$. Finally, we have that $\psi^*\phi^*(c_2)=t^2$ and $\psi^*\pi_1^*(\beta_2)=x_1x_2$. As these four classes are linearly independent in $\CH^2(B(\Gm\times\Gm^2))$, the result follows.
	\end{proof}
	
	In the previous proof we have used the next well known result, which we record here also for ease of reference later in the paper.
	
	\begin{lemma}\label{lemma: push-pull psi}
		Let $f:\mathcal{X}\rightarrow\mathcal{Y}$ be a $\mu_2$-torsor of quotient stacks over a field $k$ of characteristic different from 2, and let $\mu:\mathcal{X}\rightarrow\mathcal{X}$ be the automorphism relative to $-1\in\mu_2$. Then, for every $\alpha\in\CH^*(\mathcal{X})$ we have
		\[
		f^*f_*(\alpha)=\alpha+\mu^*(\alpha).
		\]
	\end{lemma}
	\begin{proof}
		It follows immediately from~\cite[Example 1.7.6]{Ful98}.
	\end{proof}

	\subsection{$\GL_3$-counterparts}\label{subsec: GL3 counterparts}
	In this subsection we briefly recall the main definitions and results on $\GL_3$-counterparts. For a more thorough treatment see the paper~\cite{DL18} where they were firstly introduced, or~\cite[Subsection 3.4]{CLI}.
	
	\begin{definition}\label{def: GL3 counterparts}\cite[Definition 1.3 and Definition 2.2]{DL18}
		Let $H$ be a smooth affine group scheme over $k$. Let $X$ be a scheme of finite type over $\mathrm{Spec}(k)$ endowed with a $H\times\PGL_2$-action. Then a $H\times\GL_3$-counterpart of $X$ is a scheme $X'$ endowed with a $H\times\GL_3$-action and an isomorphism $ [X'/H\times\GL_3] \cong [X/H\times\PGL_2]$. Given two schemes $X$ and $Y$ with a $H\times\PGL_2$-action, and a $H\times\PGL_2$-equivariant morphism $f:X\rightarrow Y$, a $H\times\GL_3$-counterpart of $f$ is the datum of $H\times\GL_3$-counterparts $X'$ and $Y'$ of $X$ and $Y$ respectively, together with a $H\times\GL_3$-equivariant morphism $f':X'\rightarrow Y'$ such that the diagram
		\[
		\begin{tikzcd}
			{\left[\frac{X}{H\times\PGL_2}\right]}\arrow[r,"f"]\arrow[d,"\cong"'] & {\left[\frac{Y}{H\times\PGL_2}\right]}\arrow[d,"\cong"]\\
			{\left[\frac{X'}{H\times\GL_3}\right]}\arrow[r,"f'"] & {\left[\frac{Y'}{H\times\GL_3}\right]}
		\end{tikzcd}
		\]
		commutes.
	\end{definition}
	We will mainly be interested in the case when $H$ is trivial or $H=\mu_2$.
	In \cite[Section 1]{DL18}, the author proved that every proper $\PGL_2$-morphism admits a $\GL_3$-counterpart, and he explicitly computed a $\GL_3$-counterpart of the quotient stack $[\P(W_m) / \PGL_2]$. Here, we recall his construction as well as the definition of certain natural substacks.  
	
	Roughly speaking, $[\P(W_m) / \PGL_2]$ parametrizes degree $2m$ divisors on a smooth rational curve. Let $\mathcal{S} \subseteq k[X_1,X_2,X_3]_2 \smallsetminus \{ 0\}$ be the subset of smooth degree $2$ polynomials. A matrix $A \in \GL_3$ acts on a point of $\mathcal{S}$ by precomposition with $A^{-1}$ and we have isomorphisms $[\mathcal{S}/\GL_3]\cong B \PGL_2 \cong \mathfrak{M}_0$, where $\mathfrak{M}_0$ denotes the moduli stack of smooth genus $0$ curves. 
	
	For $m \in \Z_{\geq 1}$, let $V_m$ be the cokernel of the bundle map
	\begin{equation*}
		\begin{tikzcd}[row sep=tiny]
			\mathcal{S} \times k[X_1,X_2,X_3]_{m-2} \arrow[r] & \mathcal{S} \times k[X_1,X_2,X_3]_{m}\\
			(q,f) \arrow[r,mapsto] & (q,qf)
		\end{tikzcd}    
	\end{equation*}
	
	The action of $\GL_3$ on $\mathcal{S}$ extends to an action on the projective bundle $\P(V_m)$ by setting $A \cdot (q,[f])=(q \circ A^{-1}, [f \circ A^{-1}])$. Moreover, we have an isomorphism $[\P(V_m) /\GL_3] \cong [\P(W_m)/ \PGL_2]$ obtained by identifying$(q,[f]) \in \P(V_m)$ with the divisor given by the zero locus of $f$ in the conic defined by $q$ (see \cite[Proposition 3.2]{DL18}).
	
	\begin{notation}
		We denote by $\xi_{2m} \in \CH^*_{\GL_3}(\P(V_m))$ the $\GL_3$-equivariant first Chern class of $\cO_{\P(V_m)}(1)$. By an abuse of notation we will also denote with the same symbol its pullback to $\CH^*_{\Gm^3}(\P(V_m))$, where $\Gm^3 \subseteq \GL_3$ is the maximal torus consisting of diagonal matrices.
	\end{notation}
	
	For $r,\ell \in \Z_{\geq 0}$ such that $r+\ell \leq m$, we have inclusions
	\begin{equation}\label{eqn: map j-m,r,l}
		\begin{tikzcd}
			j_{m;r,\ell}: \P(V_{m-r-\ell}) \arrow[r,hookrightarrow] & \P(V_m)
		\end{tikzcd}
	\end{equation}
	given by $(q,[f]) \mapsto (q,[f \cdot X_0^r X_1^\ell])$. These are not $\GL_3$-equivariant, but only $\Gm^3$-equivariant. Thus we have $\Gm^3$-equivariant classes $[W_{m;r,\ell}] \in \CH^*_{\Gm^3}(\P(V_m))$ defined by the images of the maps $j_{m;r,\ell}$.
	
	In this paper we will only need the case $(r,l)=(1,0)$, that is, $j_{m;1,0}$ and $[W_{m;1,0}]$. This class has already been computed in~\cite{CLI}, and we report here the result for the convenience of the reader.
	
	\begin{lemma}[{\cite[Proposition 3.16, Remark 3.17]{CLI}}]\label{lem: computation Wm10}
		Let $m\geq2$. The $\Gm^3$-equivariant class of $[W_{m;1,0}]$ is
		\[
		[W_{m;1,0}]=\xi_{2m}^2+m^2c_2+(2m-1)\xi_{2m}t_1+m(2m-1)t_1^2.
		\]
		where the $t_i$ denote the first Chern classes of the standard $\Gm$-representations, each corresponding to the $i$-th factor in $\Gm^3$, and $c_2= t_1 t_2 + t_1 t_3 + t_2 t_3$.
	\end{lemma}
	In \S\ref{subsec: conclusion computation}, it will be useful to work with the restriction of the vector bundle $V_{m}\rightarrow\mathcal{S}$ over some open substacks of $\mathcal{S}$. We recall the definition of these restrictions.
	
	A quadric form $q \in \mathcal{S}$ can be written as $\sum_{\underline{i}} a_{\underline{i}} X_1^{i_1} X_2^{i_2} X_3^{i_3}$ where $\underline{i}=(i_1,i_2,i_3) \in \Z_{\geq 0}^3$ is such that $|\underline{i}|=i_1+i_2+i_3=2$. Let $\mathcal{S}_{\underline{i}} \subseteq \mathcal{S}$ be the open locus where $a_{\underline{i}} \neq 0$. The restriction projective bundles $\P(V_m)$ to each $\mathcal{S}_{\underline{i}}$ is trivial and 
	\begin{equation}\label{eqn: Chow P(Vm) restricted to Si}
		\CH^*_{\Gm^3}(\P(V_m)_{|\mathcal{S}_{\underline{i}}})=\frac{\CH^*_{\Gm^3}(\P(V_m))}{(i_1 t_1+ i_2 t_2 + i_3 t_3).}
	\end{equation}
	See \cite[Lemmas 4.1 and 4.5]{DL18}.
	
	Other than counterparts for the $\PGL_2$-schemes $\P(W_m)$, Di Lorenzo worked out the counterparts of some natural morphisms between product of those stacks, such as multiplication and squaring maps. See~\cite[Subsection 3.4]{DL18} and~\cite[Subsection 3.4]{CLI}. We will combine them to produce $\GL_3$-counterparts of some proper morphisms along which we need to compute the pushforward.
	
	We will also need to work with $\mu_2\times\GL_3$-counterparts of $\mu_2\times\PGL_2$-morphisms and schemes. For instance,
	\[
	\left[\frac{\P(W_m)\times\P(W_m)}{\mu_2\times\PGL_2}\right]\cong\left[\frac{\P(V_m)\times_{\mathcal{S}}\P(V_m)}{\mu_2\times\GL_3}\right],
	\]
	where in both cases $\mu_2$ acts by exchanging the two factors. This is a consequence of the next more general lemma.
	\begin{lemma}\label{lem: mu2xGL3 counterparts of symmetric products}
		Let $X$ be a scheme of finite type over $\mathrm{Spec} k$ endowed with a $\PGL_2$-action, and let $Y$ be a $\GL_3$-counterpart of it. Then, there exists a morphism $g:Y\rightarrow\mathcal{S}$ that is a $\GL_3$-counterpart to $X\rightarrow\mathrm{Spec} k$.
		
		Consider the $\mu_2\times\PGL_2$-action on $X\times X$ where $\mu_2$ exchanges the two factors and $\PGL_2$ acts diagonally. Then, $Y\times_{\mathcal{S}}Y$ endowed with the analogous $\mu_2\times\GL_3$-action is a $\mu_2\times\GL_3$-counterpart to $X\times X$. 
	\end{lemma}
	\begin{proof}
		The first part of the statement follows from the proof of~\cite[Theorem 1.4]{DL18}. For the second part, by construction we have cartesian cube
		\[
		\begin{tikzcd}[row sep=small, column sep=small]
			& Y\times_{\mathcal{S}}Y \arrow[rr] \arrow[dd] \arrow[dl] & & {\left[\frac{X\times X}{\PGL_2}\right]} \arrow[dl] \arrow[dd] \\
			Y \arrow[rr] \arrow[dd] & & {\left[\frac{X}{\PGL_2}\right]} \arrow[dd] \\
			& Y \arrow[rr] \arrow[dl] & & {\left[\frac{X}{\PGL_2}\right]} \arrow[dl] \\
			\mathcal{S} \arrow[rr] & & \mathfrak{M}_0
		\end{tikzcd}
		\]
		where the maps from the left to the right side of the cube are $\GL_3$-torsors. One concludes by noticing that the $\mu_2$-torsors $[Y\times_{\mathcal{S}}Y/\GL_3]\rightarrow[Y\times_{\mathcal{S}}Y/\mu_2\times\GL_3]$ and $[X\times X/\PGL_2]\rightarrow[X\times X/\mu_2\times\PGL_2]$ are compatible with the isomorphism $[Y\times_{\mathcal{S}}Y/\GL_3]\cong[X\times X/\PGL_2]$.
	\end{proof}
	One can prove a similar statement for products of morphism, disjoint unions and so forth. As we will not make extensive use of these techniques, we will limit ourselves in stating each time what is the $\mu_2\times\GL_3$-counterpart in the particular case we have at hand.
	
	The reason for passing to the $\GL_3$-counterparts is that we can then further restrict to the maximal torus $\Gm^3 \subseteq \GL_3$, whose representations are much simpler. 
	
	The following notation will be useful.

	\begin{notation}
		For an ideal $I$ in $\CH^*_{H \times \GL_3}(X)$, we will denote by $\widetilde{I}$ its extension to $\CH^*_{H \times \Gm^3}(X)$.
	\end{notation}
	
	\begin{remark}\label{rmk: extension commutes with maps}
		Recall that one can always recover $I$ from $\widetilde{I}$ using the identity
		$$
		\widetilde{I} \cap \CH^*_{H \times \GL_3}(X)=I
		$$
		which follows from~\cite[Lemma 2.1]{FuVi}. Furthermore, for a $H \times \GL_3$ proper map $f:X \to Y$, one has 
		$$
		\widetilde{\operatorname{im}(f_*)}= \operatorname{im}(f_*: \CH^*_{H \times \Gm^3}(X) \to \CH^*_{H \times \Gm^3}(Y)).
		$$
		This is an application of \cite[Proposition 3.4]{CLI}. 
	\end{remark}
	
	\subsection{Chow-K\"unneth Property}\label{subsec: Chow Kunneth Property}
	
	In this subsection we study the Chow K\"unneth property for quotient stacks, and prove that it is satisfied by $B\PGL_2$, $B\SL_n$ and $B\GL_n$ for every $n$. Some of the results are probably known to experts, but we report them here due to a lack of a reference.
	
	Every stack is assumed to be of finite type over a field $k$.
	We start by recalling some definitions.
	
	\begin{definition}
		Let $\cX$ be a quotient stack. For every quotient stack $\cT$, consider the Chow-Kunneth morphism
		\[
		\begin{tikzcd}
			\varphi_{\cT}:\CH_*(\cX) \otimes_{\Z} \CH_*(\cT)\arrow[r] & \CH_*(\cX\times\cT)
		\end{tikzcd}
		\]
		\begin{itemize}
			\item We say that $\cX$ satisfies the \emph{Chow-K\"unneth property}, CKP for short, if $\varphi_{\cT}$ is an isomorphism for every quotient stack $\cT$.
			\item We say that $\cX$ satisfies the \emph{Chow-K\"unneth generation property}, CKgP for short, if $\varphi_{\cT}$ is surjective for every quotient stack $\cT$.
		\end{itemize}
	\end{definition}
	The above definition was already given and studied in~\cite[Section 2.3]{Bae-Schmitt},~\cite[Section 3]{CL24} with rational coefficient.
	
	We start by stating two lemmas whose proof is similar to that of~\cite[Section 3]{CL24}.
	
	\begin{lemma}\label{lem: basic properties of CKgP}
		\begin{enumerate}
			\item Let $\cX_1,\ldots,\cX_n$ be quotient stacks satisfying CKgP. Then, the $n$-fold product $\cX_1\times\ldots\times\cX_n$ satisfies CKgP.
			\item Let $\cX$ be a quotient stack satisfying CKgP. Suppose that $f:\cY\rightarrow\cX$ is either
			\begin{itemize}
				\item[a)] a $\Gm$-torsor, or
				\item[b)] an affine bundle, or
				\item[c)] an open immersion, or
				\item[d)] a Grassmann bundle.
			\end{itemize}
			Then, $\cY$ satisfies CKgP. In the case b), also the converse holds.
			\item Suppose that $\cX=[X/G]$, $\cZ=[Z/G]$, and suppose given a $G$-equivariant Chow envelope $g:X\rightarrow Z$. If $\cX$ satisfies CKgP then also $\cZ$ does.
			\item If $\cX$ is a quotient stack with a finite stratification (see \cite[Page 27]{EH16} for the definition) whose strata are quotient stacks that satisfies CKgP, then $\cX$ satisfies CKgP.
		\end{enumerate}
	\end{lemma}
	\begin{proof}
		The proofs of points 1), 3), 4) and parts b), c) and d) of point 2) are the same as the proofs for~\cite[Lemmas 3.2, 3.3, 3.4, 3.5, 3.7]{CL24}. Part a) of point 2) follows from points b) and c) noticing that a $\Gm$-torsor is the complement of the 0-section of the associated line bundle $\cL$.
	\end{proof}
	
	\begin{lemma}\label{lem: basic properties of CKP}
		\begin{enumerate}
			\item Let $\cX_1,\ldots,\cX_n$ be quotient stacks satisfying CKP. Then, the $n$-fold product $\cX_1\times\ldots\times\cX_n$ satisfies CKP.
			\item Let $\cX$ be a quotient stack satisfying CKP. Suppose that $f:\cY\rightarrow\cX$ is either
			\begin{itemize}
				\item[a)] a $\Gm$-torsor, or
				\item[b)] an affine bundle, or
				\item[c)] an open immersion whose reduced closed complement satisfies CKgP, or
				\item[d)] a Grassmann bundle.
			\end{itemize}
			Then, $\cY$ satisfies CKP. In the case b), also the converse holds.
		\end{enumerate}
	\end{lemma}
	\begin{proof}
		Point 1) follows directly from the definitions. Part b) of $2$ follows because the pullback  $\CH_*(\cX) \to \CH_*(\cY)$ is an isomorphism. Part c) follows from the excision formula, part a) follows from part b) and c). Finally, part d) follows from the formula~\cite[Proposition 14.6.5]{Ful98} for Chow groups of Grassmann bundles.
	\end{proof}
	
	\begin{proposition}\label{prop: CKP for some classifying stacks}
		The classifying stacks
		\begin{enumerate}
			\item $B\Gm^n$ for every $n$,
			\item $B\GL_n$ for every $n$,
			\item $B\SL_n$ for every $n$,
			\item $B\PGL_2$, and
			\item $BG$, where $G\cong (\Gm\times\Gm)\rtimes\mu_2$
		\end{enumerate}
		satisfy the CKP, assuming that the characteristic of $k$ is different from $2$ in the last case.
	\end{proposition}
	\begin{proof}
		By induction on $n$, the case of $B\Gm^n$ reduces to that of $B\Gm$, which follows from~\cite[Lemma 10]{Oe18}. Briefly, $B\Gm$ can be approximated by projective spaces, in the sense that the composite
		\[
		\begin{tikzcd}
			\P^n\cong \left[\frac{\chi^{\oplus n+1}\smallsetminus 0}{\Gm}\right]\subset\left[\frac{\chi^{\oplus n+1}}{\Gm}\right]\arrow[r] & B\Gm
		\end{tikzcd}
		\]
		realizes $\P^n$ as an open subscheme of codimension $n$ of a vector bundle over $B\Gm$. Then, the statement follows from the fact that $\P^n$ satisfies CKP.
		Similarly, $B\GL_n$ satisfies CKP because it can be approximated by Grassmannians (in the sense above), that in turn satisfy CKP by part e) of point 2) of Lemma~\ref{lem: basic properties of CKP}.
		
		Since $B\SL_n\rightarrow B\GL_n$ is a $\Gm$-torsor, part a) of point 2) of Lemma~\ref{lem: basic properties of CKP} implies the result for $B\SL_n$ as well.
		The fact that $BG$ satisfies CKP follows from~\cite[Lemma 2.12]{Bur14}. We are left with the case of $B\PGL_2$.
		
		Denote by $V_{\GL_3}$ the standard representation of $\GL_3$ and recall that $B\PGL_2\cong [\mathcal{S}/\GL_3]$, where $\mathcal{S}$ is the open complement in $\Sym^2V_{\GL_3}^{\vee}$ of the closed subscheme $\Delta$ parametrizing singular conics. From points c) and d) of Lemma~\ref{lem: basic properties of CKgP} it follows that $B\PGL_2$ satisfies CKgP, and from point d) of Lemma~\ref{lem: basic properties of CKP} it is enough to show that $[\Delta/\GL_3]$ satisfies CKgP. Consider the action of $\Gm$ on $\Sym^{2}V_{\GL_3}^{\vee}$ given by multiplication, which commutes with the $\GL_3$-action, and let $\Delta\rightarrow\uD=[\Delta/\Gm]$ be the induced $\GL_3$-equivariant projection to the $\Gm$-quotient $\uD$. By part a) of point 2) of Lemma~\ref{lem: basic properties of CKgP} it is enough to prove that $[\uD/\GL_3]$ satisfies CKgP. Notice that there is a $\GL_3$-equivariant envelope consisting of the disjoint union of the two $\GL_3$-equivariant morphisms 
		\[
		\begin{tikzcd}
			\P(V_{\GL_3}^{\vee})\times\P(V_{\GL_3}^{\vee})\arrow[r] & \P(\Sym^2V_{\GL_3}^{\vee}), & (\ell_1,\ell_2)\mapsto \ell_1\ell_2
		\end{tikzcd}
		\]
		\[
		\begin{tikzcd}
			\P(V_{\GL_3}^{\vee})\arrow[r] & \P(\Sym^2V_{\GL_3}^{\vee}), & \ell\mapsto \ell^2.
		\end{tikzcd}
		\]
		By part b) of point 2) of Lemma~\ref{lem: basic properties of CKgP}, the sources of the two morphisms satisfy CKgP, hence also $[\uD/\GL_3]$ does by point 3) of the same lemma.
	\end{proof}
	
	It would be interesting to know if $B\PGL_n$ satisfies CKP for $n>2$. This is true with rational coefficients. Indeed, $B\SL_n\rightarrow B\PGL_n$ is a $\mu_n$-gerbe and the pullback along gerbes banded by finite groups induces isomorphisms between rational Chow groups (see ~\cite[Lemma 3.6]{CL24}).
	
	\section{Computation  $\CH^*([\cD_{g+1,g+1}/\mu_2])$ and $\RH_g^{(g+1)/2}$}
	
	As explained in the introduction, the presentation of $\RH_{g}^{\frac{g+1}{2}}$ given in Theorem~\ref{thm: presentation RHg,g+1/2} involves a group $H$ that is complicated, and the Chow ring of the corresponding classifying stack $BH$ is not known. To avoid working directly with $H$, we first compute $\CH^*\big([\cD_{g+1,g+1}/\mu_2]\big)$ using the presentation in Theorem~\ref{thm: presentation Dg+1g+1mu2}, and then apply \cite[Proposition~3.5]{CLI} to the $\mu_2$-root stack $\RH_{g}^{\frac{g+1}{2}} \to [\cD_{g+1,g+1}/\mu_2]$  described in Lemma~\ref{lemma: key cartesian diagram} to bootstrap the result to $\RH_g^{(g+1)/2}$.
	
	\subsection{Generators and first relations in $\CH^*([\cD_{g+1,g+1}/\mu_2])$}
	
	First, we pass to projective setting.
	By \cite[Proposition 4.5.3]{Rom22} (or the version of~\cite[Lemma 4.3]{AI17}) and Theorem \ref{thm: presentation Dg+1g+1mu2}, we have an isomorphism
	\begin{equation}
		[\cD_{g+1,g+1}/\mu_2]\cong \left[\frac{\P(W_{\frac{g+1}{2}})\times\P(W_{\frac{g+1}{2}})\smallsetminus\uD}{\mu_2\times\PGL_2}\right]
	\end{equation}
	
	Notice that the closed subscheme $Z_{g+1}:=0\times\mathbb{A}^{g+2}\sqcup\mathbb{A}^{g+2}\times0 \subseteq \mathbb{A}^{g+2}\times\mathbb{A}^{g+2} \cong V^\vee\otimes W_{\frac{g+1}{2}}$ is contained in $\Delta$. Let $j:Z_{g+1}\hookrightarrow V^{\vee}\otimes W_{\frac{g+1}{2}}$ be the inclusion. Recall that $G=\GI$, and observe that
	\[
	\left[\frac{Z_{g+1}}{G\times\PGL_2}\right]\cong\left[\frac{(\chi^{(1)})^{-1}\otimes W_{\frac{g+1}{2}}\times0}{\Gm^{(1)}\times\Gm^{(2)}\times\PGL_2}\right]\cong\left[\frac{0\times(\chi^{(2)})^{-1}\otimes W_{\frac{g+1}{2}}}{\Gm^{(1)}\times\Gm^{(2)}\times\PGL_2}\right].
	\]
	In particular, its Chow ring is isomorphic to 
	\begin{equation}\label{def of xi}
		\CH^*(B(\Gm\times\Gm\times\PGL_2))\cong\frac{\Z[x_1,x_2,c_2,c_3]}{(2c_3)},
	\end{equation}
	where $x_i \in \CH^*(B \Gm^{(i)})$ is the first Chern class of the standard representation. The $G\times\PGL_2$-quotient of $j$ factors as
	\begin{equation}\label{eqn: j factorization}
		\begin{tikzcd}
			j: \left[\frac{(\chi^{(1)})^{-1}\otimes W_{\frac{g+1}{2}}\times0}{\Gm\times\Gm\times\PGL_2}\right]\arrow[r,"\widetilde{j}"] & \left[\frac{(\chi^{(1)})^{-1}\otimes W_{\frac{g+1}{2}}\times(\chi^{(2)})^{-1}\otimes W_{\frac{g+1}{2}}}{\Gm\times\Gm\times\PGL_2}\right]\arrow[r,"\psi_{g+1}"] & \left[\frac{V^\vee \otimes W_{\frac{g+1}{2}}}{G\times \PGL_2 }  \right].
		\end{tikzcd}
	\end{equation}
	The last map is the one obtained from the degree 2 cover $B(\Gm\times\Gm\times\PGL_2)\rightarrow B(G\times\PGL_2)$ induced by the natural inclusion of groups.

	\begin{lemma}\label{lem: Chow alternative presentation Dg+1g+1mu2}
		We have
		\begin{align*}
			\CH^*\left(\left[\frac{\P(W_{\frac{g+1}{2}})\times\P(W_{\frac{g+1}{2}})}{\mu_2\times\PGL_2}\right]\right) & \cong\CH^*\left(\left[\frac{V^\vee \otimes W_{\frac{g+1}{2}} \smallsetminus Z_{g+1}}{G\times \PGL_2 }\right]\right)\\
			& \cong\frac{\Z[\beta_1,\beta_2,\gamma,c_2,c_3]}{(2\gamma,\gamma(\beta_1+\gamma),2c_3,c_{\mathrm{top}}(V^{\vee}\otimes W_{\frac{g+1}{2}}),j_*(1),j_*(x_1))}.
		\end{align*}
	\end{lemma}
	\begin{proof}
		Since $G$ (or $\PGL_2$) satisfies the K\"unneth formula by Proposition~\ref{prop: CKP for some classifying stacks}, by \cite[Theorem 5.2]{Lar19} we have
		\[
		\CH^*\left(\left[\frac{V^{\vee}\otimes W_{\frac{g+1}{2}}\smallsetminus0}{G\times\PGL_2}\right]\right)\cong\frac{\Z[\beta_1,\beta_2,\gamma,c_2,c_3]}{(2\gamma,\gamma(\beta_1+\gamma),2c_3,c_{\mathrm{top}}(V^{\vee}\otimes W_{\frac{g+1}{2}}))}.
		\] 
		
		Denote by $j^o$ the inclusion
		$
		j^o : Z_{g+1} \smallsetminus \{0\} \hookrightarrow V^{\vee} \otimes W_{\frac{g+1}{2}} \setminus \{0\}.
		$
		Our goal is to compute the pushforward $j^o_*$.
		Notice that
		\[
		\left[\frac{Z_{g+1} \smallsetminus 0}{G\times\PGL_2}\right]\cong\left[\frac{(\chi^{(1)})^{-1}\otimes W_{\frac{g+1}{2}}\smallsetminus0}{\Gm^{(1)}\times\Gm^{(2)}\times\PGL_2}\right],
		\]
		whose Chow ring is then a quotient of $\CH^*(B(\Gm\times\Gm\times\PGL_2))$. Since the pushforward $j_*^o$ is a homomorphism of $\CH^*(B(G \times \PGL_2))$-modules and $\CH^*(B(\Gm^{(1)}\times\Gm^{(2)}\times\PGL_2))$ is generated by $1$ and $x_1$ as a $\CH^*(B(G \times \PGL_2))$-module, it is enough to quotient by $j_*^o(1)$ and $j_*^o(x_1)$.
	\end{proof}
	
	The morphism $\psi_{g+1}$ in Equation \eqref{eqn: j factorization} and its variants will play an important role later, as it is often convenient to work $\PGL_2$-equivariantly rather than $\mu_2 \times \PGL_2$-equivariantly. Another incarnation of the morphism $\psi_{g+1}$ is
	\begin{equation}\label{eqn: def psi}
		\psi_{g+1}: \left[ \frac{\P(W_{\frac{g+1}{2}}) \times \P(W_{\frac{g+1}{2}}) }{\PGL_2} \right] \longrightarrow \left[ \frac{\P(W_{\frac{g+1}{2}}) \times \P(W_{\frac{g+1}{2}}) }{\mu_2 \times \PGL_2} \right]
	\end{equation}
	induced by the inclusion $\PGL_2 \hookrightarrow \mu_2 \times \PGL_2$.
	In general, we will use $\psi: B(\Gm^2 \times \PGL_2) \to B(G \times \times \PGL_2)$ for the morphism between classifying stacks, while we will decorate the symbol with subscripts and/or superscripts to denote morphism between various quotient stacks obtained from $\psi$ by some base change.
	
	We record here the pullbacks of the standard generators under $\psi$ and $\psi_{g+1}$, which are equal: 
	\begin{equation}\label{eqn: relation beta and x}
		\psi^*(\beta_1) = x_1 + x_2 \quad \text{and} \quad \psi^*(\beta_2) = x_1 x_2.
	\end{equation}
	Again, here $x_i$ is the first Chern class of the tautological line bundle associated with $\chi^{(i)}$ on $B\Gm^{(i)}$ for $i = 1, 2$. In terms of the projective spaces above, we have $x_i = -\xi_{g+1}^{(i)}$, where $\xi_{g+1}^{(i)}$ denotes the first Chern class of the line bundle $\cO(1)$ from the $i$-th factor, for $i=1,2$.
	
	Let $\mu: [ (\P(W_{\frac{g+1}{2}}) \times \P(W_{\frac{g+1}{2}}) / \PGL_2 ] \to [\P(W_{\frac{g+1}{2}}) \times \P(W_{\frac{g+1}{2}}) / \PGL_2 ]$ be the involution swapping the two factors.
	
	\begin{lemma}\label{lemma: image psi*}
		We have that
		$$
		\operatorname{im}(\psi_{g+1}^*)= \CH^*_{\mu_2 \times \PGL_2}( (\P(W_{\frac{g+1}{2}}) \times \P(W_{\frac{g+1}{2}}) )^{\mu^*} 
		$$
		is the $\mu^*$-invariant subring of $\CH^*_{\mu_2 \times \PGL_2}( (\P(W_{\frac{g+1}{2}}) \times \P(W_{\frac{g+1}{2}}) )$.
	\end{lemma} 
	\begin{proof}
		This follows from Equation \eqref{eqn: relation beta and x} and the fact that the classes $x_1, x_2, \gamma, c_2, c_3$ (respectively, $\beta_1, \beta_2, \gamma, c_2, c_3$) generate the Chow ring of the domain (respectively, codomain) of $\psi_{g+1}^*$ by \cite[Lemma 5.1]{CLI} and Lemma~\ref{lem: Chow alternative presentation Dg+1g+1mu2}.
	\end{proof}
	\begin{remark}
		Notice that the same proof shows that the image of the pullback along the morphism $\psi$ is
		$$
		\frac{\Z[\beta_1,\beta_2,c_2,c_3]}{(2c_3)}\subset\frac{\Z[x_1,x_2,c_2,c_3]}{(2c_3)}
		$$ which coincides with the image of $\CH^*(B(\GL_2\times\PGL_2))\rightarrow\CH^*(B(\Gm^2\times\PGL_2))$. Moreover, it is clear that the kernel of $\psi^*$ is the ideal generated by $\gamma$.
	\end{remark}
	
	We conclude the subsection with one last result that will be useful later in the paper.
	
	\begin{lemma}\label{lem: image pushforward psi}
		The image of
		\[
		\begin{tikzcd}
			\psi_*:\CH^*(B(\Gm^2\times\PGL_2))\arrow[r] & \CH^*(B(G\times\PGL_2))
		\end{tikzcd}
		\]
		is equal to the ideal $(2,\beta_1+\gamma)$. In particular, $j_*(1)\in(2,\beta_1+\gamma)$, where $j$ is the map in Equation \eqref{eqn: j factorization}.
	\end{lemma}
	\begin{proof}
		Notice that $\CH^*(B(\Gm^2\times\PGL_2))$ is generated by $1$ and $t_1$ as a $\CH^*(B(G\times\PGL_2))$-module via pullback along $\psi$. Then, by the push-pull formula, the image of $\psi_*$ is the ideal generated by $\psi_*(1)=2$ and $\psi_*(t_1)=\beta_1+\gamma$, by~\cite[Lemma 7.3]{Lar19}. For the last part of the statement, recall that in the proof of Lemma~\ref{lem: Chow alternative presentation Dg+1g+1mu2} we have also shown that $j_*(1)=\psi_{*}(\widetilde{j}(1))$, hence it is in the image of $\psi_*$.
	\end{proof}
	\subsection{A Chow envelope for $\underline{\Delta}$} \label{sec: envelope}
	
	By Lemma \ref{lem: Chow alternative presentation Dg+1g+1mu2} and the excision sequence
	
	$$
	\CH^*_{\mu_2\times\PGL_2}(\uD) \to \CH^*_{\mu_2\times\PGL_2}(\P(W_{\frac{g+1}{2}})\times\P(W_{\frac{g+1}{2}})) \to \CH^*_{\mu_2\times\PGL_2}([\P(W_{\frac{g+1}{2}})\times\P(W_{\frac{g+1}{2}})\smallsetminus\uD) \to 0
	$$
	the problem of computing $\CH^*_{\mu_2\times\PGL_2}(\P(W_{\frac{g+1}{2}})\times\P(W_{\frac{g+1}{2}})\smallsetminus\uD)$ reduces to computing the $\mu_2 \times \PGL_2$-equivariant pushforward of
	\[
	\begin{tikzcd}
		i : \uD \arrow[r,hookrightarrow] & \P\left(W_{\frac{g+1}{2}}\right) \times \P\left(W_{\frac{g+1}{2}}\right).
	\end{tikzcd}
	\]
	
	We begin by observing that $[\uD/ \mu_2 \times \PGL_2]$ is reducible, with a decomposition $[\uD/ \mu_2 \times \PGL_2] = [\uD_1 /\mu_2 \times \PGL_2] \cup [\uD_2/\mu_2 \times \PGL_2]$, where $\uD_1$ parametrizes pairs of polynomials $(f, g)$ admitting a non-invertible common factor, and $\uD_2$ parametrizes pairs for which either $f$ or $g$ is singular. Let
	\[
	\begin{tikzcd}
		i_j : \uD_j \arrow[r,hookrightarrow] & \P\left(W_{\frac{g+1}{2}}\right) \times \P\left(W_{\frac{g+1}{2}}\right), \quad j = 1, 2,
	\end{tikzcd}
	\]
	denote the natural inclusion maps. It follows that
	$
	\mathrm{Im}(i_*) = \mathrm{Im}((i_1)_*) + \mathrm{Im}((i_2)_*),
	$
	so it is enough to compute the pushforwards $(i_1)_*$ and $(i_2)_*$ separately.
	
	To do this, we employ two equivariant Chow envelopes. See~\cite[Page 603 and Lemma 3]{EG98} or \cite[Definition 27 and Lemma 28]{CIL24} for the definition of Chow envelope; for our purposes, it is important that the induced pushforwards are surjections in Chow.
	
	Define $\mu_2\times\PGL_2$-equivariant morphisms
	\[
	\begin{tikzcd}[row sep=tiny]
		M_r:\P^r\times\P^{g+1-r}\times\P^{g+1-r}\arrow[r] & \P(W_{\frac{g+1}{2}})\times\P(W_{\frac{g+1}{2}}),\\
		(h,f,g)\arrow[r] & (hf,hg),
	\end{tikzcd}
	\]
	
	\[
	\begin{tikzcd}[row sep=tiny]
		S_r:\P^r\times((\P^{g+1-2r}\times \P^{g+1})\sqcup(\P^{g+1}\times \P^{g+1-2r}))\arrow[r] & \P(W_{\frac{g+1}{2}})\times\P(W_{\frac{g+1}{2}}),\\
		(h,f,g)\arrow[r,mapsto] & (h^2f,g)\text{ or }(f,h^2g),
	\end{tikzcd}
	\]
	\[
	\begin{tikzcd}[row sep=tiny]
		H_r:\P^r\times\P^r\times\P^{g+1-2r}\times \P^{g+1-2r}\arrow[r] & \P(W_{\frac{g+1}{2}})\times\P(W_{\frac{g+1}{2}}),\\
		(h,l,f,g)\arrow[r,mapsto] & (h^2f,l^2g),
	\end{tikzcd}
	\]
	where $\mu_2$ acts on the domain of $S_r$ by exchanging the second and third components, and it acts on the domain of $H_r$ by simultaneously swapping the first with the second and the third with the fourth entries.
	
	\begin{lemma}\label{lem: equivariant Chow envelope g+1 case}
		The disjoint union of the maps $M_r$ for $r=1,\ldots, g+1$ form an equivariant Chow envelope of $\uD_1$. Similarly, the disjoint union of $S_r$ and $H_r$ with $r=1,\ldots, (g+1)/2$ form an equivariant Chow envelope of $\uD_2$.
	\end{lemma}
	
	\begin{proof} 
		The proof for the maps $M_r$ is analogous to~\cite[Lemma 5.2]{CLI}. By the same lemma, the morphism $S_r$ form a (non-equivariant) envelope. Let $k\subset K$ be any field extension and $p=(f(x_0,x_1),g(x_0,x_1))$ a $\mu_2\times\PGL_2$-invariant $K$-valued point of $\uD_2$. We aim to show that there exists a $\mu_2 \times \PGL_2$-equivariant $ K $-valued point in the domain of some $ S_r $ or $ H_r $ mapping to $ p $. In fact, we will show that such a point exists in the domain of some $ H_r $. The $\mu_2$-invariance implies that there exists a degree $r$ polynomial $h_1 \in K[x_0,x_1]$ such that $h_1^2|f$ if and only if there exists $h_2 \in K[x_0,x_1]$ of the same degree such that $h_2^2|g$. If $r$ is taken to be as big as possible, then since $\mathrm{char}(k)$ is bigger than the degree of $f$ and $g$, we have $r>0$. Moreover, there is a unique $(h_1,h_2,f_1,f_2)\in(\P^r(K))^2\times(\P^{g+1-2r}(K))^2$ such that $H_r(h_1,h_2,f_1,f_2)=(f,g)$. Indeed, by the uniqueness of such a point and the $\mu_2 \times \PGL_2$-invariance of $(f,g)$, we get that also $(h_1,h_2,f_1,f_2)$ is $\mu_2\times\PGL_2$-invariant.
	\end{proof}
	The images of $S_{r*}$, $M_{r*}$ and $H_{r*}$ are computed in \S\ref{subsec: computation Sr},~\S\ref{subsec: computation Mr} and~\S\ref{subsec: computation Hr}, respectively.
	
	\begin{remark}\label{rmk: GL3 counterparts of Chow envelope}
		When $r$ is even, we can construct $\mu_2\times\GL_3$-counterparts to the morphisms that form the Chow envelope in Lemma~\ref{lem: equivariant Chow envelope g+1 case}, see \S\ref{subsec: GL3 counterparts} for basics on $\GL_3$-counterparts. Recall that, from Definition \ref{def: GL3 counterparts} we denote them by the same letter but adding an apostrophe. As an example, for every $r\geq1$ the $\GL_3$-counterpart of $M_{2r}$ is
		\[
		\begin{tikzcd}[column sep=1.7em]
			M_{2r}':\P(V_r)\times_{\mathcal{S}}\P(V_{\frac{g+1}{2}-r})\times_{\mathcal{S}}\P(V_{\frac{g+1}{2}-r})\arrow[r] & \P(V_{\frac{g+1}{2}})\times_{\mathcal{S}}\P(V_{\frac{g+1}{2}}), & (q,h,f,g)\mapsto(q,hf,hg).
		\end{tikzcd}
		\]
	\end{remark}
	
	\subsection{Strategy for the computation  $\CH^*([\cD_{g+1,g+1}/\mu_2])$ }\label{subsec: strategy}

	By Lemmas \ref{lem: Chow alternative presentation Dg+1g+1mu2} and \ref{lem: equivariant Chow envelope g+1 case}, the Chow ring $\CH^*([\cD_{g+1,g+1}/\mu_2])$ is a quotient of
	\[
	\CH^*\big(B(G\times \PGL_2)\big) \cong \frac{\Z[\beta_1,\beta_2,\gamma,c_2,c_3]}{(2 \gamma, \gamma(\beta_1 + \gamma), 2 c_3)}
	\]
	where $G=(\Gm \times \Gm) \rtimes \mu_2$, by the ideal generated by:
	\begin{itemize}
		\item[(a)] the classes $c_{\mathrm{top}}(V^{\vee} \otimes W_{\frac{g+1}{2}})$, $j_*(1)$, and $j_*(x_1)$;
		\item[(b)] (lifts of) the ideal generated by the images $\operatorname{im}(S_r)_*$ for $r = 1, \ldots, (g+1)/2$;
		\item[(c)] (lifts of) the ideal generated by the images $\operatorname{im}(H_r)_*$ for $r = 1, \ldots, (g+1)/2$;
		\item[(d)] (lifts of) the ideal generated by the images $\operatorname{im}(M_r)_*$ for $r = 1, \ldots, g+1$.
	\end{itemize}
	
	We will show that this ideal is in fact generated by the classes in (b), together with $\operatorname{im}(M_1)_*$ and $\operatorname{im}(M_2)_*$. 
	
	\begin{notation}
		We denote by $I_S$ the ideal generated by the classes in (b), by $I_H$ the ideal generated by the classes in $(c)$ and by $I_M$ the ideal generated by the classes in $(d)$. 
		
		Finally, we denote by $I_M^{\leq 2}$ the ideal generated by $\operatorname{im}(M_1)_*$, and $\operatorname{im}(M_2)_*$.
	\end{notation}
	
	The computation of $I_S$ is carried out in \S\ref{subsec: computation Sr}. In \S\ref{subsec: computation Mr}, we explicitly compute $\operatorname{im}(M_1)_*$ and $\operatorname{im}(M_2)_*$, and show that $I_M \subseteq I_S + I_{M}^{\leq 2} + I_H$. Finally in \S\ref{subsec: computation Hr} we show that $I_H \subseteq I_S + I_M^{\leq 2}$,
	and in \S\ref{subsec: conclusion computation} we show that the classes in (a) are also in $I_S + I_M^{\leq 2}$.

	\subsection{Computation of the ideal generated by the maps $S_{r*}$}\label{subsec: computation Sr}
	
	To compute the pushforward along $S_{r}$, we write $S_r$ as a transfer of a $\PGL_2$-equivariant map. This $\PGL_2$-equivariant map can be chosen between two maps $F_r$ and $G_r$, whose pushforward are computed in~\cite[Section 5.1]{CLI}. We now introduce these maps and explain their relation with the map $S_r$.
	
	For $r=1,\ldots, (g+1)/2$, we have a commutative diagram 
	
	\begin{equation}\label{eqn: transfer Fr}
		\begin{tikzcd}
			\left[\frac{\P^r\times\P^{g+1-2r}\times\P^{g+1}}{ \PGL_2}\right] \arrow[r,"F_r"] \arrow[d,"\phi_r","\cong"'] & \left[\frac{\P^{g+1}\times\P^{g+1}}{\PGL_2}\right]\arrow[d,"\psi_{g+1}"]\\
			\left[\frac{\P^r\times((\P^{g+1-2r}\times \P^{g+1})\sqcup(\P^{g+1}\times \P^{g+1-2r}))}{\mu_2 \times \PGL_2}\right]  \arrow[r,"S_r"] & \left[\frac{\P^{g+1}\times\P^{g+1}}{\mu_2 \times \PGL_2}\right]
		\end{tikzcd}
	\end{equation}
	
	where the isomorphism $\phi_r$ on the left is induced by the inclusion $\P^r\times\P^{g+1-2r}\times\P^{g+1} \hookrightarrow \P^r\times((\P^{g+1-2r}\times \P^{g+1})\sqcup(\P^{g+1}\times \P^{g+1-2r}))$ of the first summand, the map $\psi_{g+1}$ is the same as in \eqref{eqn: def psi} and the map $F_r$ is defined by $(h,f,g) \mapsto (h^2f,g)$.
	
	The map
	\begin{equation}\label{eqn: def Gr}
		\begin{tikzcd}
			G_r : \left[\frac{\P^r \times \P^{g+1} \times \P^{g+1 - 2r}}{\PGL_2}\right] \arrow[r] &
			\left[\frac{\P^{g+1} \times \P^{g+1}}{\PGL_2}\right], & (h, f, g) \mapsto (f, h^2 g)
		\end{tikzcd}
	\end{equation}
	is defined similarly, and there is an analogous commutative diagram if we replace $\phi_r$ with the inclusion into the second summand. However, we will not need this and will only use the maps $G_r$ in the next sections.

	It will be useful to fix the following notation.

	\begin{notation}\label{not: notation projective space}
		For $r \geq 1$, we denote by $\P^{2r}$ the space $\P(W_r)$ (endowed with the induced $\PGL_2$-action), and define 
		$
		\xi_{2r} := c_1^{\PGL_2}(\cO_{\P^{2r}}(1)) \in \CH^*_{\PGL_2}(\P^{2r})
		$
		to be the first Chern class of the $\PGL_2$-linearized line bundle $\cO_{\P(W_r)}(1)$ (and its pullbacks to other spaces).
		
		When $r$ is odd, we can still consider $\P^r$ as the space of degree $r$ forms in two variables, up to scalar, with the $\PGL_2$-action defined by $[B] \cdot [f] = [f \circ B^{-1}]$. However, this is not the projectivization of any $\PGL_2$-representation.
		
	\end{notation}
	
	Note that under the $\GL_3$-counterpart isomorphism $[\P(W_r)/\PGL_2] \cong [\P(V_r)/\GL_3]$, the classes $\xi_{2r}$ correspond to each other. Recall also that the $\PGL_2$-equivariant Chow ring of $\P^1$ is well-known \cite[Lemma 5.1]{GV08}:
	$$
	\CH^*_{\PGL_2}(\P^1)= \frac{\Z[c_2,c_3, \tau, ]}{(c_3, \tau^2+c_2)} \cong \Z[\tau]
	$$
	where $\tau=c_1^{\PGL_2}(\cO(2))$.
	
	\begin{lemma}\label{lemma: IS}
		The ideal generated by all $\operatorname{im}(S_r)_*$ for $r = 1, \ldots, (g+1)/2$ is the ideal generated by the following classes:
		
		\begin{align*}
			S_{1*}\phi_{1*}(1)&=-2g\beta_1,\\
			S_{1*}\phi_{1*}(\tau)&=2\beta_1^2-4\beta_2-(g^2-1)c_2,\\
			S_{2*}\phi_{2*}(\xi_2^2)&=\beta_1^4+\beta_1^3\gamma-4\beta_1^2\beta_2+2\beta_2^2+\left(\frac{g^2+1}{2}\right)\beta_1^2c_2+\beta_1\gamma c_2\\
			&-(g^2+1)\beta_2c_2+c_3(\beta_1+\gamma)+\frac{(g+1)^2(g-1)^2}{8}c_2^2,
		\end{align*}
		and
		\begin{align*}
			S_{1*}\phi_{1*}(\xi_{g+1})&=4g\beta_2,\\
			S_{1*}\phi_{1*}(\tau\cdot\xi_{g+1})&=-2\beta_1\beta_2+\frac{g^2-1}{2}\beta_1c_2,\\
			S_{2*}\phi_{2*}(\xi_2^2\cdot\xi_{g+1})&=\beta_1^3\beta_2+\beta_1^2\beta_2\gamma-3\beta_1\beta_2^2-\beta_2^2\gamma+\left(\frac{g^2+1}{2}\right)\beta_1\beta_2c_2\\
			&+\beta_2\gamma c_2+\frac{(g+1)^2(g-1)^2}{16}\beta_1c_2^2,
		\end{align*}
	\end{lemma}
	\begin{proof}
		First, we compute the classes in the statement, using the equality $S_r\circ\phi_F=\psi_{g+1}\circ F_r$. The computation of the pushforward along $F_1$ and $F_2$ has already been computed in~\cite[Lemma 5.3]{CLI}. Then, the result follows from the formulas of~\cite[Lemma 7.3]{Lar19} for the computation of $\psi_*$. As an example,
		\[
		S_{1*}\phi_{1*}(1)=\psi_{g+1*}F_{1*}(1)=\psi_{g+1*}(2g\xi_{g+1}^{(1)})=\psi_*(-2gx_1)=-2g\beta_1,
		\]
		and
		\[
		S_{1*}\phi_{1*}(\xi_{g+1})=\psi_{g+1*}F_{1*}(\xi_{g+1})=\psi_{g+1*}(2g\xi_{g+1}^{(1)}\xi_{g+1}^{(2)})=\psi_*(2gx_1x_2)=4g\beta_2.
		\]
		
		Next, we explain why these classes generate the entire ideal. By \cite[Lemma 5.3]{CLI}, the ideal generated by the images of all the maps $F_{r*}$ is generated by $F_{1*}(1)$, $F_{1*}(\tau)$, and $F_{2*}(\xi_2^2)$.
		
		The map $\psi_{g+1}$ is a homomorphisms of $\CH^*_{\mu_2 \times \PGL_2}((\P^{g+1})^2)$-modules, and, by Equation \eqref{eqn: relation beta and x}, $\CH^*_{\PGL_2}((\P^{g+1})^2)$ is generated by $1$ and $-x_2=\xi_{g+1}^{(2)}$ as a $\CH^*_{\mu_2 \times \PGL_2}((\P^{g+1})^2)$-module.
		
		The conclusion follows from the observation that for all $r$, we have $F_r^*(\xi_{g+1}^{(2)}) = \xi_{g+1}$.
	\end{proof}

	\subsection{Computation of the ideal generated by the maps $M_{r*}$}\label{subsec: computation Mr}
	
	Let $I_M$ be the ideal generated by the images of the maps $M_{r*}$ for $r=1,\ldots,g+1$. In this section we compute generators for $I_M + I_S$.
	
	\subsubsection{Computation of the pushforward of low degree classes under $M_{1*}$ and $M_{2*}$}\label{sec: low degree classes unde M1 and M2}
	
	We start by computing the pushforward along $M_2$ of some classes in low degree.
	
	\begin{lemma}\label{lem: M2 computation low degree classes}
		The following equations hold:
		\begin{align*}
			M_{2*}(1)&=-g\beta_2+g\left(\frac{g+1}{2}\right)\beta_1^2+g\left(\frac{g+1}{2}\right)^2c_2 \\
			M_{2*}(\xi_2)&=-g\beta_1\beta_2+g\left(\frac{g^2-1}{4}\right)\beta_1c_2+\left(\frac{g+1}{2}\right)^2c_3 \\
			M_{2*}(\xi_2^2)&=\beta_2^2-\left(\frac{g^2-1}{4}\right)\beta_1^2c_2+\left(\frac{g^2-1}{2}\right)\beta_2c_2+\left(\frac{g+1}{2}\right)^2\left(\frac{g-1}{2}\right)^2c_2^2+\left(\frac{g+1}{2}\right)\beta_1c_3.
		\end{align*}
	\end{lemma}
	
	Before proceeding with the proof of the lemma, we first make some preliminary observations. For $r=1,\ldots,(g+1)/2$ we have a commutative diagram 
	\begin{equation}\label{eq: huge commutative diag}
		\begin{tikzcd}
			\left[\frac{\P^{2r}\times(\P^{g+1-2r}\times\P^{g+1-2r})}{\PGL_2}\right]\arrow[d,"\psi_{2r}^M"]\arrow[r,"M_{2r}"]\arrow[ddddd,bend right=70,leftrightarrow,"="'] & \left[\frac{\P^{g+1}\times\P^{g+1}}{\PGL_2}\right]\arrow[d,"\psi_{g+1}"]\arrow[ddddd,bend left=70,leftrightarrow,"="]\\        
			\left[\frac{\P^{2r}\times(\P^{g+1-2r}\times\P^{g+1-2r})}{\mu_2\times\PGL_2}\right]\arrow[d,"\cong"]\arrow[r,"M_{2r}"] & \left[\frac{\P^{g+1}\times\P^{g+1}}{\mu_2\times\PGL_2}\right]\arrow[d,"\cong"]\\        
			\left[\frac{\P^{2r}\times(W_{\frac{g+1}{2}-r}\otimes V^\vee\setminus Z_{g+1-2r})}{G \times\PGL_2}\right]\arrow[d,"d"]\arrow[r,"M_{2r}"] & \left[\frac{W_{\frac{g+1}{2}}\otimes V^{\vee}\setminus Z_{g+
					1}}{G \times\PGL_2}\right]\arrow[d,"\delta"]\\
			\left[\frac{\P^{2r}\times(W_{\frac{g+1}{2}}\otimes V^\vee\setminus Z_{g+1-2r})}{\GL_2\times\PGL_2}\right]\arrow[r,"M_{2r}"] & \left[\frac{W_{\frac{g+1}{2}}\otimes V^{\vee}\setminus Z_{g+1}}{\GL_2\times\PGL_2}\right]\\        
			\left[\frac{\P^{2r}\times(W_{\frac{g+1}{2}-r}\otimes V^\vee\setminus Z_{g+1-2r})}{(\Gm \times \Gm)\times\PGL_2}\right]\arrow[u,"u_1"']\arrow[r,"M_{2r}"] & \left[\frac{(W_{\frac{g+1}{2}}\otimes V^\vee\setminus Z_{g+1})}{(\Gm \times \Gm)\times\PGL_2}\right]\arrow[u,"v_1"']\\
			\left[\frac{\P^{2r}\times(\P^{g+1-2r}\times\P^{g+1-2r})}{\PGL_2}\right]\arrow[u,"\cong","u_2"']\arrow[r,"M_{2r}"] & \left[\frac{\P^{g+1}\times\P^{g+1}}{\PGL_2}\right]\arrow[u,"\cong","v_2"']
		\end{tikzcd}
	\end{equation}
	
	where we are denoting by $V$ both the standard $\GL_2$ representation and its restriction to $G \subseteq \GL_2$.

	\begin{lemma}\label{lem: reduction to PGL2}
		Let 
		$$
		\alpha \in \CH^*_{\GL_2 \times \PGL_2}(\P^{2r}\times(W_{\frac{g+1}{2}-r}\otimes V^{\vee}\setminus Z_{g+1-2r})) 
		$$
		and 
		$$
		\beta \in \CH^*_{\GL_2 \times \PGL_2}(W_{\frac{g+1}{2}}\otimes V^{\vee} \smallsetminus Z_{g+1})
		$$
		Then, the following are equivalent:
		\begin{itemize}
			\item[(i)] $M_{2r*}(\alpha)=\beta$ ;
			\item[(ii)] $M_{2r*}d^*(\alpha)=\delta^*(\beta)$;
			\item[(iii)] $M_{2r*}(\psi_{2r}^M)^*d^*(\alpha)=\psi_{g+1}^*\delta^*(\beta)$.
		\end{itemize}
	\end{lemma}
	\begin{proof}
		The implications $(i) \implies (ii) $ and $(ii) \implies (iii)$ follow from the push-pull formula. Suppose that condition $(iii)$ holds. The map $v_2^*v_1^*$ is injective by~\cite[Proposition 3.6]{EG98} and \cite[page 2]{EF09}.
		Therefore, it suffices to show that $(i)$ holds after pulling back via $v_1 \circ v_2$. We compute:
		\[
		v_2^* v_1^* M_{2r*}(\alpha) = M_{2r*} u_2^* u_1^*(\alpha) = M_{2r*} (\psi_{2r}^M)^* d^*(\alpha) = \psi_{g+1}^* \delta^*(\beta) = v_2^* v_1^* \beta.
		\]
		This concludes the proof.
	\end{proof}

	\begin{proof}[Proof of Lemma \ref{lem: M2 computation low degree classes}]
		By Lemma \ref{lem: reduction to PGL2}, it suffices to compute the listed pushforward $\PGL_2$-equivariantly. This is carried out in \cite[Corollary 5.9]{CLI}, and we obtain the stated formula after making the substitutions $\beta_1 = -\xi_{2a} - \xi_{2b}$, $\beta_2 = \xi_{2a} \xi_{2b}$, and $a = b = (g+1)/2$ in that corollary.
	\end{proof}
	
	For $M_1$ we have the following result.
	
	\begin{lemma}\label{lem: M1 computation low degree classes}
		The following equations hold:
		\begin{align*}
			M_{1*}(1)&=-(g+1)\beta_1,\\
			M_{1*}(\tau)&=2\beta_2-\frac{(g+1)^2}{2}c_2.
		\end{align*}
	\end{lemma}
	
	\begin{proof}
		We would like to use the same technique as for Lemma~\ref{lem: M2 computation low degree classes}; the only issue to implement it is that $\P^g$ is not the projectivization of any $\PGL_2$-representation, as $g$ is odd. To remedy this, notice that the pullback morphism
		\[
		\begin{tikzcd}
			\CH^{\leq2}\left(\left[\frac{\P^{g+1}\times\P^{g+1}}{\mu_2\times\PGL_2}\right]\right)\arrow[r] & \CH^{\leq2}\left(\left[\frac{\P^{g+1}\times\P^{g+1}}{\mu_2\times\SL_2}\right]\right)
		\end{tikzcd}
		\]
		is injective in degree $\leq 2$ as, in these degrees, this is the same as the pullback of $B(G \times \PGL_2) \to B(G \times \SL_2)$ which is visibly injective in this range. In particular, $c_2\in\CH^*(B\PGL_2)$ is mapped to the generator $c_2$ of $\CH^*(B \SL_2)= \Z[c_2]$. Since $M_{1*}(1)$ and $M_{1*}(\tau)$ have degree 1 and 2 respectively, we can compute their pushforward $\mu_2\times\SL_2$-equivariantly. Then $\P^g$ is the projectivization of a $\SL_2$-representation, allowing to construct a diagram analogous to~\eqref{eq: huge commutative diag}. As for the case of $M_{2*}$, this shows that the $\mu_2\times\PGL_2$-equivariant pushforward of the classes in the statement can be recovered from their $\PGL_2$-equivariant pushforward, which has already been computed in~\cite[Lemma 5.4]{CLI}. This yields the result.
	\end{proof}
	\subsubsection{Computation of the ideal generated by the maps $M_{r*}$ after inverting $2$}
	
	Inverting $2$, the computation is easier.
	
	\begin{lemma}\label{lem: Mr inverting 2}
		Working with $\Z[1/2]$-coefficients, for $r \geq 2$, we have 
		$$
		\operatorname{im}(M_{r*}) \subseteq I_S  + \operatorname{im}(M_{1*}).
		$$
		Moreover, still with $\Z[1/2]$-coefficients, $\operatorname{im}(M_{1*})$ is generated by $M_{1*}(1)$ and $M_{1*}(\tau)$.
	\end{lemma}
	
	\begin{proof}
		Since we are working with $\Z[1/2]$-coefficients, the pushforward along the projection
		\[
		\begin{tikzcd}
			\psi_r^M:{\left[\frac{\P^r\times\P^{g+1-r}\times\P^{g+1-r}}{\PGL_2}\right]}\arrow[r] & {\left[\frac{\P^r\times\P^{g+1-r}\times\P^{g+1-r}}{\mu_2\times\PGL_2}\right]}
		\end{tikzcd}
		\]
		is surjective. Moreover, we have a commutative diagram
		\[
		\begin{tikzcd}
			{\left[\frac{\P^r\times\P^{g+1-r}\times\P^{g+1-r}}{\mu_2\times\PGL_2}\right]}\arrow[rr,"M_r"] & & {\left[\frac{\P^{g+1}\times\P^{g+1}}{\mu_2\times\PGL_2}\right]}\\
			{\left[\frac{\P^r\times\P^{g+1-r}\times\P^{g+1-r}}{\PGL_2}\right]}\arrow[u,"\psi_r^M"]\arrow[rr,"M_r^{\PGL_2}"] & &{\left[\frac{\P^{g+1}\times\P^{g+1}}{\PGL_2}\right]}\arrow[u,"\psi_{g+1}"].
		\end{tikzcd}
		\]
		By~\cite[Proposition 5.5]{CLI}, we know that the image of $M_{r*}^{\PGL_2}$ is contained in $\mathrm{im}(M_{1*}^{\PGL_2})+\mathrm{im}(F_{1*})+\mathrm{im}(G_{1*})$; see $\S$\ref{subsec: computation Sr} for the definition of the morphisms $F_r$ and $G_r$. The diagram above shows that $\mathrm{im}(\psi_{g+1*}\circ M_{1*}^{\PGL_2})=\mathrm{im}(M_{1*})$. Moreover, diagram~\eqref{eqn: transfer Fr} yields $\psi_{g+1}\circ F_1=S_1\circ\phi_1$, hence $\mathrm{im}(\psi_{g+1*}\circ F_{1*})=\mathrm{im}(S_{1*})$. The same holds for $G_{1*}$, thus concluding the first part.
		
		For the last statement, by~\cite[Proposition 5.5]{CLI}, we have that $\mathrm{im}(M_{1*}^{\PGL_2})$ is generated by $M_{1*}^{\PGL_2}(1)$ and $M_{1*}^{\PGL_2}(\tau)$. Also, by Lemma~\ref{lem: M1 computation low degree classes} and~\cite[Lemma 5.4]{CLI}, these classes coincide with the pullbacks $\psi_{g+1}^*(M_{1*}(1))$ and $\psi_{g+1}^*(M_{1*}(\tau))$ respectively. Therefore,
		\[
		\psi_{g+1*}(M_{1*}^{\PGL_2}(1),M_{1*}^{\PGL_2}(\tau)) \subseteq(M_{1*}(1),M_{1*}(\tau))
		\]
		by the push-pull formula.
	\end{proof}
	\subsubsection{Computation of $M_{1*}$ and $M_{2*}$}
	
	We are ready to compute the ideal generated by the images of $M_{1*}$ and $M_{2*}$ modulo $\mathrm{im}(S_{1*})$. This is enough for our objective.
	
	\begin{lemma}\label{lem: ideal M1 + M2}
		The ideal $\mathrm{im}(M_{1*})+\mathrm{im}(M_{2*})+\mathrm{im}(S_{1*})$ is generated by $\mathrm{im}(S_{1*})$ and the classes computed in Lemmas~\ref{lem: M2 computation low degree classes} and~\ref{lem: M1 computation low degree classes}.
	\end{lemma}
	\begin{proof}
		We start with computing $\mathrm{Im}(M_{2*})$. Notice that the source of $M_2$ is
		\[
		\left[\frac{\P^2\times\P^{g-1}\times\P^{g-1}}{\mu_2\times\PGL_2}\right]\cong\left[\frac{\P(W_1)}{\PGL_2}\right]\times_{B\PGL_2}\left[\frac{(V^{\vee}\otimes W_{\frac{g-1}{2}})\setminus Z_{g-1}}{G\times\PGL_2}\right],
		\]
		and this is a projective bundle over $[(V^{\vee}\otimes W_{\frac{g-1}{2}}\setminus Z_{g-1})/G\times\PGL_2]$. Therefore, its Chow ring is generated by $1$, $\xi_2$ and $\xi_2^2$ as a $\CH^*(B(G\times\PGL_2))$-modulo. Moreover, the target of $M_2$ is isomorphic to $[(V^{\vee}\otimes W_{\frac{g+1}{2}}\setminus Z_{g+1})/G\times\PGL_2]$, and we have a commutative diagram
		\[
		\begin{tikzcd}
			{\left[\frac{\P(W_1)}{\PGL_2}\right]\times_{B\PGL_2}\left[\frac{(V^{\vee}\otimes W_{\frac{g-1}{2}})\setminus Z_{g-1}}{G\times\PGL_2}\right]}\arrow[r,"M_2"]\arrow[d,"u"] & {\left[\frac{(V^{\vee}\otimes W_{\frac{g+1}{2}})\setminus Z_{g+1}}{G\times\PGL_2}\right]}\arrow[d]\\
			B(\Gm\times G\times\PGL_2)\arrow[r,"v"] & B(G\times\PGL_2)
		\end{tikzcd}
		\]
		where $u$ is induced by the isomorphism $\P(W_1)\cong (\chi^{-1}\otimes W_1\setminus0)/\Gm$ and $v$ comes from the homomorphism $\Gm\times G\times\PGL_2\rightarrow G\times\PGL_2$ given by $(\lambda,A,[B])\mapsto(\lambda\cdot A,[B])$. Together with the push-pull formula, this shows that $\mathrm{im}(M_{2*})$ is generated by the classes in Lemma~\ref{lem: M2 computation low degree classes}. 
		
		Now, we pass to $M_1$, and we show that $\mathrm{im}(M_{1*})$ is contained in $J:=(M_{1*}(1),M_{1*}(\tau))+\mathrm{im}(M_{2*})+\mathrm{im}(S_{1*})$. Thanks to Lemma~\ref{lem: Mr inverting 2}, we know that this is true tensor $\Z[1/2]$, thus it suffice to prove the equality after taking the tensor product with $\Z_{(2)}$. For the rest of the proof we will work with $\Z_{(2)}$-coefficients without further mention. Notice that the morphism
		\[
		\begin{tikzcd}
			\theta:\left[\frac{\P^1\times(\P^1\times\P^1)\times\P^{g-1}\times\P^{g-1}}{\mu_2\times\PGL_2}\right]\arrow[r] & \left[\frac{\P^1\times\P^{g}\times\P^{g}}{\mu_2\times\PGL_2}\right], & (h,h_1,h_2,f,g)\mapsto(h,h_1f,h_2g)
		\end{tikzcd}
		\]
		has odd degree, hence $\theta_*$ is surjective by the push-pull formula. We will show that $\mathrm{im}(M_{1*}\circ\theta_*)$ is contained in $J$. Notice that
		\[
		\left[\frac{(\P^1)^3\times(\P^{g-1})^2}{\mu_2\times\PGL_2}\right]\cong \left[\frac{(\P^1)^3\times(V^{\vee}\otimes W_{\frac{g-1}{2}}\setminus Z_{g-1})}{G\times\PGL_2}\right]
		\]
		hence its Chow ring is a quotient of $\CH_{G\times\PGL_2}^*((\P^1)^3)$. Let $D=\{(h,h_1,h_2)\in(\P^1)^3:\ h=h_1\text{ or }h=h_2\}$, and $\Delta=D\cup(\P^1\times\Delta_{\P^1\times\P^1})$. Since $[((\P^1)^3\setminus\Delta)/G\times\PGL_2]\cong BG$, the Chow ring of $[(\P^1)^3/G\times\PGL_2]$ is additively generated by the classes obtained by pullback from $BG$ and the classes obtained by pushforward along the closed immersions $i_1:D\hookrightarrow(\P^1)^3$, $i_2:\P^1\times\Delta_{\P^1\times\P^1} \hookrightarrow(\P^1)^3$. By the push-pull formula, we only need to show that $\mathrm{im}(M_{1*}\circ\theta_*\circ i_{1*})+\mathrm{im}(M_{1*}\circ\theta_*\circ i_{2*})$ is contained in the ideal $J$. This follows from the fact that $M_1\circ\theta\circ i_1$ and $M_1\circ\theta\circ i_2$ factor through $S_1$ and $M_2$ respectively.
	\end{proof}
	
	\subsubsection{Computation of $M_{r*}$ for $r>2$}\label{subsec: Mr end of computation}
	In this subsection we show that $\mathrm{im}(M_{r*}) \subseteq I_S + I_{M}^{\leq 2} + I_H$. Recall that we have already shown that this holds after taking the tensor product with $\Z[1/2]$ in Lemma \ref{lem: Mr inverting 2}. This will allow us to work with $\Z_{(2)}$-coefficients whenever necessary.
	
	\begin{lemma}\label{lem: Mr reduction to r even}
		For odd $r>1$, we have
		\[
		\mathrm{im}(M_{r*})\subseteq I_S+\mathrm{im}(M_{1*}).
		\]
	\end{lemma}
	\begin{proof}
		We work with $\Z_{(2)}$-coefficients. It this case, the pushforward along the multiplication map
		\[
		\begin{tikzcd}
			\left[\frac{(\P^{r-1}\times\P^1)\times(\P^{g+1-r})^2}{\mu_2\times\PGL_2}\right]\arrow[r] & \left[\frac{\P^{r}\times(\P^{g+1-r})^2}{\mu_2\times\PGL_2}\right]
		\end{tikzcd}    
		\]
		is surjective, as the map is of odd degree. Moreover, its composite with $M_r$ factors through $M_1$, yielding the desired inclusion.
	\end{proof}
	
	Thanks to the previous lemma and Lemma~\ref{lem: ideal M1 + M2}, we are left with studying $M_{2r}$ for $r>1$.
	
	\begin{lemma}\label{lem: M2r obvious generators image}
		The image of $M_{2r*}$ is generated by the classes $M_{2r*}(\xi_{2r}^i)$ for $0\leq i\leq r$.
	\end{lemma}
	\begin{proof}
		The case $r=1$ has already been taken care of at the beginning of the proof of Lemma~\ref{lem: ideal M1 + M2}. The proof in the case $r>1$ is completely analogous.
	\end{proof}
	
	\begin{lemma}\label{lem: M2r i>0}
		For every $r>1$ and $i>0$, we have
		\[
		M_{2r*}(\xi_{2r}^i)\in I_S+\mathrm{im}(M_{1*})+\mathrm{im}(M_{2r-2}).
		\]
	\end{lemma}
	\begin{proof}
		Thank to Lemma \ref{lem: Mr inverting 2}, we can work with $\Z_{(2)}$-coefficient. In this proof, we will use $\GL_3$-counterparts; see \S\ref{subsec: GL3 counterparts} for the  definitions and results. For instance, we have
		\[
		\left[\frac{\P^{g+1}\times\P^{g+1}}{\mu_2\times\PGL_2}\right]\cong \left[\frac{\P(V_{\frac{g+1}{2}})\times_{\mathcal{S}}\P(V_{\frac{g+1}{2}})}{\mu_2\times\GL_3}\right].
		\]
		To prove the statement we will work $\mu_2\times\Gm^3$-equivariantly, where $\Gm^3\subset\GL_3$ is the maximal torus consisting of diagonal matrices. The commutative diagram
		\[
		\begin{tikzcd}[column sep=large]
			\left[\frac{\P(V_{r-1})\times_{\mathcal{S}}\P(V_{\frac{g+1}{2}-r})\times_{\mathcal{S}}\P(V_{\frac{g+1}{2}-r})}{\mu_2\times\Gm^3}\right]\arrow[d,hookrightarrow,"j_{r;1,0}"']\arrow[rr,hookrightarrow,"\left(j_{\frac{g+1}{2}-r+1;1,0}\right)^2"] & & \left[\frac{\P(V_{r-1})\times_{\mathcal{S}}\P(V_{\frac{g+1}{2}-r+1})\times_{\mathcal{S}}\P(V_{\frac{g+1}{2}-r+1})}{\mu_2\times\Gm^3}\right]\arrow[d,"M'_{2r-2}"]\\
			\left[\frac{\P(V_{r})\times_{\mathcal{S}}\P(V_{\frac{g+1}{2}-r})\times_{\mathcal{S}}\P(V_{\frac{g+1}{2}-r})}{\mu_2\times\Gm^3}\right]\arrow[rr,"M'_{2r}"] & &\left[\frac{\P(V_{\frac{g+1}{2}})\times_{\mathcal{S}}\P(V_{\frac{g+1}{2}})}{\mu_2\times\Gm^3}\right]
		\end{tikzcd}
		\]
		shows that $M'_{2r*}([W_{r;1,0}])\in\mathrm{im}(M'_{2r-2*})$. By~\cite[Proposition 3.16, Remark 3.17]{CLI}, we have
		\[
		\xi_{2r}^i[W_{r;1,0}]=\xi_{2r}^{i+2}+r^2c_2\xi_{2r}^i+(2r-1)\xi_{2r}^{i+1}t_1+r(2r-1)\xi_{2r}^it_1^2
		\]
		for every $i\geq0$. Also, from Remark \ref{rmk: extension commutes with maps} and \cite[Proposition 3.4]{CLI} and using the fact that $2r-1$ is always odd, we obtain that $M_{2r*}(\xi_{2r}^i)\in\mathrm{im}(M_{2r-2})$ for every $i>0$.
	\end{proof}
	
	\begin{lemma}\label{lem: M2r i=0}
		For every $r>1$, we have $M_{2r*}(1)\in\mathrm{im}(M_{1*})+I_S+I_H$.
	\end{lemma}
	\begin{proof}
		Let $m:[\P^{g+1} \times \P^{g+1} / \mu_2 \times \PGL_2 ] \to [\P^{2g+2} /\PGL_2]$ be the multiplication map $(f,g) \mapsto fg$ and let $\mathsf{sing} \subseteq \P^{2g+2}$ be the locus of singular polynomials.
		
		We have a stratification of
		$$
		\mathsf{sing}=\mathsf{sing}_1 \supseteq \mathsf{sing}_2 \supseteq \ldots \mathsf{sing}_{g+1}
		$$
		where $\mathsf{sing}_r$ is the locus of polynomials that are divisible by the square of a polynomial of degree $r$ for each $1 \leq r \leq g+1$. Equivalently, it is the image of the natural map $\pi_r:  \P^r \times \P^{2g+2-2r} \to \P^{2g+2}$. As shown in \cite[Corollary 6.3]{DL18} and \cite[Proposition 4.13]{CLI}, the maps $\pi_r$ form a $\PGL_2$-envelope for $\mathsf{sing}$. Moreover, by \cite[Propositions 5.7 and 6.1, and Lemma 6.7]{DL18}, the pushforwards of the fundamental classes $\pi_{r*}(1)$  for all $r \geq 1$ lie in the ideal generated by two classes:
		\begin{itemize}
			\item $(4g+2) \xi_{2g+2}$ whose pullback under $m^*$ is equal to $-2(2g+1) \beta_1$, which in turn is a multiple of $-2 \beta_1=  2 M_{1*}(1)-S_{1*}(1) \in I_S + \operatorname{im}(M_{1*})$;
			\item $2 \xi_{2g+2}^2-2(g+1) g c_2$ whose pullback under $m^*$ is $2 \beta_1^2-2g(g+1) c_2 = S_{1*} \phi_{1*}(\tau)+2 M_{1*}(\tau) \in I_S+  \operatorname{im}(M_{1*})$;
		\end{itemize}
		
		In particular, $m^*(\mathsf{sing}_k) \in I_S + \operatorname{im}(M_{1*})$ for all $k = 1, \ldots, g+1$. A point $(f,g)\in\P^{g+1}\times\P^{g+1}$ is in $m^{-1}(\mathsf{sing}_k)$ if and only if there exists polynomials $u$, $h_1$, $h_2$, $h$, $f_1$ and $g_1$ with $d_1:=\deg u$, $d_2:=h$, $d_2:=\deg h_1=\deg h_2$, satisfying $d_1+d_2+2d_3=k$, and such that $(f,g)=(uh^2h_1^2f_1,uh_2^2g_1)$ or $(f,g)=(uh_1^2f_1,uh^2h_2^2g_1)$. For every such triple $\underline{d}=(d_1,d_2,d_3)$, we denote by $\mathsf{sing}_{\underline{d}}$ the locus of points as above. If $d_2=0$, then $\mathsf{sing}_{\underline{d}}$ is the image of
		\[
		\begin{tikzcd}
			\P^{d_1}\times(\P^{d_3})^2\times(\P^{g+1-k+d_2})^2\arrow[r,"\varphi_{\underline{d}}"] & \P^{g+1}\times\P^{g+1}, & (u,(h_1,h_2),(f_1,g_1))\mapsto(uh_1^2f_1,uh_2^2g_1).
		\end{tikzcd}
		\]
		If $d_2\geq1$, then $\mathsf{sing}_{\underline{d}}$ is the image of the morphism
		\[
		\begin{tikzcd}
			\P^{d_1}\times\P^{d_2}\times(\P^{d_3})^2\times((\P^{g+1-k-d_2}\times\P^{g+1-k+d_2})\sqcup(\P^{g+1-k+d_2}\times\P^{g+1-k-d_2}))\arrow[r,"\varphi_{\underline{d}}"] & \P^{g+1}\times\P^{g+1} 
		\end{tikzcd}
		\]
		sending $(u,h,(h_1,h_2),(f_1,g_1))$ to $(uh^2h_1^2f_1,uh_2^2g_1)$ or to $(uh_1^2f_1,uh^2h_2^2g_1)$ depending on the connected component where the point lies in. In both cases, $\varphi_{\underline{d}}$ is $\mu_2\times\PGL_2$-equivariant under the natural $\mu_2\times\PGL_2$-action, and it is birational onto its image, thus showing that $\mathrm{im}(\varphi_{\underline{d}})=\mathsf{sing}_{\underline{d}}$ is integral of codimension $k$. Denote by $\alpha_{\underline{d}}$ the $\mu_2\times\PGL_2$-equivariant pushforward along $\varphi_{\underline{d}}$ of the fundamental class of the source, which is then the fundamental class of $\mathsf{sing}_{\underline{d}}$. If $\underline{d}=(d_1,d_2,d_3)$, we set $\mathrm{size}(\underline{d}):=d_1+d_2+2d_3$.
		
		Note that $m$ is flat and, scheme-theoretically, we have 
		\[
		m^{-1}(\mathsf{sing}_k) = \bigcup_{\mathrm{size}(\underline{d})=k} \mathsf{sing}_{\underline{d}},
		\]
		and consequently that
		\[
		m^*(\mathsf{sing}_k) = \sum_{\mathrm{size}(\underline{d})=k}\alpha_{\underline{d}}.
		\]
		Clearly, $\alpha_{(k,0,0)}=M_{k*}(1)$. Moreover, if $\underline{d}=(d_1,d_2,d_3)$ with $d_2\geq1$ or $d_3\geq1$, then $\varphi_{\underline{d}}$ factors through $S_{d_2}$ or $H_{d_2}$, respectively. Putting everything together, we have proved that $M_{k*}(1)\in\mathrm{im}(M_{1*})+I_S+I_h$ for every $k\geq1$, as wanted.
	\end{proof}
	
	Putting everything together, we obtain the following desired conclusion.
	
	\begin{proposition}\label{prop: Mr conclusion}
		We have the inclusion
		\[
		I_M\subset I_S+I_H+I_M^{\leq2}.
		\]
	\end{proposition}
	\begin{proof}
		By Lemmas~\ref{lem: ideal M1 + M2} and~\ref{lem: Mr reduction to r even}, it is enough to prove that $\mathrm{im}(M_{2r*})\subset I_S+I_H+I_M^{\leq2}$ for every $r>1$. We prove it by induction on $r>1$ and base case $r=1$. By Lemmas~\ref{lem: M2r obvious generators image} and~\ref{lem: M2r i>0}, it suffices to prove that $M_{2r*}(1)\in I_S+I_H+I_M^{\leq2}$, which is Lemma~\ref{lem: M2r i=0}.
	\end{proof}
	
	\subsection{Computation of the ideal generated by the maps $H_{r*}$}\label{subsec: computation Hr}
	
	In this section we show that $I_H \subseteq I_{S}+ I_M^{\leq 2}$.

	\subsubsection{Preliminary observations}
	
	\begin{lemma}\label{lemma: Hr inverting 2}
		After taking the tensor product $- \otimes_{\Z} \Z[1/2]$, the image of $H_{r*}$ is contained in $I_S$ for all $r=1,\ldots,(g+1)/2$.
	\end{lemma}
	
	\begin{proof}
		This follows from the fact that the composite
		\[
		\begin{tikzcd}
			\left[\frac{(\P^r)^2 \times (\P^{g+1-2r})^2}{\PGL_2}\right] \arrow[r,"\psi_r^H"] & \left[\frac{(\P^r)^2 \times (\P^{g+1-2r})^2}{\mu_2 \times \PGL_2}\right] \arrow[r,"H_r"] & \left[\frac{(\P^{g+1})^2}{\mu_2 \times \PGL_2}\right]
		\end{tikzcd}
		\]
		factors through $S_r$, and that ${(\psi_r^H)}_*$ is surjective in Chow after inverting $2$ (as $\psi_r^H$ is a finite of degree $2$).
	\end{proof}
	
	In particular, this will allow us to work with $\Z_{(2)}$-coefficients whenever necessary.
	
	\begin{lemma}
		For odd $r>1$, we have 
		$$
		\operatorname{im} (H_{r*}) \subseteq I_S + \operatorname{im} (H_{1*}).
		$$ 
	\end{lemma}
	
	\begin{proof}
		By Lemma~\ref{lemma: Hr inverting 2}, it is enough to show the inclusion after tensoring with $\otimes_{\Z} \Z_{(2)}$. In this case, the composite
		\[
		\begin{tikzcd}
			\left[\frac{(\P^{r-1} \times \P^1)^2 \times (\P^{g+1 - 2r})^2}{\mu_2 \times \PGL_2}\right] \arrow[r] & \left[\frac{(\P^r)^2 \times (\P^{g+1 - 2r})^2}{\mu_2 \times \PGL_2}\right] \arrow[r,"H_r"] & \left[\frac{(\P^{g+1})^2}{\mu_2 \times \PGL_2}\right],
		\end{tikzcd}
		\]
		where the first arrow is induced by the multiplication map, factors through $H_{1}$. Since this first map has odd degree, the induced pushforward is surjective on Chow groups after tensoring with $\Z_{(2)}$. This concludes.
	\end{proof}

	\subsubsection{Computation of $H_{1*}$}
	
	The main result of this subsection is the following.
	
	\begin{proposition}\label{prop: pushforward H1}
		We have $\operatorname{im}(H_{1*}) \subseteq I_{S}+ I_{M}^{\leq 2}$.
	\end{proposition}
	
	We first show that $H_{1*}(1)$ lies in $I_{S}+ I_M^{\leq 2}$ separately.
	
	\begin{lemma}\label{lemma: pushforward H1(1)}
		We have $H_{1*}(1) \in \mathrm{im}(S_{1*})+(M_{1*}(1),M_{1*}(\tau),M_{2*}(1))$.
	\end{lemma}
	\begin{proof}
		Let $m : [\P^{g+1} \times \P^{g+1} / \mu_2 \times \PGL_2 ] \to [\P^{2g+2} /\PGL_2]$, $\mathsf{sing}$, $\mathsf{sing}_k$, and $\alpha_{\underline{d}}$ be as in the proof of Lemma \ref{lem: M2r i=0}. Then, we have
		$$
		\mathrm{im}(S_{1*})+(M_{1*}(1),M_{1*}(\tau)) \ni m^*([\mathsf{sing}^2])= \alpha_{(2,0,0)} + \alpha_{(1,1,0)} + \alpha_{(0,2,0)} +H_{1*}(1)
		$$
		and thus $H_{1*}(1) \in \mathrm{im}(S_{1*})+(M_{1*}(1),M_{1*}(\tau),M_{2*}(1))$, as $\alpha_{(2,0,0)}=M_{2*}(1)$, $\alpha_{(1,1,0)}\in\mathrm{im}(S_{1*})$, and $\alpha_{(0,2,0)}=S_{2*}(1)\in\mathrm{im}(S_{1*})$.
	\end{proof}
	\begin{proof}[Proof of Proposition \ref{prop: pushforward H1}]
		The map
		\begin{equation*}
			\begin{tikzcd}
				g: \bigg[ \frac{(\P^1)^2 \times (\P^{g-1})^2}{ \mu_2 \times \PGL_2} \bigg] \cong \bigg[ \frac{(\P^1)^2 \times (W_{\frac{g-1}{2}} \otimes V^\vee \smallsetminus Z_{g-1})}{ G \times \PGL_2} \bigg] \subseteq \bigg[ \frac{(\P^1)^2 \times (W_{\frac{g-1}{2}} \otimes V^\vee )}{ G \times \PGL_2} \bigg]  \arrow[r] &  \bigg[ \frac{(\P^1)^2}{ G \times \PGL_2} \bigg]
			\end{tikzcd}
		\end{equation*}
		exhibits the domain of $H_1$ as an open in a vector bundle over $[(\P^1)^2 / G \times \PGL_2]$. Also, we have an open/closed decomposition 
		\begin{equation*}
			\begin{tikzcd}
				\bigg[\frac{ \Delta }{G \times \PGL_2} \bigg] \subset  \bigg[\frac{ \P^1\times\P^1}{G \times \PGL_2} \bigg] \supset   \bigg[\frac{ (\P^1\times\P^1) \smallsetminus \Delta}{G \times \PGL_2} \bigg] \cong B H_{1,1},
			\end{tikzcd}
		\end{equation*}
		where $\Delta \subseteq \P^1 \times \P^1$ denotes the diagonal and $H_{1,1}= (\Gm \times \Gm^{\times 2}) \rtimes \mu_2$ is the group introduced in \S\ref{sec: groups H}. The last isomorphism is obtained by looking at the stabilizer of any point of $\P^1\times\P^1 \smallsetminus \Delta$; for instance, taking the point $(X,Y)$ then $\mu_2$ corresponds to the pair of matrices 
		\begin{equation*}
			(A, [B])= \bigg( \begin{pmatrix}
				0 & 1 \\
				1 & 0 
			\end{pmatrix},
			\ \
			\bigg[ 
			\begin{pmatrix}
				0 & 1 \\
				1 & 0 
			\end{pmatrix}
			\bigg] \bigg)\in G\times\PGL_2.
		\end{equation*}
		
		We have obtained a commutative diagram
		\begin{equation}\label{eqn: diagram with map g*}
			\begin{tikzcd}
				\CH^*_{G \times \PGL_2}(\Delta) \arrow[r] \arrow[d] & \CH^*_{G \times \PGL_2}( (\P^1)^2) \arrow[d, "g^*" ] \arrow[r] &\CH^*(BH_{1,1}) \arrow[r] &0\\
				\CH^*_{\mu_2 \times \PGL_2}(\Delta \times (\P^{g-1})^2)  \arrow[r,"i_*"] \arrow[dr] & \CH^*_{\mu_2 \times \PGL_2}( (\P^1)^2 \times (\P^{g-1})^2) \arrow[d,"H_{1*}"] & & \\
				&  \CH^*_{\mu_2 \times \PGL_2}( (\P^g)^2) & &
			\end{tikzcd}
		\end{equation}
		where the top row is exact and the vertical arrows from the first to the second row are surjective. Note that $H_{1*} \circ i_*$ factors through $M_{1*}$, and recall that, by Proposition~\ref{prop: Chow classifying Hl1l2}, the group $\CH^*(B H_{1,1})$ is additively generated by classes in $\operatorname{im}(\psi_*)$ and classes of the form $\phi^*(\epsilon) \pi_{1}^*(\beta_2)^i$ (see \S\ref{sec: groups H} for the notation). It suffices therefore to show that the image of $H_{1*} \circ g^*$ applied to lifts of such classes is in $I_{S}+ I_M^{\leq 2}$.
		
		\emph{Case (a): Lifts of the classes in $\operatorname{im}(\psi_*)$.}
		
		In this case, we observe that there is a commutative diagram 
		
		\begin{equation*}
			\begin{tikzcd}
				\CH^*_{\Gm^2 \times \PGL_2}( (\P^1)^2) \arrow[r] \arrow[d, "\psi'_*"] & \CH^*_{\Gm^2 \times \PGL_2}((\P^1)^2 \smallsetminus \Delta) \cong \CH^*(B \Gm^3) \arrow[d,"\psi_*"]\\
				\CH^*_{G \times \PGL_2}( (\P^1)^2)  \arrow[r] &\CH^*(BH_{1,1})
			\end{tikzcd}
		\end{equation*}
		where the top horizontal arrow is surjective. It is therefore enough to show that $\operatorname{im}(H_{1*} \circ g^* \circ \psi'_*) \subseteq I_{S}+I_M^{\leq 2}$. However, $H_{1*} \circ g^* \circ \psi'_*$ is equal to the pullback under
		\[
		\begin{tikzcd}
			\bigg[ \frac{(\P^1)^2 \times (\P^{g-1})^2}{ \PGL_2} \bigg] \cong \bigg[ \frac{(\P^1)^2 \times (W_{\frac{g-1}{2}} \otimes V^\vee \smallsetminus Z_{g-1})}{ \Gm^2 \times \PGL_2} \bigg] \subseteq \bigg[ \frac{(\P^1)^2 \times (W_{\frac{g-1}{2}} \otimes V^\vee )}{ \Gm^2 \times \PGL_2} \bigg]  \arrow[r] & \bigg[ \frac{(\P^1)^2}{ \Gm^2 \times \PGL_2} \bigg]
		\end{tikzcd}
		\]
		followed by the pushforward $\psi_* \circ F_{1*} \circ G_{1*} $ and thus lies in $I_{S}$ by diagram \eqref{eqn: transfer Fr}.
		
		\emph{Case (b): Lifts of the classes of the form $\phi^*(\epsilon) \pi_{1}^*(\beta_2)^i$.} 
		
		Recall that, by \cite[Theorem 3.10]{CL24}, we have
		$$
		\CH^*(B(\Gm \rtimes \mu_2))= \frac{\Z[c_2,\gamma]}{(2 \gamma)}
		$$
		where $c_2$ is pulled back from $B(\Gm \rtimes \mu_2)\to B\PGL_2$. In particular, we may assume $\epsilon= c_2^k \gamma^j$ for some $k,j \geq 0$. Since $\phi^*(\gamma)= \pi_1^*(\gamma)$, we can write
		$$
		g^*(\text{lift of }(\phi^*(\epsilon) \pi_{1}^*(\beta_2)^i)) = g^*(c_2^k \gamma^j  \beta_2^i ).
		$$
		We claim that
		\begin{equation}\label{eqn: H1*(beta2)}
			H_1^*(\beta_2) = g^*(\beta_2)+ \alpha
		\end{equation}
		for some $\alpha= \alpha_1+\alpha_2+\alpha_3$ where $\alpha_1 \in \operatorname{im}(g^* \circ \psi'_*)$, $\alpha_2 $ is a polynomial in $c_2$ and $\gamma$ and $\alpha_3 \in \operatorname{im}(i_*)$. Assume the claim for a moment. Then, by the push-pull formula, we have
		\begin{align*}
			\begin{split}
				H_{1*}( g^*(\text{lift of }(\phi^*(\epsilon) \pi_1^*(\beta_2)^i)))&= H_{1*}(g^*(c_2^k \gamma^j  \beta_2^i )) \\
				&= H_{1*}(g^*(c_2^k \gamma^j) \cdot (H_1^*(\beta_2)-\alpha)^i) \\
				&= H_{1*}(H_1^*(c_2^k \gamma^j) \cdot (H_1^*(\beta_2)-\alpha)^i) \\
				&=c_2^k \gamma^j H_{1*}((H_1^*(\beta_2)-\alpha)^i)
			\end{split}
		\end{align*}
		Since $H_{1*}(H_1^*(\beta_2)^y \alpha^x) = \beta_2^y \cdot H_{1*}(\alpha^x)$, it is then enough to check that
		\[
		H_{1*}(\alpha_1^{x_1} \alpha_2^{x_2} \alpha_3^{x_3}) \in I_S + I_M^{\leq 2} \quad \text{for all } x_1, x_2, x_3 \geq 0.
		\]
		If $x_3 > 0$, this follows from the fact that $\operatorname{im}(i_*)$ is an ideal and that $H_{1*} \circ i_*$ factors through $M_{1*}$. If $x_3 = 0$ and $x_1 > 0$, then since $g^*$ is surjective, $\operatorname{im}(g^* \circ \psi'_*)$ is also an ideal, and the result follows from Case (a) above. If instead $x_1 = x_3 = 0$, then this follows from the push-pull formula again and Lemma \ref{lemma: pushforward H1(1)}.

		Finally, we establish the existence of a decomposition as in Equation \eqref{eqn: H1*(beta2)}. It suffices to do so modulo $\operatorname{im}(i_*)$ and $\gamma^2$. In other terms, by Corollary \ref{cor: Chow2 of H11}, it is enough to prove such a decomposition after pulling back $( j^{\PGL_2})^* \circ (\psi^H_1)^*$, where $j^{\PGL_2}$ and $\psi_1^H$ are the maps in the commutative square 
		
		\begin{equation*}
			\begin{tikzcd}
				\bigg[ \frac{((\P^1)^2 \smallsetminus \Delta) \times (\P^{g-1})^2}{\PGL_2} \bigg] \arrow[hookrightarrow, r, "j^{\PGL_2} "] \arrow[d, " (\psi_1^H)' "] &\bigg[ \frac{(\P^1)^2 \times (\P^{g-1})^2}{\PGL_2} \bigg] \arrow[r, "H_1^{\PGL_2}"] \arrow[d, "\psi^H_1 "] & \bigg[ \frac{(\P^{g+1})^2}{\PGL_2}  \bigg] \cong \bigg[ \frac{W_{\frac{g+1}{2}} \otimes V^\vee \smallsetminus Z_{g+1}}{\Gm^2 \times \PGL_2}\bigg] \arrow[d, "\psi"]\\
				\bigg[ \frac{((\P^1)^2 \smallsetminus \Delta) \times (\P^{g-1})^2}{\mu_2 \times \PGL_2} \bigg] \arrow[hookrightarrow, r, "j "]&\bigg[ \frac{(\P^1)^2 \times (\P^{g-1})^2}{\mu_2 \times \PGL_2} \bigg]  \arrow[r, "H_1" ] & \bigg[ \frac{(\P^{g+1})^2}{\mu_2 \times \PGL_2} \bigg] \cong \bigg[ \frac{W_{\frac{g+1}{2}} \otimes V^\vee \smallsetminus Z_{g+1}}{\Gm^2 \rtimes \mu_2 \times \PGL_2}\bigg]
			\end{tikzcd}
		\end{equation*}
		Indeed, the top left vertical map can be identified with the restriction of
		\begin{equation*}
			\begin{tikzcd}
				\bigg[ \frac{(\P^1)^2 \smallsetminus \Delta) \times (V^\vee \otimes W_{\frac{g-1}{2} })}{ \Gm^2 \times \PGL_2} \bigg]  \arrow[r] & \bigg[ \frac{((\P^1)^2 \smallsetminus \Delta) \times (V^\vee \otimes W_{\frac{g-1}{2} })}{ G \times \PGL_2} \bigg]
			\end{tikzcd}
		\end{equation*}
		on the open $((\P^1)^2 \smallsetminus \Delta) \times (V^\vee \otimes W_{\frac{g-1}{2}}\smallsetminus Z_{g-1})$, and in turn this is obtained as a vector bundle base change of 
		\begin{equation*}
			\begin{tikzcd}
				\psi: B \Gm^3 \cong \bigg[ \frac{(\P^1)^2 \smallsetminus \Delta) )}{ \Gm^2 \times \PGL_2} \bigg]  \arrow[r] & \bigg[ \frac{((\P^1)^2 \smallsetminus \Delta) }{ G \times \PGL_2} \bigg] \cong B H_{1,1}
			\end{tikzcd}
		\end{equation*}
		Since $Z_{g-1}$ has codimension $g$ in $V^\vee \otimes W_{\frac{g-1}{2}}$ and $g>2$, Corollary \ref{cor: Chow2 of H11} implies that $\ker(((\psi_1^H)')^*) \cap \CH^2(B\Gm^3) = \langle \gamma^2 \rangle$.
		
		Therefore, it is enough to show that $(H_1^{\PGL_2})^*(x_1x_2)$ has a decomposition as in Equation~\eqref{eqn: H1*(beta2)}, eventually after pullback along $j^{\PGL_2}$.    
		Because $\PGL_2$ has no non-trivial characters, we have 
		$$
		(H^{\PGL_2}_1)^*(x_1 x_2)=(\tau_1+\xi_{g-1}^{(1)})(\tau_2+\xi_{g-1}^{(2)})= \tau_1 \tau_2 + \xi_{g-1}^{(1)} \tau_2 + \xi_{g-1}^{(2)} \tau_1 + \xi_{g-1}^{(1)} \xi_{g-1}^{(2)}
		$$
		where $\tau_i=c_1^{\PGL_2}(\cO_{\P^1}(2))$ lie on the $i$-th projective line for $i=1,2$, and similarly $\xi_{g-1}^{(i)}= c_1^{\PGL_2}(\cO_{\P^{g-1}}(1))$. Moreover,
		$$
		\xi_{g-1}^{(1)} \tau_2 + \xi_{g-1}^{(2)} \tau_1= -x_1 \tau_2 -x_2 \tau_1= {(\psi_1^H)}^* {(\psi_1^H)}_*(-x_1 \tau_2)= {(\psi_1^H)}^*g^*( \psi'_*(-x_1\tau_2)) \in {(\psi_1^H)}^*( \operatorname{im}(g^* \circ \psi_*'))
		$$
		where the second equality uses Lemma \ref{lemma: push-pull psi}, and
		$$
		\xi_{g-1}^{(1)} \xi_{g-1}^{(2)}= g^*(\beta_2).
		$$
		Finally, $\tau_1 \tau_2+c_2$ restricts to $0$ on the complement $B \Gm$ of $[\Delta/\PGL_2] \subseteq  [(\P^1)^2/\PGL_2]$ by \cite[page 5]{Vez98}, and we are done.
	\end{proof}
	
	\subsubsection{Computation of $H_{2r*}(1)$ }
	
	In this subsection we prove the following.
	
	\begin{proposition}\label{prop: H2r(1)}
		For $r=1,\ldots,\lfloor (g+1)/4\rfloor$ we have $H_{2r*}(1) \in I_{S}+I_{M}^{\leq 2}$.
	\end{proposition}
	
	The proof is rather long, so we first explain the strategy.
	
	For $r=1,\ldots, \lfloor (g+1)/4 \rfloor$, let 
	\[
	\begin{tikzcd}
		\Phi_{2r}: \bigg[ \frac{V^\vee \otimes W_r \smallsetminus Z_{2r}} {G \times \PGL_2} \bigg] \cong \bigg[ \frac{(\P^{2r})^2} {\mu_2 \times \PGL_2} \bigg] \arrow[r] & \bigg[ \frac{(\P^{4r})^2}{\mu_2 \times \PGL_2} \bigg] \cong \bigg[ \frac{V^\vee \otimes W_{2r} \smallsetminus Z_{4r}} {G \times \PGL_2} \bigg]
	\end{tikzcd}
	\]
	be the squaring map. Notice that $H_{2r}$ factors through $\Phi_{2r} \times \mathrm{id}: [ (\P^{2r})^2 \times (\P^{g+1-4r})^2 /\mu_2 \times \PGL_2] \to [(\P^{4r})^2 \times (\P^{g+1-4r})^2 /\mu_2 \times \PGL_2]$. We will show that:
	
	\begin{itemize}
		\item ${\Phi_{2r}}_*(1)$ (and thus also $H_{2r*}(1)$) is $2$ divisible (see Lemma \ref{lemma: H2r(1) in (2)});
		\item and $\psi^* H_{2r*} (1) \in \psi^*\big( \operatorname{im}(S_{1*}) + (M_{1*}(1),M_{1*}(\tau) ))$ (see Lemma \ref{lemma: psi* H2r(1)})
	\end{itemize}
	
	An application of the next lemma, together with Lemmas \ref{lemma: IS} and \ref{lem: M1 computation low degree classes}, will then yield Proposition \ref{prop: H2r(1)}, whose proof is presented at the end of the subsection.
	
	\begin{lemma}\label{lemma: aux div plus pullback}
		Let $\alpha \in \CH^k_{\mu_2 \times \PGL_2}((\P^{g+1})^2)$ with $k \leq 2g+3$. Assume that $\alpha=0$ modulo $2$ and that $\psi_{g+1}^*(\alpha)=0 \in \CH^k_{\PGL_2}((\P^{g+1})^2)$. Then $\alpha=0$.
	\end{lemma}
	
	\begin{proof}
		Observe that 
		$$
		\bigg[ \frac{\P^{g+1} \times \P^{g+1}}{ \mu_2\times\PGL_2} \bigg] \cong \bigg[ \frac{V_{\frac{g+1}{2}} \otimes V^\vee \smallsetminus Z }{G \times \GL_3} \bigg] 
		$$
		where $Z:=V_{\frac{g+1}{2}} \times 0 \sqcup 0 \times V_{\frac{g+1}{2}} \xhookrightarrow{j} V_{\frac{g+1}{2}} \otimes V^\vee$. By~\cite[Proposition 3.6]{EG98},~\cite[page 2]{EF09}, we have an injection 
		$$
		\CH_{\mu_2\times\PGL_2}^*((\P^{g+1})^2)\cong \CH_{G\times\GL_3}^*(V_{\frac{g+1}{2}} \otimes V^\vee \smallsetminus Z)\hookrightarrow\CH_{G \times \Gm^3}^*(V_{\frac{g+1}{2}} \otimes V^\vee \smallsetminus Z)
		$$
		given by the maximal torus $\Gm^3 \subseteq \GL_3$. As in the proof of Lemma~\ref{lem: Chow alternative presentation Dg+1g+1mu2},
		\[
		\CH_{G \times \Gm^3}^*(V_{\frac{g+1}{2}} \otimes V^\vee \smallsetminus Z) \cong \frac{\Z[\beta_1,\beta_2,\gamma,t_2,t_3]}{(2\gamma,\gamma(\beta_1+\gamma),2t_2t_3(t_2+t_3), c_{2g+4}( V_{\frac{g+1}{2}} \otimes V^\vee),  j_*(1),j_*(x_1))},
		\]
		
		Since by assumption $\alpha$ is in degree smaller that $2g+4$, we will ignore $c_{2g+4}( V_{\frac{g+1}{2}} \otimes V^\vee)$ in the sequel.
		
		The pullback $\psi^*$ (this time $\psi$ denotes the base change by $V_{\frac{g+1}{2}} \otimes V^\vee \smallsetminus Z$ of $B (\Gm^2 \times \Gm^3) \to B(G \times \Gm^3)$ ) amounts to setting $\gamma=0$, $\beta_1= x_1 + x_2$, $\beta_2=x_1x_2$ and 
		$$
		j_{*}(1)= p_1 + p_2,\qquad j_*(x_1)= x_1 p_1 +x_2 p_2
		$$
		where $p_1,p_2 \in \Z[x_1,x_2,t_2,t_3]$ should be thought of as the monic polynomials of degree $g+2$ that define the projective bundle relations of $\CH_{\Gm^3}^*(\P(V_{(g+1)/2}))$. Indeed, we can factor $j = \psi \circ \widetilde{j}$, as in Equation \eqref{eqn: j factorization}, and compute
		\[
		\psi^* j_*(1) = \psi^* \psi_* (\widetilde{j}_*(1)) = \psi^* \psi_* (p_1) = p_1 + p_2,
		\]
		where the last equality follows from Lemma~\ref{lemma: push-pull psi}. Similarly, we have $ j_*(x_1) = x_1 p_1 + x_2 p_2$.
		
		Now, let $\alpha$ be as in the statement, and let $\widetilde{\alpha}$ be a lift of it in $\Z[\beta_1,\beta_2,\gamma,t_2,t_3]$. By assumption, we can choose it to be divisible by 2, hence $\widetilde{\alpha}=2\theta$ for some $\theta \in \Z[\beta_1,\beta_2,\gamma,t_2,t_3]$. Moreover, there exist $\theta_1,\theta_2,\theta_3,\theta_4\in\Z[\beta_1,\beta_2,\gamma,t_2,t_3]$ such that
		\begin{equation}\label{eq: alphatilde}
			2\theta=\widetilde{\alpha}=\gamma\theta_1+2t_2t_3(t_2+t_3)\theta_2+\widetilde{j_*(1)}\theta_3+\widetilde{j_*(x_1)}\theta_4
		\end{equation}
		where $\widetilde{j_*(1)}$ and $\widetilde{j_*(x_1)}$ are lifts of the respective classes $j_*(1)$ and $j_*(x_1)$. Up to modifying $\theta_1$, we may assume that no monomial in $\theta_3$ and $\theta_4$ is divisible by $\gamma$. To conclude that $\alpha=0$, it is enough to show that 2 divides $\theta_1$.
		Passing modulo $\gamma$ and setting $\beta_1=x_1+x_2$ and $\beta_2= x_1 x_2$, we get 
		\[
		2\overline{\theta}=2t_2t_3(t_2+t_3)\overline{\theta}_2+(p_1 +p_2)\overline{\theta}_3+(x_1 p_1+x_2 p_2)\overline{\theta}_4\in\Z[\beta_1,\beta_2,t_2,t_3],
		\]
		It follows that
		\[
		(\overline{\theta}_3+x_1\overline{\theta}_4)p_1+(\overline{\theta}_3+x_2\overline{\theta}_4)p_2\in2\Z[\beta_1,\beta_2,t_2,t_3].
		\]
		Since $\overline{\theta}_3,\overline{\theta}_4$ have smaller degree than the degree of the monic polynomials $p_1,p_2$, Gauss Lemma implies that both $\overline{\theta}_3+x_1\overline{\theta}_4$ and $\overline{\theta}_3+x_2\overline{\theta}_4$ are divisible by 2. In turn, this implies that $\overline{\theta}_3$ and $\overline{\theta}_4$ are both divisible by 2. Since $\gamma$ does not divide any monomial in $\theta_3$ or $\theta_4$, $2$ divides $\theta_1$, as wanted.
	\end{proof}
	
	We will need a further auxiliary lemma. Recall that
	\[
	\begin{tikzcd}
		\psi_{r}: \bigg[ \frac{(\P^{r})^2}{ \PGL_2} \bigg] \arrow[r] & \bigg[ \frac{(\P^{r})^2}{\mu_2 \times \PGL_2} \bigg] 
	\end{tikzcd}
	\]
	denotes the base change of $B \PGL_2 \to B(\mu_2 \times \PGL_2)$. 
	
	\begin{lemma}\label{lemma: aux H2r(1)}
		For $r=1,\ldots, (g+1)/2$ we have
		$$
		\mathrm{ker}(\psi_{2r}^*) \cap \CH^{\leq 2r+1}_{\mu_2 \times \PGL_2}( (\P^{2r})^2)= (\gamma) \cap \CH^{\leq 2r+1}_{\mu_2 \times \PGL_2}( (\P^{2r})^2).
		$$
	\end{lemma}
	\begin{proof}
		We can identify the maps $\psi_{2r}$ with 
		\[
		\begin{tikzcd}
			\psi_{2r}: \bigg[  \frac{V^\vee \otimes W_r \smallsetminus Z_{2r}}{\Gm^2 \times \PGL_2}\bigg] \arrow[r] & \bigg[  \frac{V^\vee \otimes W_r \smallsetminus Z_{2r}}{G \times \PGL_2}\bigg]
		\end{tikzcd}
		\]
		given by the inclusion of groups $\Gm^2 \subseteq G$. As in the proof of the previous lemma, up to degree $2r+1$, the induced pullback $\psi_{2r}^*$ can be identified with the ring homomorphism  
		\[
		\frac{\Z[\beta_1,\beta_2,\gamma,c_2,c_3]}{(2c_3,2 \gamma, \gamma(\gamma+\beta_1),j_*(1),j_*(x_1))} \longrightarrow \frac{\Z[x_1,x_2,c_2,c_3]}{(2c_3, p_1, p_2)}
		\]
		sending $\gamma \mapsto 0$, $\beta_1 \mapsto x_1+x_2$, $\beta_2 \mapsto x_1 x_2$ and $j_*(1) \mapsto p_1+p_2, j_*(x_1) \mapsto x_1 p_1 +x_2 p_2$. Here $j$ denotes the inclusion $Z_{2r} \hookrightarrow V^\vee \otimes W_{r}$. We wish to show that the kernel of the above morphism is $(\gamma)$. Let $x$ be an element in the kernel and let $\tilde x $ be a lift in $\CH^*(B(G \times \PGL_2))$ and $\widetilde y$ its image in $\CH^*(B (\Gm^2 \times \PGL_2))$.
		Write 
		$$
		\tilde y = a_1 p_1 +a_2 p_2
		$$
		for some $a_1,a_2 \in \CH^*(B (\Gm^2 \times \PGL_2))$. Recall that  the pullback along $\psi:B(\Gm^2 \times \PGL_2) \to B(G\times \PGL_2)$ has kernel $(\gamma)$ and image $\Z[x_1+x_2,x_1x_2,c_2,c_3]/(2c_3)$, which coincides with the image of the pullback along $B(\Gm^2\times\PGL_2)\rightarrow B(\GL_2\times\PGL_2)$. Let $I$ be the ideal $(p_1+p_2,x_1p_1+x_2p_2)$ in $\CH^*(B(\GL_2\times\PGL_2))$, and let $\widetilde{I}$ be its extension in $\CH^*(B(\Gm^2\times\PGL_2))$. If we show that $\tilde y\in\widetilde{I}$, then we would get
		\[
		\tilde y\in\widetilde{I}\cap\CH^*(B(\GL_2\times\PGL_2))=I=(\psi^*j_*(1),\psi^*j_*(x_1)),
		\]
		which in turn implies $\widetilde{x}\in(\gamma,j_*(1),j_*(x_1))$, hence $x\in(\gamma)$.
		In conclusion, it is enough to show that $\tilde y \in (p_1+p_2, x_1 p_1+x_2 p_2)\subset\CH^*(B(\Gm^2\times\PGL_2))$. 
		
		Recall that $\mu^*$ is the endomorphism of $\CH^*(B \Gm^2 \times \PGL_2)$ swapping $x_1$ and $x_2$. First, we observe the following.
		
		\emph{Claim: we have $\mu^*(a_2)=a_1$ (and conversely).} Indeed, the identity
		$$
		\mu^*(a_1) p_2 + \mu^*(a_2) p_1 =\mu^*(\tilde y)= \tilde y= a_1 p_1 +a_2 p_2
		$$
		implies that the degree $d$ in $x_1$ of $a_1$ (as a polynomial in $x_1$) is the same as the degree in $x_2$ of $a_2$. This uses the fact that $p_i$ has degree $2r+1$ in $x_i$ (and degree $0$ in $x_j$ for $j \neq i$) and each $a_i$ have total degree at most $2r<2r+1$. The proof of the claim is by induction on $d$:
		\begin{itemize}
			\item when $d=0$ we have
			$$
			a_1=\mathrm{Coeff}(a_1p_1+a_2 p_2, x_1^{2r+1})= \mu^* \mathrm{Coeff}(a_1p_1+a_2 p_2, x_2^{2r+1})= \mu^*(a_2,x_2^d);
			$$
			\item when $d>0$ instead, then
			$$
			\mathrm{Coeff}(a_1,x_1^d)= \mathrm{Coeff}(a_1p_1+a_2 p_2, x_1^{2r+1+d})= \mu^* \mathrm{Coeff}(a_1p_1+a_2 p_2, x_2^{2r+1+d})= \mu^*(\mathrm{Coeff}(a_2))
			$$
			Call the coefficient above $c\in \Z[x_2,c_2,c_3]/(2 c_3)$. Then, 
			$$
			\tilde y ':= a_1 p_1 +a_2 p_2 - (c x_1^d p_1 + \mu^*(c) x_2^d p_2)= (a_1 - c x_1^d) p_1 + (a_2- \mu^*(c) x_2^d) p_2
			$$
			satisfies $\mu^*(a_1 - c x_1^d)= (a_2- \mu^*(c) x_2^d)$ by the inductive step and thus also $\mu^*(a_1)=a_2$ as desired.
		\end{itemize}
		
		Having established the claim, we show that if $a_1= c_2^{x} c_3^{y} x_1^{i_1} x_{2}^{i_2}$ and $a_2= c_2^{x} c_3^{y} x_2^{i_1} x_1^{i_2}$ then $a_1 p_1 +a_2 p_2 \in (p_1+p_2, x_1 p_1+x_2 p_2)$. Clearly, we may assume $x=y=0$. Also, up to factorizing some power of $x_1x_2$ we may assume that either $i_1=0$ or that $i_2=0$.
		
		If $i_1=0$, then we have
		$$
		x_2^{i} p_1 +x_1^{i}p_2 - \bigg[ x_2^i (p_1+p_2) - \sum_{\substack{1\leq j\leq i, \\ j \ \text{odd}}} x_2^{i-j}x_1^{j-1} (x_1p_1+x_2p_2) + \sum_{\substack{1 \leq j \leq i, \\ j \ \text{even}}} x_2^{i-j}x_1^{j-1} (x_2p_1 +x_1 p_2) \bigg]
		$$
		is equal to $x_1^i(p_1+p_2)\in(p_1+p_2,x_1p_1+x_2p_2)$ when $i$ is odd and to $0$ when $i$ is even. Noticing that
		$$
		x_2 p_1 +x_1 p_2= (x_1+x_2) (p_1+p_2) -(x_1 p_1+x_2 p_2)\in(p_1+p_2,x_1p_1+x_2p_2)
		$$
		we are done in this case. The case $i_2=0$ is completely analogous, by swapping $p_1$ and $p_2$ everywhere. All together, this shows that $\widetilde{y}\in(p_1+p_2,x_1p_1+x_2p_2)$, concluding the proof.
	\end{proof}
	
	We are ready to prove that $\Phi_{2r*}(1) \in(2)$.
	
	\begin{lemma}\label{lemma: H2r(1) in (2)}
		For $r=1,\ldots,\lfloor (g+1)/4\rfloor$ we have $\Phi_{2r*}(1) \in(2)$.
	\end{lemma}
	\begin{proof}
		From \cite[Lemma 6.8]{DL18}, we immediately have $\psi_{4r}^* \Phi_{2r*}(1) \in (2)$. Thus, Lemma \ref{lemma: aux H2r(1)} implies that $\Phi_{2r*}(1) \in (2, \gamma)$. We may therefore write
		\begin{equation}\label{eqn: H2r(1)}
			\Phi_{2r*}(1) = 2\theta + \alpha \gamma
		\end{equation}
		for some $\theta, \alpha \in \CH^*_{\mu_2 \times \PGL_2}\big( (\P^{4r})^2 \big)$. We will show that $\alpha\gamma = 0$.
		
		We claim that $\Phi_{2r}^*(\beta_2)=2\beta_2$. To show this, first notice that $\psi_{2r}^*\Phi_{2r}^*(\beta_2)=2\beta_2$. Since the degree of $\beta_2$ is $2<2r+1$, this implies, as in lemma \ref{lemma: aux H2r(1)}, that $\Phi_{2r}^*(\beta_2)=2\beta_2+\gamma\cdot\tau$, with $\tau$ a class of degree 1, thus either $\gamma$ or 0. To show that $\tau=0$, consider the pullback along the morphism $B(\mu_2\times\PGL_2)\rightarrow B(G\times\PGL_2)$ induced by the inclusion $\mu_2\subset G$. As the pullback of both $\beta_2$ and $\Phi_{2r}^*(\beta_2)$ is 0, we get $\gamma\cdot\tau=0$ in $\CH^*(B(\mu_2\times\PGL_2))$, thus $\tau\not=\gamma$.
		
		The claim implies that
		\[
		2\theta\beta_2+\alpha\gamma\beta_2=\Phi_{2r*}(1)\beta_2=\Phi_{2r*}(\Phi_{2r}^*(\beta_2))=2\Phi_{2r*}(\beta_2)\in(2)
		\]
		thus $\alpha\gamma\beta_2\in(2)\cap(\gamma)$. It follows from Lemma~\ref{lemma: aux div plus pullback} that $\alpha\gamma\beta_2=0$ in $\CH_{G\times\PGL_2}^*(V^{\vee}\otimes W_{2r}\setminus Z_{4r})\cong\CH^*(B(G\times\PGL_2))/(j_*(1),j_*(x_1),c_{\mathrm{top}}(V^{\vee}\otimes W_{2r}))$. As the degree of $\alpha\gamma\beta_2$ is $4r+2$, we obtain an equality
		\[
		\alpha\gamma\beta_2=\theta_1j_*(1)+\epsilon j_*(x_1),
		\]
		in $\CH^*(B(G\times\PGL_2))\cong\Z[\beta_1,\beta_2,\gamma,c_2,c_3]/(2c_3,2\gamma,\gamma(\beta_1+\gamma))$, with $\theta_1$ of degree 1 and $\epsilon\in\Z$. Taking the pullback along $\psi$ as above, this implies that
		\[
		\psi^*(\theta_1)\cdot(p_1+p_2)+\epsilon(x_1p_1+x_2p_2)=0
		\]
		where $p_1,p_2$ should be thought of as the monic polynomials of degree $4r+1$ that define the projective bundle relations of $\CH_{\PGL_2}^*(\P(W_{2r}))$. By looking at the degree in $x_1$ and $x_2$ separately, the above equation yields $\psi^*(\theta_1)=-\epsilon(x_1+x_2)$, thus $\epsilon(x_2p_1+x_1p_2)=0$, which implies $\epsilon=0$.
		It follows that either $\theta_1=0$ or $\theta_1=\gamma$. By Lemma~\ref{lemma: image psi*}, we have $ j_*(1) \in (2, \beta_1 + \gamma) $, which is contained in the annihilator of $\gamma$, thus $ \gamma \alpha \beta_2 = 0 $ in $\CH^*(B(G \times \PGL_2))$. Since $\beta_2$ is not a zero-divisor in $\CH^*(B(G\times\PGL_2))$, we get that $\gamma\alpha=0$, as desired.
	\end{proof}
	
	\begin{lemma}\label{lemma: psi* H2r(1)}
		Let $ r = 1, \ldots, \left\lfloor \frac{g+1}{4} \right\rfloor $. Then
		\[
		\psi_{g+1}^*\big(H_{2r*}(1)\big) \in \psi_{g+1}^*\big( \operatorname{im}(S_{1*}) + (M_{1*}(1),M_{1*}(\tau) )),
		\]
		where Chow groups are taken with $ \mathbb{Z}_{(2)} $-coefficients.
	\end{lemma}
	
	\begin{proof}[Proof of Proposition \ref{prop: H2r(1)}]
		Note that 
		$$
		\psi_{g+1}^*( H_{2r*}(1)) = F_{2r *}(1) \cdot G_{2r*}(1) \in \CH^*_{\PGL_2}( (\P^{g+1})^2)
		$$
		where the maps $F_r$ and $G_r$ are defined in \eqref{eqn: transfer Fr} and \eqref{eqn: def Gr} respectively. Set $f_1= F_{1*}(1)$, $f_2=F_{1*}(\tau)$, $g_1= G_{1*}(1)$, and $g_2=G_{2*}(\tau)$. Then, by \cite[Lemma 6.7]{DL18}, we can write 
		$$
		F_{2r*}(1)= \theta_1 f_1 + \theta_2 f_2
		$$
		for some $\theta_1,\theta_2 \in \CH^*_{\PGL_2}( (\P^{g+1})^2)$. Recall that $\mu: [ (\P^{g+1})^2 / \PGL_2 ] \to [ (\P^{g+1})^2 / \PGL_2 ]$ denotes the involution swapping the two factors. Then, we have
		\[
		G_{2r*}(1) = \mu^* F_{2r*}(1) = \mu^*(\theta_1) g_1 + \mu^*(\theta_2) g_2,
		\]
		and thus
		\begin{align}\label{eqn: H2R(1)}
			\psi_{g+1}^*(H_{2r*}(1))= (f_1 \theta_1 g_1 \mu^*(\theta_1) + f_2 \theta_2 g_2 \mu^*(\theta_2)) + (f_1 \theta_1 g_2 \mu^*(\theta_2) + f_2 \theta_2 g_1 \mu^*(\theta_1)).
		\end{align}
		We wish to show that, when working with $\mathbb{Z}_{(2)}$-coefficients, this is of the form $\psi_{g+1}^*(z)$ for some $z \in \operatorname{im}(S_{1*}) + (M_{1*}(1),M_{1*}(\tau))$. By Lemma \ref{lemma: IS}, we have
		\[
		f_1 g_1 = 4g^2 x_1 x_2 = 4g^2 \beta_2 = g\, \psi_{g+1}^* S_{1*} \phi_{1*}(\xi_{g+1}).
		\]
		Since $\theta_1 \mu^*(\theta_1)$ is $\mu^*$-invariant, it lies in the image of $\psi_{g+1}^*$, by Lemma~\ref{lemma: image psi*}. It follows that $f_1 g_1 \theta_1 \mu^*(\theta_1)$ is in the image of $\psi_{g+1}^* \circ S_{1*}$.
		
		Similarly, we have
		\[
		f_2 g_2 = 4 \beta_2^2 - (g^2 - 1) c_2 (\beta_1^2 - 2 \beta_2) + \left( \frac{g^2 - 1}{2} \right)^2 c_2^2.
		\]
		With $\Z_{(2)}$-coefficients, this is seen to lie in the image of $\psi_{g+1}^* \circ S_{1*} + \psi_{g+1}^* ( M_{1*}(1), M_{1*}(\tau))$. Since $\theta_2 \mu^*(\theta_2)$ is also in the image of $\psi_{g+1}^*$, it follows that $f_2 g_2 \theta_2 \mu^*(\theta_2)$ lies in the image of $\psi_{g+1}^* \circ S_{1*} + \psi_{g+1}^* \circ M_{1*}$ as well.

		For the second term in parenthesis of Equation \eqref{eqn: H2R(1)}, consider the commutative diagram
		
		\begin{equation*}
			\begin{tikzcd}
				& \left[\frac{\P^1\times((\P^{g-1}\times\P^{g+1}) \sqcup (\P^{g+1} \times \P^{g-1}))}{ \PGL_2}\right] \arrow[r,"F_1 \sqcup G_1"] \arrow[d,"\psi_1^S"] & \left[\frac{\P^{g+1}\times\P^{g+1}}{\PGL_2}\right]\arrow[d,"\psi_{g+1}"]\\
				\left[\frac{\P^1\times(\P^{g-1}\times\P^{g+1})}{\PGL_2}\right]\arrow[r,"\phi_1","\cong"'] & \left[\frac{\P^1\times((\P^{g-1}\times\P^{g+1}) \sqcup (\P^{g+1} \times \P^{g-1}))}{\mu_2 \times \PGL_2}\right]  \arrow[r,"S_1"] & \left[\frac{\P^{g+1}\times\P^{g+1}}{\mu_2 \times \PGL_2}\right]
			\end{tikzcd}
		\end{equation*}
		
		Define $z_F=F_1^*(\theta_1 g_2 \mu^*(\theta_2))$ and $z_G= G_1^*(\theta_2 f_2 \mu^*(\theta_1))$, so that the second term in parenthesis in Equation \eqref{eqn: H2R(1)} is given by
		$$
		F_{1*}(z_F)+G_{1*}(z_G)= (F_1 \sqcup G_1)_* ((\psi_1^S)^*(\phi_{1*}(z_F)))= \psi_{g+1}^*(S_{1*}(\phi_{1*}(z_F))).
		$$
		This concludes the proof.
	\end{proof}
	
	\begin{proof}[Proof of Proposition \ref{prop: H2r(1)}]
		By Lemma \ref{lemma: Hr inverting 2}, it is enough to show this with $\Z_{(2)}$ coefficients. By Lemmas \ref{lemma: IS} and \ref{lem: M1 computation low degree classes}, every $z \in \operatorname{im}(S_{1*}) +((M_{1*}(1), M_{1*}(\tau))$ is divisible by $2$. The conclusion follows from Lemmas \ref{lemma: aux div plus pullback}, \ref{lemma: H2r(1) in (2)} and \ref{lemma: psi* H2r(1)}.    
	\end{proof}
	
	\subsubsection{Computation of $H_{2r*}$}
	
	In this subsection we conclude the computation and prove the following.
	
	\begin{proposition}\label{prop: full H2r}
		For $r=1,\ldots, (g+1)/2$ we have $\operatorname{im}(H_{2r*}) \subseteq I_{S}+ I_M^{\leq 2}$.
	\end{proposition}
	
	Before starting the proof we make the following reduction.
	
	\begin{remark}
		It suffices to prove Proposition \ref{prop: full H2r} in the case where $r$ is even (or even a power of $2$), or when $r = 1$. 
		
		Indeed, if $r > 1$ is not a power of $2$, we can write $2r = 2^{\ell} + u$ with $0 < u < 2^{\ell}$, and consider the composite
		\[
		\begin{tikzcd}
			\left[ \frac{(\P^{2^\ell} \times \P^{u})^2 \times (\P^{g+1 - 4r})^2}{\mu_2 \times \PGL_2} \right]
			\arrow[r] &
			\left[ \frac{(\P^{2r})^2 \times (\P^{g+1 - 4r})^2}{\mu_2 \times \PGL_2} \right]
			\arrow[r,"H_{2r}"] &
			\left[ \frac{\P^{g+1} \times \P^{g+1}}{\mu_2 \times \PGL_2} \right],
		\end{tikzcd}
		\]
		where the first map is the natural multiplication map. The resulting morphism factors through $H_{2^\ell}$. Moreover, the first map has degree $\binom{2r}{2^\ell}$, which is odd (see also \cite[Lemma 5.24]{CLI}). In particular, when working with $\Z_{(2)}$-coefficients, its pushforward in Chow is surjective. Since Proposition \ref{prop: full H2r} is proved after inverting $2$ in Lemma \ref{lemma: Hr inverting 2}, the claim follows.
	\end{remark}
	
	\begin{proof}[Proof of Proposition \ref{prop: full H2r}]
		We will assume throughout the proof that $r=1$ or that $r$ is a even. We start by identifying 
		\[
		\begin{tikzcd}
			\bigg[ \frac{(\P^{2r})^2 \times (\P^{g+1-4r})^2}{\mu_2 \times \PGL_2} \bigg] \cong \bigg[ \frac{\big( (V^{(1)})^\vee \otimes W_{2r} \smallsetminus Z_{4r} \big) \times \big( (V^{(2)})^\vee \otimes W_{\frac{g+1}{2}-2r} \smallsetminus Z_{g+1-4r} \big)}{H_{0,2} \times \PGL_2} \bigg] 
		\end{tikzcd}
		\]
		where $H_{0,2}$ acts on the scheme in the numerator on the right via $ \pi_1: H_{0,2} \to G$ acting on $ (V^{(1)})^\vee \otimes W_{2r} \smallsetminus Z_{4r} $ and via $\pi_2: H_{0,2} \to G$ acting on $(V^{(2)})^\vee \otimes W_{\frac{g+1}{2}-2} \smallsetminus Z_{g+1-2r}$. Since $\PGL_2$ satisfies CKP by Proposition \ref{prop: CKP for some classifying stacks}, by Proposition \ref{prop: Chow classifying Hl1l2}, the Chow ring of this stack is additively generated by classes of the form:
		\begin{enumerate}
			\item\label{item: type1} $\alpha \cdot \beta$ where $\alpha$ is pulled back from $B \PGL_2$ and $\beta \in \operatorname{im}(\psi^H_{2r*})$;
			\item\label{item: type2} $\alpha \gamma^\ell \pi_1^*(\beta_2)^{i_1} \pi_2^*(\beta_2)^{i_2}$ for some $\alpha$ pulled back from $B\PGL_2$ and integers $\ell, i_1, i_2 \geq 0$.
		\end{enumerate}
		
		The pushforward $H_{2r*}(\alpha \beta )= \alpha H_{2r*}(\beta)$ of classes of type~\eqref{item: type1} is contained in $I_S$. Indeed, $H_{2r} \circ \psi^H_{2r}= \psi_{g+1} \circ H_{2r}^{\PGL_2}$ factors as 
		$$
		\bigg[ \frac{(\P^{2r})^2 \times (\P^{g+1-4r})^2}{ \PGL_2} \bigg] \to \bigg[ \frac{\P^{2r} \times \P^{g+1-4r} \times \P^{g+1}}{ \PGL_2} \bigg] \xrightarrow{F_{2r}} \bigg[ \frac{\P^{g+1} \times \P^{g+1}}{ \PGL_2} \bigg] \xrightarrow{\psi_{g+1}} \bigg[ \frac{\P^{g+1} \times \P^{g+1}}{ \mu_2 \times \PGL_2} \bigg]
		$$
		and hence the image of its pushforward lies in $I_S$ by the diagram \eqref{eqn: transfer Fr}.
		
		Now, we deal with classes of type~\eqref{item: type2}. By the push-pull formula, to show that the pushforward via $H_{2r*}$ of such classes lies in $I_S + I_{M}^{\leq 2}$, it suffices to treat the case $\alpha = 1$, $l=0$. Moreover,
		\begin{align*}
			H_{2r}^*(\beta_2) &= c_2^{\mu_2 \times \PGL_2}( \cO_{\P^{2r}}(2) \otimes \cO_{\P^{g+1-4r}}(1) \otimes \cO_{\P^{2r}}(2) \otimes \cO_{\P^{g+1-4r}}(1)) \\
			& \equiv \pi_{2}^*(\beta_2) \mod (2)
		\end{align*}
		and $(2) \subseteq \operatorname{im}(\psi^H_{2r*})$ because $\psi^H_{2r}$ is a degree-$2$ map. It follows from the above and another application of the push-pull formula that it suffices to show $H_{2r*}(\pi_1^*(\beta_2))^i \in I_{S}+I_M^{\leq 2}$ for all $i \geq 1$. For this we use $\GL_3$-counterparts (see \S\ref{subsec: GL3 counterparts}) and proceed by induction on $r$.
		
		First, notice that following. For $r=1$, the composite $H_{2r*}' \circ j_{r;1,0} $ factors through $M_2'$ and thus 
		$$
		H'_{2*}( [W_{1;1,0} \times W_{1;1,0}]) \in \widetilde{I_S}+\widetilde{I_{M}^{\leq 2}}
		$$
		
		For $r>1$, the composite $H'_{2r*} \circ j_{r;1,0} $, where $j_{r;1,0}$ is the map in Equation \eqref{eqn: map j-m,r,l}, factors through $H'_{2r-2}$, thus 
		\begin{equation}
			H'_{2r*}( [W_{r;1,0} \times W_{r;1,0}]) \in \operatorname{im}({H'_{2r-2}}_*)
		\end{equation}
		and the latter will sit inside $\widetilde{I_S}+\widetilde{I_M^{\leq 2}}$ by the inductive step.
		
		Now, recall that $\CH^*(B \Gm^3)$ is a free module over $\CH^*(\GL_3)$ with basis $t_1,t_2,t_1^2,t_1 t_2, t_1^2 t_2$. It follows from this that we can uniquely write any element in $\CH^*_{\mu_2 \times \Gm^3}(\P(V_r)^2)$ as a linear combination of these monials in $t_i$ with coefficients in $\CH^*_{\GL_3}(\P(V_r))$; see \cite[Proposition 3.4]{CLI}. A computation using Lemma \ref{lem: computation Wm10} shows that, for $r=1$ or $r$ even, and working modulo $(2,\gamma)$, the coefficient of $t_1^2$ of $[W_{r;1,0} \times W_{r;1,0}]$ is given by:
		\begin{itemize}
			\item If $r = 1$, the coefficient is simply $\pi_1^*(\beta_2)$;
			\item If $r > 1$ is a power of $2$, the coefficient is  
			$$
			\pi_1^*(\beta_2)-\pi_1^*(\beta_1)^2-c_2
			$$
			Note that, modulo $(2,\gamma)$, we have $\psi_*(x_1^2)=\beta_1^2$ and thus, working further modulo $\operatorname{im}(\psi^H_{2r*})$, the expression remains $\pi_1^*(\beta_2)-c_2$. Finally, note that $H_{2r*}(c_2)= c_2 H_{2r*}(1) \in I_S+I_{M}^{\leq 2}$.
		\end{itemize}
		Note that for every term of the form $c \cdot \gamma$ in the coefficient of $t_1^2$ of $[W_{r;1,0} \times W_{r;1,0}]$ it must be $c \in \operatorname{im}(\psi_{2r*}^H)$ by Proposition \ref{prop: Chow classifying Hl1l2}, as $c$ has degree 1.
		
		In particular, by \cite[Proposition 3.4]{CLI}, we immediately conclude that
		$
		H_{2*}(\pi_1^*(\beta_2)) \in I_S + I_M^{\leq 2}.
		$
		For $r > 1$, we apply the inductive hypothesis to see that 
		$
		H'_{2r*}([W_{r;1,0} \times W_{r;1,0}]) \in \widetilde{I_S}+\widetilde{I_M^{\leq 2}}
		$
		and again by \cite[Proposition 3.4]{CLI}, we conclude that
		$
		H_{2r*}(\pi_1^*(\beta_2)) \in I_S + I_M^{\leq 2}.
		$
		Finally, to show that $H_{2r*}(\pi_1^*(\beta_2)^i) \in I_{S}+ I_M^{\leq 2}$ for $i>1$, argue as above using the identity
		$$
		H'_{2r*}( [W_{r;1,0} \times W_{r;1,0}] \cdot \pi_1^*(\beta_2)^{i-1})= H'_{2r*}({ j_{r;1,0}}_*( j_{r;1,0}^*( \pi_1^*(\beta_2)))) \subseteq  \widetilde{I_S}+\widetilde{I_M^{\leq 2}}.
		$$
		This concludes the proof.
	\end{proof}
	
	\subsection{Conclusion of the computation}\label{subsec: conclusion computation}
	
	So far we have computed the ideal in $\CH_{\mu_2\times\PGL_2}^*(\P(W_{\frac{g+1}{2}})^2)$ generated by image of $M_{r*}$, $S_{r*}$ and $H_{r*}$ for all $r$. We are left with showing that the classes 
	$$
	c_{\mathrm{top}}(V^{\vee}\otimes W_{\frac{g+1}{2}}), \ j_*(1), \ \text{and } j_*(x_1)
	$$
	appearing in Lemma~\ref{lem: Chow alternative presentation Dg+1g+1mu2} are superfluous, in the following sense. Let $I$ be the ideal in 
	$$\
	\CH^*(B(G\times\PGL_2))\cong \frac{\Z[\beta_1,\beta_2,\gamma,c_2,c_3]}{(2\gamma,\gamma(\beta_1+\gamma),2c_3)}
	$$
	generated by the classes computed in Lemmas~\ref{lemma: IS},~\ref{lem: M2 computation low degree classes} and~\ref{lem: M1 computation low degree classes}. Then, we seek to prove that any lift in the above ring of the last three classes of Lemma~\ref{lem: Chow alternative presentation Dg+1g+1mu2} is contained in $I$. 
	
	\begin{notation}
		By abuse of notation, we simply write that such a class is contained in $I$, for instance $j_*(1)\in I$.
	\end{notation}
	
	The following is the main result of this section.
	
	\begin{proposition}\label{prop: projective relations are superfluous}
		The classes $c_{\mathrm{top}}(V^{\vee}\otimes W_{\frac{g+1}{2}})$, $j_*(1)$ and $j_*(x_1)$, viewed in $\CH^*(B(G\times\PGL_2))$, are contained in $I$.
	\end{proposition}
	
	We split the proof in two lemmas.
	
	\begin{lemma}\label{lem: ctop is superfluos}
		Let $I_M^2$ be the ideal in 
		$$
		\CH^*(B(G\times\PGL_2))\cong \frac{\Z[\beta_1,\beta_2,\gamma,c_2,c_3]}{(2\gamma,\gamma(\beta_1+\gamma),2c_3)}
		$$
		generated by the classes appearing in Lemma~\ref{lem: M2 computation low degree classes}. Then, $c_{\mathrm{top}}(V^{\vee}\otimes W_{\frac{g+1}{2}})\in I_M^{2}\subset I$.
	\end{lemma}
	\begin{proof}
		Under the isomorphism
		\begin{equation}\label{eq: isomorphisms presentation in superfluous classes}
			\left[\frac{V^{\vee}\otimes W_{\frac{g+1}{2}}}{G\times\PGL_2}\right] \cong \left[\frac{V^{\vee}\otimes V_{\frac{g+1}{2}}}{G\times\GL_3}\right]
		\end{equation}
		the class $c_{\mathrm{top}}(V^{\vee}\otimes W_{\frac{g+1}{2}})$ corresponds to $c_{\mathrm{top}}(V^{\vee}\otimes V_{\frac{g+1}{2}})$ in $\CH^*(BG\times[\mathcal{S}/\GL_3])$. From now on, we will work in the $\GL_3$-equivariant setting.
		
		Note that $C:=c_{\mathrm{top}}(V^{\vee}\otimes V_{\frac{g+1}{2}})$ is equal to the restriction to $BG\times[\mathcal{S}/\GL_3]$ of the top Chern class of the $\GL_2\times\GL_3$-equivariant bundle $V_{\GL_2}^{\vee}\otimes V_{\frac{g+1}{2}}$ over $B\GL_2\times[\mathcal{S}/\GL_3]$. Moreover, by diagram~\eqref{eq: huge commutative diag}, we know that $M_2'$ is the restriction of a $\GL_2\times\GL_3$-equivariant morphism, that we denote by the same name. This allows us to work equivariantly with respect to $\GL_2\times\GL_3$ and show that the statement holds in $\CH^*(B\GL_2\times[\mathcal{S}/\GL_3]))$.
		
		Let $\Gm^2\times\Gm^3\subset\GL_2\times\GL_3$ be the maximal torus consisting of diagonal matrices, and $\widetilde{I}_M^2$ be the extension of $I_M^2$ to $\CH^*(B\Gm^2\times[\mathcal{S}/\Gm^3])$.    
		Notice that
		\[
		C=c_{\mathrm{top}}((\chi^{(1)})^{-1}\otimes V_{\frac{g+1}{2}})\cdot c_{\mathrm{top}}((\chi^{(2)})^{-1}\otimes V_{\frac{g+1}{2}})\in\CH^*(B\Gm^2\times[\mathcal{S}/\Gm^3])
		\]
		and recall that
		\[
		\CH^*([\mathcal{S}/\Gm^3])\cong\frac{\Z[t_1,t_2,t_3]}{(t_1+t_2+t_3,2t_1t_2t_3)}\cong\frac{\Z[t_1,t_2]}{(2(t_1+t_2)t_1t_2)}.
		\]
		By the key technical result~\cite[Proposition 3.4]{CLI}, it is enough to show that
		\[
		t_1t_2C=t_1t_2c_{\mathrm{top}}^{\Gm^2\times\Gm^3}(V^\vee\otimes V_{\frac{g+1}{2}})\in\widetilde{I}_M^2.
		\]   
		Therefore, from now on we will work $(\Gm^2\times\Gm^3)$-equivariantly, and we will still denote by $C$ its restriction to $B\Gm\times[\mathcal{S}/\Gm^3]$.
		
		Consider the restriction to the locus $\mathcal{S}_{0,0,2}$, which is the same as working modulo $2t_3$ (equivalently modulo $2(t_1+t_2)$), see equation~\eqref{eqn: Chow P(Vm) restricted to Si} in Section~\ref{subsec: GL3 counterparts}. Then, analogously to~\cite[Lemma 5.11]{CLI}, we see that the unique lifting of
		\[
		M_{2*}'([W_{1;1,0}])=\left[W_{\frac{g+1}{2};1,0}^{(1)}\right]\times\left[W_{\frac{g+1}{2};1,0}^{(2)}\right]
		\]
		to $\CH^*(B\Gm^2\times[\mathcal{S}/\Gm^3])$ divides the class $C$ modulo $2t_3$ in the same ring. This implies that $C=\alpha+2t_3C_1$ for some classes $\alpha\in\widetilde{I}_M^2$ and $C_1$. Using that $2t_1t_2t_3=0$, we get
		\[
		t_1t_2C=t_1t_2\alpha\in\widetilde{I}_M^2,
		\]
		as wanted.
	\end{proof}
	
	\begin{lemma}\label{lem: j* is superfluous}
		The classes $j_*(1)$ and $j_*(x_1)$ are contained in $I$.
	\end{lemma}
	\begin{proof}
		We follow a similar strategy as for Lemma~\ref{lem: ctop is superfluos}. Again, we start by working $G\times\GL_3$-equivariantly, that is, in the space $[V^{\vee}\otimes V_{\frac{g+1}{2}}/G\times\GL_3]$, see equation~\eqref{eq: isomorphisms presentation in superfluous classes}. By Lemma~\ref{lem: ctop is superfluos}, we can then always work with $0\in V^{\vee}\otimes V_{\frac{g+1}{2}}$ removed, as its class corresponds exactly to the class in that statement. More precisely, we consider the commutative diagram
		\begin{equation}\label{eq: commutativity js}
			\begin{tikzcd}
				\left[\frac{Z_{g+1}\smallsetminus0}{G\times\GL_3}\right]\arrow[r,"j^o"] & \left[\frac{V^{\vee}\otimes V_{\frac{g+1}{2}}\smallsetminus0}{G\times\GL_3}\right]\arrow[r] & BG\times[\mathcal{S}/\GL_3]\\
				\left[\frac{((\chi^{(1)})^{-1}\otimes V_{\frac{g+1}{2}}\smallsetminus0)\times0}{\Gm^2\times\GL_3}\right]\arrow[r,"\widetilde{j}^o"]\arrow[u,"v","\cong"'] & \left[\frac{((\chi^{(1)})^{-1}\otimes V_{\frac{g+1}{2}}\times(\chi^{(2)})^{-1}\otimes V_{\frac{g+1}{2}})\smallsetminus0}{\Gm^2\times\GL_3}\right]\arrow[u,"\psi_{g+1}"']\arrow[r] & B\Gm^2\times[\mathcal{S}/\GL_3]\arrow[u,"\pi"]
			\end{tikzcd}
		\end{equation}
		and show that $j^o_*(1)$, $j^o_*(v_*(x_1))\in I \subseteq \CH^*( BG \times [\mathcal{S} /\GL_3])$, with the notation above. By Lemma~\ref{lem: Chow alternative presentation Dg+1g+1mu2} or the diagram above, we know that $j^o_*(1)=\psi_{g+1*}\widetilde{j}^o_*(1)$ and $j^o_*(v_*(x_1))=\psi_{g+1*}(x_1\widetilde{j}^o_*(1))$.
		
		Let us work $\Gm^3$-equivariantly everywhere in diagram~\eqref{eq: commutativity js}, instead of $\GL_3$-equivariantly, and let $\widetilde{I}$ be the extension of $I$ in $\CH^*(BG\times[\mathcal{S}/\Gm^3])$. First, notice that the $\GL_3$-counterparts $G_2'$ of $G_2$ (see Equation~\eqref{eqn: def Gr}) and $S_2'$ of $S_2$ extend to morphisms $\overline{G}_2'$ and $\overline{S}'_2$ fitting in the following commutative square
		\begin{equation}\label{eq: commutativity G2 S2 bar}
			\begin{tikzcd}
				\left[\frac{(\chi^{(3)})^{-1}\otimes V_{1}\times_{\mathcal{S}}(\chi^{(1)})^{-1}\otimes V_{\frac{g+1}{2}-r}\times_{\mathcal{S}}(\chi^{(2)})^{-1}\otimes V_{\frac{g+1}{2}}}{\Gm^2\times \Gm\times\GL_3}\right]\arrow[r,"\overline{G}_2'"]\arrow[d,"\phi_G","\cong"'] & \left[\frac{(\chi^{(1)})^{-1}\otimes V_{\frac{g+1}{2}}\times_{\mathcal{S}}(\chi^{(2)})^{-1}\otimes V_{\frac{g+1}{2}}}{\Gm^2\times\GL_3}\right]\arrow[d,"\psi_{g+1}"]\\
				\left[\frac{ X }{G \times \Gm \times\GL_3}\right]\arrow[r,"\overline{S}_2'"] & \left[\frac{V^{\vee}\otimes V_{\frac{g+1}{2}}}{G\times\GL_3}\right]
			\end{tikzcd}
		\end{equation}
		where $X$ is the disjoint union of two copies of the source of $\phi_G$, with the usual $G \times \Gm \times \GL_3$-action. Let $t=c_1(\chi^{(3)})$ be the first Chern class of the standard representation of the $\Gm$ factor appearing in the middle.
		
		Restricting again to the $\Gm^3$-invariant open subscheme $\mathcal{S}_{0,0,2}$ of $\mathcal{S}$, by~\cite[Lemma 5.11]{CLI} we know that $\widetilde{j}^o_*(1)$ is equal modulo $2t_3$ to $\overline{G}'_{2*}(\alpha_1)\alpha_2\in\widetilde{I}$, where $\alpha_1$ is the unique lift of $[W_{1;1,0}]$ to $\CH^*(B(\Gm^2\times\Gm)\times[\mathcal{S}/\Gm^3])$, and $\alpha_2$ is some class in the same ring.
		Multiplying by $t_1t_2$ and using the fact that $2t_1t_2t_3=0$, we get that $t_1t_2\widetilde{j}^o_*(1)=t_1t_2\overline{G}'_{2*}(\alpha_1)\alpha_2$, hence
		\[
		t_1t_2j_*^o(1)=t_1t_2\psi_{g+1*}(\widetilde{j}^o_*(1))=t_1t_2\psi_{g+1*}(\overline{G}'_{2*}(\alpha_1)\alpha_2)\in\widetilde{I}.
		\]
		
		The last inclusion in $\widetilde{I}$ is explained as follows. Since $\CH^*(B(\Gm^2\times\Gm^3))$ is generated by $1$, $x_1$ as a $\CH^*(B(G\times\Gm^3))$-module, the image under $\psi_{g+1*}$ of the ideal generated by $\overline{G}'_{2*}(\alpha_1)$ is equal to  $(\psi_{g+1*}(\overline{G}'_{2*}(\alpha_1)),\psi_{g+1*}(\overline{G}'_{2*}(\alpha_1)x_1))$. In turn, by diagram~\eqref{eq: commutativity G2 S2 bar}, this ideal is equal to the ideal generated by $\overline{S}_{2*}'(\phi_{G*}(\alpha_1))$ and $\overline{S}_{2*}'(\phi_{G*}(\alpha_1x_1)))$. Since $\alpha_1=\theta_0+\theta_1t+\theta_2t^2$ for some $\theta_i\in\CH^*([\mathcal{S}/\GL_3])$, this last ideal is contained in the ideal generated by $\overline{S}_{2*}'\phi_{G*}(t^i)$ and $\overline{S}_{2*}'\phi_{G*}(t^ix_1)$ for $i=0,1,2$ and these classes are contained in $\widetilde{I}$ by Lemma~\ref{lemma: IS}.
		
		Similarly, we have
		\[
		t_1t_2j^o_*(v_*(x_1))=t_1t_2\psi_{g+1*}(\widetilde{j}^o_*(x_1))=t_1t_2\psi_{g+1*}(x_1\widetilde{j}^o_*(1))
		=t_1t_2\psi_{g+1*}(x_1\overline{G}'_{2*}(\alpha_1)\alpha_2)\in\widetilde{I}.
		\]    
		Applying~\cite[Lemma 3.3]{CLI} we get that $j^o_*(1),j^o_*(v_*(x_1))\in I$.
	\end{proof}
	Putting everything together, we get the Chow ring of $[\cD_{g+1,g+1}/\mu_2]$ and $\RH_g^{(g+1)/2}$.
	\begin{proof}[Proof of Theorem~\ref{thm: Chow Dg+1g+1mu2}]
		This follows from the discussion in \S\ref{subsec: strategy}, Lemmas~\ref{lemma: IS}, \ref{lem: M1 computation low degree classes}, \ref{lem: M2 computation low degree classes}, \ref{lem: ideal M1 + M2}, and Propositions \ref{prop: Mr conclusion}, ~\ref{prop: full H2r},~\ref{prop: projective relations are superfluous}. Note that in rewriting the generators of $I$ we used the identities
		$$
		-2 M_{1*}(1) - S_{1*}(1)= 2 \beta_1
		$$
		and 
		$$
		(g-1)M_{1*}(\tau)+\left(\frac{g+1}{2}\right)(-S_{1*}\phi_{1*}(\tau)+2\beta_1^2)= 4g \beta_2= S_{1*}\phi_{1*}(\xi_{g+1}).
		$$
	\end{proof}
	
	\begin{proof}[Theorem~\ref{thm: Chow RHgg+1 odd}]
		By Lemma~\ref{lemma: key cartesian diagram} and the discussion in \S\ref{subsec: pres g odd}, we know that $\RH_g^{(g+1)/2}$ is the root gerbe over $[\cD_{g+1,g+1}/\mu_2]$ associated with the pullback $\cL$ of $\cO_{\P(W_{g+1})}(-1)$ along the multiplication map
		\[
		\begin{tikzcd}
			\left[\frac{\cD_{g+1,g+1}}{\mu_2}\right]\cong\left[\frac{\P(W_{\frac{g+1}{2}})\times\P(W_{\frac{g+1}{2}})}{\mu_2\times\PGL_2}\right]\arrow[r] & \left[\frac{\P(W_{g+1})}{\PGL_2}\right]\cong\cD_{2g+2}.
		\end{tikzcd}
		\]
		By Theorem~\ref{thm: Chow Dg+1g+1mu2} and~\cite[Proposition 3.5]{CLI}, it is enough to compute $c_1(\cL)$. Under the identifications in Lemma~\ref{lem: Chow alternative presentation Dg+1g+1mu2}, this is the same as the pullback of $t=c_1(\chi)\in\CH^*(B\Gm)$ along the morphism $B(G\times\PGL_2)\rightarrow B(\Gm\times\PGL_2)$ sending $(a,b;\epsilon, [B])$ to $(ab,[B])$. Therefore, we have $c_1(\cL)=\beta_1+\gamma$.
	\end{proof}
	
	\subsection{Geometric Interpretation of the Generators}\label{subsec: interpretation generators g+1}
	
	By Theorem~\ref{thm: Chow RHgg+1 odd}, we know that $\CH^*(\RH_g^{(g+1)/2})$ is generated by $\beta_1,\beta_2,\gamma,c_2,c_3,t$. In this section, we give a geometric interpretation of these classes, producing vector bundles over $\RH_g^{(g+1)/2}$ whose Chern classes give all the elements written above.
	
	First, notice that the classes $c_2$, $c_3$ and $t$ are pullback of classes from $\cH_g$ along the map $\RH_g^{\frac{g+1}{2}}\rightarrow\cH_g$. For their geometric interpretation we refer to \cite[Section 5.3]{CLI} (or ~\cite[Theorem 7.2]{DL18} or \cite[Theorem 7.2]{FV11}), where the same notation $c_2,c_3$ and $t$ is adopted.
	
	The classes $\beta_1$, $\beta_2$ and $\gamma$ are the pullback of the homonymous classes of $[\cD_{g+1,g+1}/\mu_2]$ along the map $\RH_g^{\frac{g+1}{2}}\rightarrow[\cD_{g+1,g+1}/\mu_2]$ from Lemma~\ref{lemma: key cartesian diagram}. We describe such classes on $[\cD_{g+1,g+1}/\mu_2]$.
	
	First, $\gamma$ is pulled back from $B\mu_2$, where it corresponds to the universal $\mu_2$-torsor $\mathrm{Spec}(k) \to B\mu_2$. The pullback of this torsor under the map $[\cD_{g+1,g+1}/\mu_2] \to B\mu_2$ is the projection 
	\[
	p : \cD_{g+1,g+1} \to [\cD_{g+1,g+1}/\mu_2].
	\]
	By \cite[Lemma 4.17]{CLI}, it follows that
	\[
	\gamma = c_1(p_* \cO_{\cD_{g+1,g+1}}).
	\]
	
	As in~\cite[Section 5.3]{CLI}, let $\mathcal{N}_{\frac{g+1}{2}}$ be the line bundle on $\mathcal{D}_{g+1}$ defined by $\pi_*\omega_\pi^{\otimes\frac{g+1}{2}}(\mathcal{F}_{g+1})$ where $\mathcal{F}_{g+1} \subseteq \mathcal{P} \xrightarrow{\pi} \mathcal{D}_{g+1}$ are the universal divisor and the universal Brauer-Severi variety of relative dimension $1$ over $\mathcal{D}_{g+1}$. Notice that there is a $\mu_2$-action on the vector bundle $\mathrm{pr}_1^*\mathcal{N}_{\frac{g+1}{2}}\oplus\mathrm{pr}_2^*\mathcal{N}_{\frac{g+1}{2}}$ over $\cD_{g+1,g+1}$ that exchanges the two components.
	This induces a rank 2 vector bundle $\widetilde{\mathcal{N}}_{\frac{g+1}{2}}$ over $[\cD_{g+1,g+1}/\mu_2]$. The following lemma concludes the geometric description of the generators.
	
	\begin{lemma}\label{lem: geometric interpretation beta1 beta2}
		In $\CH^*([\cD_{g+1,g+1}/\mu_2])$ we have
		\[
		\beta_1=c_1(\widetilde{\mathcal{N}}_{\frac{g+1}{2}}^{\vee}),\qquad\beta_2=c_2(\widetilde{\mathcal{N}}_{\frac{g+1}{2}}^{\vee}).
		\]
	\end{lemma}
	
	\begin{proof}
		By~\cite[Lemma 5.20]{CLI}, $\mathcal{N}_{\frac{g+1}{2}}$ is the pullback of the line bundle $\chi^{-1}$ on $B\Gm$ along the composite
		\begin{tikzcd}
			D_{g+1}\cong\left[\frac{\chi^{-1} \otimes W_{\frac{g+1}{2}}\setminus\Delta}{\Gm\times\PGL_2}\right] \arrow[r] & B(\Gm \times \PGL_2) \arrow[r] & B\Gm.
		\end{tikzcd}
		It follows that $\widetilde{N}_{\frac{g+1}{2}}^{\vee}$ is the pullback along the morphism $\left[D_{g+1,g+1}/\mu_2\right]\rightarrow BG$ of the rank 2 vector bundle $V$. This concludes.
	\end{proof}
	
	\section{Application: Hyperelliptic Prym pairs of odd genus}\label{sec: Spin pairs}
	
	In this section we prove Proposition \ref{prop: decomposition Spin}, providing an identification of the irreducible components of the moduli stack $\mathcal{SH}_g$ of hyperelliptic Spin curves with those of $\RH_g$ and $\cH_g$.
	
	Before starting the proof, let us recall a useful fact. Given a (representable) finite, étale morphism $\phi:\cX\rightarrow\cY$ of algebraic stacks and a universal homeomorphism $\psi:\cY'\rightarrow\cY$, there exists a section $s:\cY\rightarrow\cX$ to $\phi$ if and only if there exists $\rho:\cY'\rightarrow\cX$ such that $\phi\circ\rho=\psi$. Notice that we are not asking $\psi$ to be representable. The set of such $\rho$ is in bijection with the set of sections $s'$ to the finite, étale cover $\phi'$ in the following cartesian diagram
	\[
	\begin{tikzcd}
		\cX'\arrow[r,"\phi'"]\arrow[d,"\psi'"] & \cY'\arrow[d,"\psi"]\\
		\cX\arrow[r,"\phi"] & \cY.
	\end{tikzcd}
	\]
	As $s'$ is a section to a finite, étale morphism, it is an open and closed immersion, hence $\cX'=\cY'\sqcup(\cX'\setminus\cY')$, where we identified $\cY'$ with its image under $s'$. As $\psi'$ is an homeomorphism, $\cX\cong \cX_1\sqcup\cX_2$ with $\cX_1\times_{\cY}\cY'\cong\cY'$. It follows that $\cX_1\rightarrow\cY$ is finite, étale and of degree 1, hence an isomorphism, thus yielding an inverse.
	
	\begin{proof}[Proof of Proposition \ref{prop: decomposition Spin}]
		
		Consider the universal commutative diagram over $\cH_g$
		\begin{equation*}
			\begin{tikzcd}
				\mathcal{C} \arrow[rr,"f"] \arrow[dr] & & \mathcal{P}\arrow[dl] \\
				& \cH_g
			\end{tikzcd}
		\end{equation*}
		where $\mathcal{P}\to\cH_g$ is a Brauer–Severi scheme of relative dimension $1$. There exists a $\mu_2$-gerbe $\psi:\widehat{\cH}_g\rightarrow\cH_g$ such that $\widehat{\mathcal{P}}:=\mathcal{P}\times_{\cH_g}\widehat{\cH}_g\rightarrow\widehat{\cH}_g$ is a trivial Brauer-Severi scheme, that is, there exists a rank-2 vector bundle $\widehat{\mathcal{E}}$ over $\widehat{\cH}_g$ such that $\widehat{\mathcal{P}}\cong\P(\widehat{\mathcal{E}})$. Such a gerbe $\widehat{\cH}_g$ may be constructed by taking the image the Brauer class $[\mathcal{P}] \in H^1_{\mathrm{\acute{e}t}}(\cH_g, \PGL_2)$ along the boundary map in cohomology associated with the short exact sequence of sheaves
		\[
		0 \to \mu_2 \to \SL_2 \to \PGL_2 \to 0.
		\]
		Under this map, the Brauer class $[\mathcal{P}]$ is sent to a class $[\widehat{\cH}_g] \in H^2_{\mathrm{\acute{e}t}}(\cH_g, \mu_2)$, which defines the desired $\mu_2$-gerbe.
		In particular, $\widehat{\mathcal{C}}:=\mathcal{C}\times_{\cH_g}\widehat{\cH}_g$ admits the $g_2^1$, namely $\widehat{f}^*\cO_{\P(\mathcal{E})}(1)$, where $\widehat{f}$ is the base change of $f$. By the above discussion, it is enough to construct a morphism $\rho:\widehat{\cH}_g\rightarrow\mathcal{SH}_g$ whose composite with $\mathcal{SH}_g\rightarrow\cH_g$ is $\psi$.
		
		Set $\widehat{M}:=\widehat{f}^*\cO_{\widehat{\mathcal{P}}}(1)^{\otimes(g-1)/2}$. Then, by Riemann-Hurwitz formula, the sheaf $\widehat{M}^{\otimes2}\otimes\omega_{\widehat{\mathcal{C}}/\widehat{\cH}_g}^{-1}$ is trivial on the fibers of $\widehat{\mathcal{C}}\rightarrow\widehat{\cH}_g$. As $\widehat{\cH}_g$ is reduced, it follows that $\widehat{M}$ is Zariski-locally a square root of the dualizing sheaf and this yields the desired morphism $\rho$. The final part of the statement follows from the splitting in Equation~\eqref{eqn: decomposition J[2]}.
	\end{proof}
	\begin{remark}\label{rmk: explicit section}
		We can explicitly describe the section to $\mathcal{SH}_g\rightarrow\cH_g$ as follows. If $C$ is an hyperelliptic curve over an algebraically closed field, and $f:C\rightarrow\P^1$ is a double cover with Weierstrass divisor $W$, then every theta characteristic is of the form
		\[
		f^*\cO_{\P^1}(1)^{\otimes(\frac{g-1}{2}+n)}\otimes\cO_C(-w_1-\ldots-w_{2n})
		\]
		for some $0\leq n\leq(g-1)/2$ and $w_i\in W$ distinct. Then, the statement of Proposition~\ref{prop: decomposition Spin} is essentially equivalent to saying that the number $n$ is a deformation invariant, and the section is given by the `isolated' theta characteristic $f^*\cO_{\P^1}((g-1)/2)$. Given a family $C\xrightarrow{f} P\rightarrow S$ of hyperelliptic curves realized as a double cover of a Brauer-Severi scheme $P\rightarrow S$ of relative dimension 1, to construct the map $S\rightarrow\mathcal{SH}_g$ one can work étale-locally on $S$, to trivialize $P$ and construct the root. The first step is replaced by a global construction in the proof of Proposition~\ref{prop: decomposition Spin}, and it is not necessary if $g$ is congruent to 1 modulo 4, as the candidate section can be taken to be $f^*\omega_{P/S}^{-\frac{g-1}{4}}$ in this case.
	\end{remark}

	\bibliographystyle{amsalpha}
	\bibliography{library}

\providecommand{\bysame}{\leavevmode\hbox to3em{\hrulefill}\thinspace}
\providecommand{\MR}{\relax\ifhmode\unskip\space\fi MR }
\providecommand{\MRhref}[2]{%
  \href{http://www.ams.org/mathscinet-getitem?mr=#1}{#2}
}
\providecommand{\href}[2]{#2}
\begin{thebibliography}{Mum83}

\bibitem[AI17]{AI17}
Shamil Asgarli and Giovanni Inchiostro, \emph{The {P}icard group of the moduli
  of smooth complete intersections of two quadrics}, Transactions of the
  American Mathematical Society (2017).

\bibitem[AOA23]{AreObAbr}
Veronica Arena, Stephen Obinna, and Dan Abramovich, \emph{The integral {C}how
  ring of weighted blow-ups}, arXiv preprint arXiv:2307.01459 (2023).

\bibitem[AV04]{AV04}
Alessandro Arsie and Angelo Vistoli, \emph{Stacks of cyclic covers of
  projective spaces}, Compos. Math. \textbf{140} (2004), no.~3, 647--666.
  \MR{2041774}

\bibitem[Bea77a]{Bea}
Arnaud Beauville, \emph{Prym varieties and the {S}chottky problem}, Inventiones
  mathematicae \textbf{41} (1977), no.~2, 149--196.

\bibitem[Bea77b]{Beabis}
\bysame, \emph{Vari\'et\'es de {Prym} et jacobiennes interm\'ediaires}, Annales
  scientifiques de l'\'Ecole Normale Sup\'erieure \textbf{4e s{\'e}rie, 10}
  (1977), no.~3, 309--391 (fr).

\bibitem[Bis24]{Bis24}
Martin Bishop, \emph{The integral {C}how ring of $\mathcal{M}_{1,n}$ for
  $n=3,\dots ,10$}, Forum of Mathematics, Sigma \textbf{12} (2024).

\bibitem[BS23]{Bae-Schmitt}
Younghan Bae and Johannes Schmitt, \emph{Chow rings of stacks of prestable
  curves {II}}, J. Reine Angew. Math. \textbf{800} (2023), 55--106.
  \MR{4609823}

\bibitem[CF23]{Cavalieri}
Renzo Cavalieri and Damiano Fulghesu, \emph{The integral chow ring of odd},
  Compositio Mathematica \textbf{159} (2023), no.~1, 184--206.

\bibitem[CIL24]{CIL24}
Alessio Cela and Aitor Iribar~Lopez, \emph{The integral {C}how ring of
  $\mathcal{R}_2$}, arXiv preprint arXiv:2406.07309 (2024).

\bibitem[CL22]{CL-hurwitz2}
Samir Canning and Hannah Larson, \emph{Chow rings of low-degree hurwitz
  spaces}, Journal f{\"u}r die reine und angewandte Mathematik (Crelles
  Journal) \textbf{2022} (2022), no.~789, 103--152.

\bibitem[CL23a]{CL23}
\bysame, \emph{The {C}how rings of the moduli spaces of curves of genus 7, 8,
  and 9}, Journal of Algebraic Geometry \textbf{33} (2023).

\bibitem[CL23b]{CL-hurwitz}
\bysame, \emph{The integral picard groups of low-degree hurwitz spaces},
  Mathematische Zeitschrift \textbf{303} (2023), no.~3, 61.

\bibitem[CL24]{CL24}
\bysame, \emph{On the chow and cohomology rings of moduli spaces of stable
  curves}, Journal of the European Mathematical Society (2024).

\bibitem[CL25]{CLI}
Alessio Cela and Alberto Landi, \emph{The integral chow rings of the moduli
  stacks of hyperelliptic prym pairs i}, arXiv preprint arXiv:2501.16320
  (2025).

\bibitem[Cor89]{Cor89}
Maurizio Cornalba, \emph{Moduli of curves and theta-characteristics}, Lectures
  on Riemann Surfaces, 1989.

\bibitem[DL18]{DL18}
Andrea {D}i {L}orenzo, \emph{The {C}how ring of the stack of hyperelliptic
  curves of odd genus}, International Mathematics Research Notices (2018).

\bibitem[DS81]{DonagiI}
Ron Donagi and Roy~Campbell Smith, \emph{{The structure of the {P}rym map}},
  Acta Mathematica \textbf{146} (1981), no.~none, 25 -- 102.

\bibitem[EF09]{EF09}
Dan Edidin and Damiano Fulghesu, \emph{The integral {C}how ring of the stack of
  hyperelliptic curves of even genus}, Math. Res. Lett. \textbf{16} (2009),
  no.~1, 27--40. \MR{2480558}

\bibitem[EG98]{EG98}
Dan Edidin and William Graham, \emph{Equivariant intersection theory}, Invent.
  Math. \textbf{131} (1998), no.~3, 595--634. \MR{1614555}

\bibitem[EH16]{EH16}
D.~Eisenbud and J.~Harris, \emph{3264 and all that: A second course in
  algebraic geometry}, Cambridge University Press, 2016.

\bibitem[EH22]{EH22}
Dan Edidin and Zhengning Hu, \emph{The integral {C}how rings of the stacks of
  hyperelliptic {W}eierstrass points}, arXiv preprint arXiv:2208.00556 (2022).

\bibitem[Fab90a]{Fab90I}
Carel Faber, \emph{Chow rings of moduli spaces of curves {I}: The {C}how ring
  of $\overline{\mathcal{m}}_3$}, Annals of Mathematics \textbf{132} (1990),
  no.~2, 331--419.

\bibitem[Fab90b]{Faber}
\bysame, \emph{{C}how rings of moduli spaces of curves {II}: Some results on
  the {C}how ring of $\mathcal{M}_4$}, Annals of Mathematics \textbf{132}
  (1990), no.~2, 421--449.

\bibitem[FR70]{Farkas}
Hershel~M. Farkas and Harry~E. Rauch, \emph{Period relations of {S}chottky type
  on {R}iemann surfaces}, Annals of Mathematics \textbf{92} (1970), no.~3,
  434--461.

\bibitem[Ful98]{Ful98}
William Fulton, \emph{Intersection theory}, 2nd ed., Springer, 1998.

\bibitem[FV11]{FV11}
Damiano Fulghesu and Filippo Viviani, \emph{The {C}how ring of the stack of
  cyclic covers of the projective line}, Ann. Inst. Fourier (Grenoble)
  \textbf{61} (2011), no.~6, 2249--2275. \MR{2976310}

\bibitem[FV18]{FuVi}
Damiano Fulghesu and Angelo Vistoli, \emph{The {C}how ring of the stack of
  smooth plane cubics}, Michigan Mathematical Journal \textbf{67} (2018),
  no.~1, 3--29.

\bibitem[GV08]{GV08}
Sergey Gorchinskiy and Filippo Viviani, \emph{Picard group of moduli of
  hyperelliptic curves}, Math. Z. \textbf{258} (2008), no.~2, 319--331.
  \MR{2357639}

\bibitem[Inc22]{Inc22}
Giovanni Inchiostro, \emph{Moduli of genus one curves with two marked points as
  a weighted blow-up}, Mathematische Zeitschrift \textbf{302} (2022).

\bibitem[Iza95]{Izadi}
E.~Izadi, \emph{The {C}how ring of the moduli space of curves of genus 5}, The
  Moduli Space of Curves (Boston, MA) (Robbert~H. Dijkgraaf, Carel~F. Faber,
  and Gerard B.~M. van~der Geer, eds.), Birkh{\"a}user Boston, 1995,
  pp.~267--303.

\bibitem[Kee92]{Kee92}
Sean Keel, \emph{Intersection theory of moduli space of stable {$n$}-pointed
  curves of genus zero}, Trans. Amer. Math. Soc. \textbf{330} (1992), no.~2,
  545--574.

\bibitem[Lan23]{Lan23}
Alberto Landi, \emph{The {P}icard group of the stack of pointed smooth cyclic
  covers of the projective line}, arXiv preprint arXiv:2310.20045 (2023).

\bibitem[Lan24]{Lan24}
\bysame, \emph{The integral {C}how ring of the stack of pointed hyperelliptic
  curves}, arXiv preprint arXiv:2404.15873 (2024).

\bibitem[Lar21]{Lar19}
Eric Larson, \emph{The integral {C}how ring of {$\overline M_2$}}, Algebr.
  Geom. \textbf{8} (2021), no.~3, 286--318. \MR{4206438}

\bibitem[Mum74]{Mumford}
David Mumford, \emph{Prym varieties {I}.}, Contributions to Analysis A
  Collection of Papers Dedicated to Lipman Bers, 1974.

\bibitem[Mum83]{Mumford1983}
\bysame, \emph{Towards an enumerative geometry of the moduli space of curves},
  pp.~271--328, Birkh{\"a}user Boston, Boston, MA, 1983.

\bibitem[Oes18]{Oe18}
Jakob Oesinghaus, \emph{Quasisymmetric functions and the {C}how ring of the
  stack of expanded pairs}, Research in the Mathematical Sciences \textbf{6}
  (2018), no.~1, 5.

\bibitem[Pan96]{pandharipande1996chow}
Rahul Pandharipande, \emph{The {C}how ring of the {H}ilbert scheme of rational
  normal curves}, arXiv preprint alg-geom/9607025 (1996).

\bibitem[Per22]{Per22}
Michele Pernice, \emph{The integral {C}how ring of the stack of 1-pointed
  hyperelliptic curves}, Int. Math. Res. Not. IMRN (2022), no.~15,
  11539--11574. \MR{4458558}

\bibitem[Per23]{Per23}
\bysame, \emph{The (almost) integral {C}how ring of
  {$\overline{\mathcal{M}}_3$}}, 03 2023.

\bibitem[PV15]{PenevVakil}
Nikola Penev and Ravi Vakil, \emph{The {C}how ring of the moduli space of
  curves of genus six}, Algebr. Geom \textbf{2} (2015), no.~1, 123--136.

\bibitem[Rom22]{Rom22}
Matthieu Romagny, \emph{Algebraicity and smoothness of fixed point stacks},
  2022.

\bibitem[Tot14]{Bur14}
Burt Totaro, \emph{Group cohomology and algebraic cycles}, Cambridge Tracts in
  Mathematics, Cambridge University Press, 2014.

\bibitem[Ver13]{Ver13}
Alessandro Verra, \emph{Rational parametrizations of moduli spaces of curves},
  Handbook of moduli. {V}ol. {III}, Adv. Lect. Math. (ALM), vol.~26, Int.
  Press, Somerville, MA, 2013, pp.~431--506. \MR{3135442}

\bibitem[Vez98]{Vez98}
Gabriele Vezzosi, \emph{The {C}how ring of $\mathrm{PGL}_2$ is generated by
  {C}hern classes of the adjoint representation}, 1998.

\bibitem[Vis98]{Vis98}
Angelo Vistoli, \emph{The {C}how ring of {$\mathcal{M}_2$}. {A}ppendix to
  ``{E}quivariant intersection theory'' [{I}nvent. {M}ath. {\bf 131} (1998),
  no. 3, 595--634; {MR}1614555 (99j:14003a)] by {D}. {E}didin and {W}.
  {G}raham}, Invent. Math. \textbf{131} (1998), no.~3, 635--644. \MR{1614559}

\end{thebibliography}

	$\,$\
	\noindent
	
	$\,$\
	\noindent
	\textsc{Department of Pure Mathematics {\it \&} Mathematical Statistics, 
		University of Cambridge, Cambridge, UK}
	
	\textit{e-mail address:} \href{mailto:au270@cam.ac.uk}{ac2758@cam.ac.uk}
	
	$\,$\
	\noindent
	
	$\,$\
	\noindent
	\textsc{Department of Pure Mathematics, Brown University, 151 Thayer Street, Providence, RI 02912, USA}
	
	\textit{e-mail address:} \href{mailto:alberto_landi@brown.edu}{alberto\_landi@brown.edu}
	
\end{document}